\documentclass[11pt]{article}

\usepackage{amsmath}
\usepackage{amssymb}
\usepackage{subcaption}
\usepackage{relsize}
\usepackage{graphicx}
\usepackage{graphics}
\usepackage{amsthm}
\usepackage{comment}
\usepackage{epsfig}
\usepackage{epsf}
\usepackage{slashbox}
\usepackage{xltabular}
\usepackage{fullpage}
\usepackage{xcolor}

\def\pist#1#2{\noindent\hangindent 2em\hangafter1\hbox to 2em{#1\hfil~~}#2}

\newtheorem*{remark}{Remark}
\newtheorem{theorem}{Theorem}[section]
\newtheorem{lemma}[theorem]{Lemma}

\newcommand{\be}{\begin{equation}}
\newcommand{\ee}{\end{equation}}
\newcommand{\pone}{p_{1}(v_c)}

\newcommand{\ptwo}{p_{2}(v_c)}

\newcommand{\qone}{q_{1}(v_c)}

\newcommand{\qtwo}{q_{2}(v_c)}

\newcommand{\bsp}{\hspace{-.05in}}

\newcommand{\DETAILS}[1]{}
\newcommand{\I}{\mathrm{i}}

\newcommand{\e}{\eqref}

\newcommand{\revc}[1]{{\color{red}\sout{}}}  
\newcommand{\revmike}[1]{{\color{cyan}}} 

\newcommand{\cblue}[1]{#1} 

\providecommand{\keywords}[1]
{
  \small	
  \textbf{\textit{Keywords---}} #1
}

\title{Exact periodic solutions of the generalized Constantin-Lax-Majda equation with  dissipation}

\author{Denis A. Silantyev$^1$, Pavel M. Lushnikov$^2$, Michael Siegel$^3$, and David M. Ambrose$^4$}

\begin{document}

\maketitle

\begin{center}
{$^1$Department of Mathematics, University of Colorado, Colorado Springs, CO
80918, USA, dsilanty@uccs.edu
\\
$^2$ Department of Mathematics and Statistics, University of New Mexico, MSC01
1115, Albuquerque, NM 87131, USA, plushnik@unm.edu\\
$^3$Department of Mathematical Sciences and Center for Applied Mathematics and
Statistics, New Jersey Institute of Technology, Newark, NJ 07102, USA, michael.s.siegel@njit.edu\\
$^4$Department of Mathematics, Drexel University, Philadelphia, PA 19104, USA, dma68@drexel.edu}
\end{center}

\begin{abstract}  We present exact pole dynamics solutions to the generalized Constantin-Lax-Majda (gCLM) equation in a periodic geometry  with dissipation $-\Lambda^\sigma$, where its spatial Fourier transform is $\widehat{\Lambda^\sigma}=|k|^\sigma$.  The gCLM equation is a simplified model for singularity formation in the 3D incompressible Euler equations. It includes an advection term with parameter $a$, which allows different relative weights for advection and vortex stretching.  There has been intense interest in the gCLM equation, and it has served as a proving ground for the development of methods to study singularity formation in the 3D Euler equations.  Several exact solutions for the problem on the real line have been previously found by the method of pole dynamics, but only one such solution has been reported for the periodic geometry.  We derive new periodic solutions for $a=0$ and $1/2$ and $\sigma=0$ and $1$, for which a closed collection of  (periodically repeated)  poles evolve in the complex plane.  Self-similar finite-time blow-up of the solutions is analyzed and compared for the different values of $\sigma$, and to a global-in-time well-posedness theory for solutions with small data presented in a previous paper of the authors.  Motivated by the exact solutions, the well-posedness theory is extended to include the
case $a=0$, $\sigma \geq 0$. Several interesting features of the solutions are discussed.

\end{abstract}

{\keywords{fluid dynamics, self-similar finite-time singularity formation, complex singularities, global existence, pole solutions, pole dynamics}}

\section{Introduction}

The question of whether solutions to the 3D incompressible Euler and Navier-Stokes equations develop a finite-time singularity from smooth initial data
is one of the most important problems in mathematical fluid dynamics. There has been exciting recent progress on this question for the 3D  Euler equations, most notably by Elgindi \cite{Elgindi2021}, who shows finite-time singularity formation for  $C^{1, \alpha}$ initial velocity data,  and Chen and Hou \cite{ChenHou2022}, who give a computer-aided proof of  finite-time blowup
for  $C^\infty$ initial data.
These results have directly stemmed from
ideas and techniques developed in the analysis of simplified 1D models.
One of the earliest and most important of the 1D models is the Constantin-Lax-Majda (CLM) equation.

The CLM equation was proposed in \cite{ConstantinLaxMajda} as a model for vortex stretching and singularity formation in the 3D Euler equations.  It was later generalized by DeGregorio \cite{DeGregorio}, Okamoto et al. \cite{Okamoto2008}, and Schochet \cite{Schochet} to include advection and dissipation.
In this paper we consider
the generalized Constantin-Lax-Majda (gCLM) equation with dissipation, which  is given by
\begin{equation}
\begin{aligned}\label{CLM}
&\tilde{\omega}_{t}  =-au \tilde{\omega}_{x}+\tilde{\omega} {\cal{H}} \tilde{\omega}
-\nu  {\Lambda}^{\sigma} \tilde{\omega},
~ ~~{\tilde{\omega}}  \in \mathbb{R}, x \in \mathbb{S}, t>0, \\
&u_x =\cal{H} \tilde{\omega}, \\
&\tilde{\omega}(x,0)=\tilde{\omega}_{0}(x),
\end{aligned}
\end{equation}
where the first, second, and third terms on the right-hand side of the first line represent advection, vortex stretching, and dissipation, respectively.
This equation is often considered on the real line but here we take the domain to be the circle $\mathbb{S}$,
 i.e., the interval $x \in [-\pi, \pi]$, extended periodically. The `vorticity'  is decomposed as
 \be
  \tilde{\omega}(x,t)  =\omega_{av}(t) + \omega(x,t), \label{eq:omega_decomp}
  \ee
   where $\omega_{av}(t)$ denotes the
average of $\omega$ over $[-\pi,\pi]$.
The operator $\cal{H}$ is the Hilbert transform, which in the periodic case takes the form
\be  \label{Hilbert_periodic}
{\cal H} f (x)= \frac{1}{2 \pi} PV \int_{-\pi}^{\pi}  f(x') \cot \left(\frac{x-x'}{2} \right) \  dx'.
\ee
The Hilbert transform has Fourier symbol
\begin{equation}\nonumber
\hat{\cal{H}}=-\I \; \mathrm{sgn}(k),
\end{equation}
which also  implies that ${\cal{H}}$ has the representation
\be \label{Hilbert_rep}
{\cal{H}} f = -\I (f_+-f_-),
\ee
where $f_+=\sum_{k>0} \widehat{f}_k e^{\I k x}$ and $f_-=\sum_{k<0} \widehat{f}_k e^{\I k x}$ are projections onto the upper and lower complex analytic components of $f$, respectively, in the complex analytic extension of the real line of $x$ into the complex plane.   The term $- \Lambda^\sigma \tilde{\omega}$ represents a generalized dissipation, in which  the operator $\Lambda$ is given by ${\cal H} \partial_{x}$
so that the symbols of $\Lambda$ and $\Lambda^{\sigma}$ are
\begin{equation}\nonumber
\hat{\Lambda}(k)=|k|,\qquad \widehat{\Lambda^{\sigma}}(k)=|k|^{\sigma}.
\end{equation}
Note that $-\Lambda^2$ gives the usual diffusion operator $\partial_{xx}$.
The equation $u_{x}={\cal H}{\tilde{\omega}}$ defines $u$ up to its mean, and unless stated otherwise we take the mean of $u$ to equal zero.
The parameters $a$, $\sigma$ and $\nu$ satisfy $a\in\mathbb{R}$, $\sigma \geq 0$ and $\nu \geq 0$.

Note that for any periodic function $f,$ we have $\int_{\mathbb{S}} f {\cal H}(f)\ dx=0$ and additionally $u  \tilde{\omega}_x = (u \tilde{\omega})_x - \tilde{\omega} \mathcal{H} \tilde{\omega}$ has zero mean. Thus  when $\sigma>0$ the mean of $\tilde{\omega}$ is
preserved under the evolution (\ref{CLM}), but for $\sigma=0$ the mean generally evolves with time.
Since $({\omega_{av}})_{x}=\mathcal{H}({\omega}_{av})=0$
we can use \e{eq:omega_decomp} to  rewrite (\ref{CLM}) as
\begin{equation}\nonumber
(\omega+\omega_{av})_{t}+au\omega_{x}=\omega \mathcal{H}\omega+{\omega}_{av} \mathcal{H}\omega- \nu \Lambda^{\sigma} (\omega+\omega_{av})  \mbox{~and~} u_x=\mathcal{H} \omega,
\end{equation}
with initial data $\omega(x,0)=\omega_0(x)$ and $\omega_{av}(0)$,
which are sometimes used instead of (\ref{CLM}).

The gCLM model  (\ref{CLM}) arises in several contexts beyond singularity formation in the 3D Euler equations. It is an exact 1D model for the evolution of stress in an Oldroyd B fluid \cite{ConstantinSun}, and is equivalent to a version of the surface quasi-geostrophic equation from geophysics \cite{CastroCordoba}.
The advection term with parameter $a$ was introduced by Okamoto et al.  \bsp  \cite{Okamoto2008} to study different relative weights of advection and vortex stretching, $\tilde{\omega} \mathcal{H} \tilde{\omega}$.  This is motivated by studies of singularity formation in fluid systems which show that advection can have an unexpected smoothing effect  \cite{hou2014finite, Iyer2021, Lei2019, OkamotoOkhitani}.
 In this paper, we consider specific values of the dissipation exponent,  $\sigma=0$ and $1$.  The `marginal' dissipation $\sigma=0$  has a physical significance; it arises in the aforementioned 1D model for the evolution of stress in an Oldroyd B fluid \cite{ConstantinSun}.  We generally allow nonzero values of the average $\omega_{av}(t)$. This is also relevant to the Oldroyd B model, which has an extra constant `forcing' term on the right-hand side of (\ref{CLM}) that generates a nonzero average in $\tilde{\omega}$.

There has been strong interest in elucidating solutions to (\ref{CLM}) which exhibit finite-time singularity formation. Such singularities are often found  to
be locally self-similar  with the form (see, e.g.,  \cite{Chen2020Singularity, Lushnikov_Silantyev_Siegel,AmbroseLushnikovSiegelSilantyev})%
\cblue{\begin{equation}\label{self-similar1}
{\omega}=\frac{1}{\tau^\beta}f\left ( \xi \right ), \ \xi=\frac{x-x_c(t)}{\tau^\alpha}, \ \tau=t_c-t,
\end{equation}
in a space-time neighborhood of $(x_c(t_c),t_c)$, where $t_c>0$\ is the singularity time, $x_c(t) \in \mathbb{R}$ is the time-dependent  spatial center location of the collapsing solution (which reaches $x=x_c(t_c)$ at $t=t_c$), and $\alpha$, $\beta$ are real similarity parameters. Typically $x_c\equiv const$, which is mostly the case in \cite{Chen2020Singularity, Lushnikov_Silantyev_Siegel,AmbroseLushnikovSiegelSilantyev} and the current work.
 }
\cblue{However, self-similar solutions of the type \e{self-similar1}  with nontrivial
$x_c(t)$ can also occur, and  
are known in other systems, e.g.,  collapse of a spherical shell  in the three-dimensional cubic nonlinear Schr\"odinger equation (NLSE) \cite{DegtyarevZakharovRudakovJETP1975} and the collapse of a soliton-like self-similar solution in the critical
Korteweg-de Vries (KdV) equation \cite{PelinovskyGrimshawPhysicaD1996}, see also more recent results and review in \cite{ChapmanKavousanakisCharalampidisKevrekidisKevrekidisNonlinearity2024}. In the KdV equation case, the center of mass of the collapsing self-similar solution moves with an increasing speed, reaching infinite speed at the collapse time.  In the  NLSE case, the spatial width of the spherical shell shrinks to zero while the amplitude of the solution diverges inside the shell; the  shell approaches the origin hitting it exactly at the collapse time.  Numerical evidence for the latter type of self-similar collapse has recently  been found in the axisymmetric Euler equations with swirl \cite{HouHuang2022,HouHuang2023,Hou2023Potential}.
In the CLM equation with $a=\nu=0$, Huang et al.\ \cite{HuangTwoScale2024}
showed self-similar finite-time singularity formation of the form (\ref{self-similar1}) with $x_c(t)={x}_c(t_c) +r_c (t_c-t)^\delta$ for constant $r_c \in \mathbb{R}$ and $0<\delta<\alpha$.  Such a `two-scale' self-similar solution is also given by equation (72) in  \cite{AmbroseLushnikovSiegelSilantyev}, for which $x_c(t)$ is the linear function of time corresponding to the motion of the spatial center location of the collapsing solution with constant velocity.   We confirm here
 that the two-scale blowup seen in   \cite{HuangTwoScale2024} also occurs in the periodic gCLM problem with dissipation for various $\alpha, \beta$ and $x_c(t)$.
}


When $a=0$, Constantin, Lax, and Majda \cite{ConstantinLaxMajda} obtained a closed form solution to (\ref{CLM}) for the inviscid problem (with $\nu=0$) which is valid in both the real line and periodic geometry. Their solution exhibits finite time singularity formation of the form (\ref{self-similar1}) for a large class of initial data.
Since the original work of \cite{ConstantinLaxMajda}, several exact solutions for the real line problem have been found by the method of pole dynamics \cblue{\cite{Chen2020Singularity, Lushnikov_Silantyev_Siegel,AmbroseLushnikovSiegelSilantyev}}.  There are numerous examples of pole dynamics solutions in both Hamiltonian and dissipative systems, see, e.g.,  \cite{BakerSiegelTanveer,CalogeroBook,LushnikovPhysLettA2004,LushnikovZubarevPRL2018,Senouf1996}.

The exact solutions of the gCLM equation have played an important role in more general analysis of finite-time singularity formation.
Elgindi and Jeong \cite{Elgindi2020} and Chen et al. \cite{Chen2021}   prove that for $\epsilon-$small values of $a>0$ and $\nu=0$,  there  exists
singularities of the form
(\ref{self-similar1}) with $\beta=1$ and $\alpha$ approaching $1$ in the limit $a\to 0^+$. Their method involves showing the nonlinear stability of an approximate self-similar profile when $a>0$, with the approximate profile provided by the exact $a=0$ solution of \cite{ConstantinLaxMajda}. Similarly,  Chen \cite{Chen2020Singularity} shows that for the problem on the real line, there exists self-similar blowup when $a$ is close to $1/2$ and $\sigma=2$.    His method is also based on establishing nonlinear stability of an approximate self-similar solution. The approximate profile is provided by an exact pole dynamics solution of the inviscid problem for $a=1/2$ \cite{AmbroseLushnikovSiegelSilantyev,Chen2020Singularity, Lushnikov_Silantyev_Siegel}.  These results highlight the importance of  exact solutions in analysis of finite-time singularity formation for  the gCLM model.

There is a rather large literature on finite time singularity formation in the inviscid problem for (1) in which $\nu=0$.
Recently, extensive numerical and analytical results on singularity formation covering a wide range of the advection parameter $a$ have been obtained by \cite{Lushnikov_Silantyev_Siegel} and \cite{Huang2024}.
We refer the reader to those papers for  reviews of the latest developments in the inviscid problem.

The current paper focuses on the problem with nonzero dissipation ($\nu \ne 0$) for which
less is known about solutions to (\ref{CLM}).  Schochet \cite{Schochet} constructs an explicit pole dynamics solution on the real line for $a=0$ and $\sigma=2$   which blows up in finite time. When $a=-1$, for which  (\ref{CLM}) is equivalent to the   Cordoba-Cordoba-Fontelos equation \cite{CCF}, finite time blowup is shown to occur for $\sigma<1/2$  \cite{Kiselev, LiRodrigo, SilvestreVicol},
whereas global well-posedness  is known for    $\sigma \geq 1$ \cite{CCF, Dong2008, Kiselev}.
Chen \cite{Chen2020Singularity} shows that for the problem on the real line, there exists a self-similar blowup when $a$ is close to $1/2$ and $\sigma=2$, and global well-posedness for  $\sigma \in [|a|^{-1}, 2]$ with $a<-1$. Ambrose et al. \cite{AmbroseLushnikovSiegelSilantyev} consider (\ref{CLM}) in the periodic setting for initial data in $L^2$ or in a Wiener algebra $B_0$, which describes the set of functions with Fourier coefficients in $l^1$.    They prove global-in-time existence of solutions with small-$\omega$ data and all $a$  when (i) $\sigma \geq 1$ and the initial data is in $B_0$, and (ii) $\sigma>1$ and the data is in $L^2$.   Solutions become analytic (or $C^\infty$ in the case of $L^2$ data) for $t>0$.
We emphasize that only data for the spatially varying component $\omega$ needs to  be small for this result to apply; $\omega_{av}$ can have any magnitude.
Ambrose et al. \cite{AmbroseLushnikovSiegelSilantyev} also derive a new pole dynamics solution to the periodic problem for $a=0$ and  $\sigma=0$
which consists of a periodic array of complex conjugate (c.c.) pole singularities.  This solution can form a finite-time singularity of the form (\ref{self-similar1}), even for data which is arbitrarily small in $L^2$, and provides a contrast to the global existence theory for $\sigma \geq 1$.

Other than the exact solution for $a=0$  in the original paper of Constantin, Lax, and Majda and the aforementioned construction of \cite{AmbroseLushnikovSiegelSilantyev}, we are unaware of any explicit solutions to (\ref{CLM}) in a periodic setting. The focus of this paper is to derive new periodic pole dynamics solutions which provide further insight on the conditions for which (\ref{CLM}) is well-posed globally in time. The new results include:  (i) solutions for $a=0$ and $\sigma=0$ and $1$ expressed as a periodic array of c.c.  poles in $\omega$, and  (ii) solutions for $a=1/2$
and $\sigma=0$ and $1$ expressed as a periodic array of c.c.  simple and double poles.

A challenge in constructing pole dynamics solutions is obtaining a closed collection of complex singularities that exactly represent the dynamics.  In particular,  higher order poles that are generated by derivatives and nonlinearities
must be avoided.  We do this by making the special choice  $a=0$ or  $1/2$ and $\sigma=0$ or $1$,  for which higher order poles cancel out.
When $a=1/2$,   the advection term generates log singularities out of poles, which cause additional difficulties. We avoid these by utilizing a special combination of single and double poles in $\omega$  for which the log singularities cancel out.

Our pole dynamics solutions are specified in terms of nonlinear systems of ODEs that give the pole locations and amplitudes.  In most cases we explicitly solve the ODEs and characterize the similarity exponents.  For one case,  $a=1/2$ and $\sigma=0$, the exact solution is derived in implicit form, and  substantial additional analysis is performed to obtain information on singularity formation.  When an explicit solution is not available,   phase-plane analysis is employed to characterize the dynamics. A difficulty in this analysis is the presence of `pathological' trajectories, which can exit regions which are otherwise invariant under the dynamics. The phase-plane analysis is facilitated by a transformation to special variables which avoids the problematic trajectories and clarifies the dynamics.


We  find that global existence versus finite-time singularity formation for our pole dynamics solutions depends on the initial location and amplitude of poles in $\mathbb{C}$. When $\sigma=1$ the solutions determined here   exist globally in time for all sufficiently small data in $L^2$ and in $B_0$,
but  blow up in finite time for large enough data.  This result is consistent with the theory of \cite{AmbroseLushnikovSiegelSilantyev} and further proves that blowup can and does occur for sufficiently large data. It also illustrates a difference with the problem on $\mathbb{R}$,  in which exact pole dynamics solutions constructed in \cite{AmbroseLushnikovSiegelSilantyev}  can blow up for arbitrarily small data  in $L^2$ and in $B_0$.

In contrast,
we  find finite-time blowup in the periodic problem for arbitrarily small $L^2$  initial data when $\sigma=0$. We still find global existence for sufficiently small data in the  Wiener norm for $\sigma=0$. This motivates revisiting the $B_0$ global existence theory from \cite{AmbroseLushnikovSiegelSilantyev}, which there applies for $\sigma \geq 1$.  We show that $a=0$  is a special case in which the  $B_0$ theory for the periodic problem can be extended to $\sigma \geq 0$ 
(our proof also applies to the real-line problem for $a=0,~\sigma=0$).  Overall, our analysis provides insight into the effect of dissipation on finite-time singularity formation in (\ref{CLM}).

The rest of this paper is organized as follows: Section \ref{sec:a=0} develops  pole dynamics solutions and more general analysis for $a=0$. This includes closed-form exact  solutions for $a=0,~\sigma=0$  in Section \ref{sec:a0sigma0},
and extension to $a=0$ of  the small-data  well-posedness theory  from  \cite{AmbroseLushnikovSiegelSilantyev}   (Sections \ref{sec:a=0_well-posedness} and \ref{sec:a=0_well-posedness_1}).
Pole dynamics solutions for $a=0,~\sigma=1$ are derived and analyzed in Section \ref{sec:a0sigma1}.
Section \ref{sec:a=one_half} considers  pole dynamics solutions for $a=1/2$. In Section \ref{sec:a0.5sigma0},  complete analysis of  an exact implicit solution for   $a=1/2,~ \sigma=0$ is provided. provided. Section \ref{sec:a0.5sigma1} presents the governing ODEs for  the pole dynamics solution in the case $a=1/2,~ \sigma=1$, which are studied by a phase-plane analysis in Section \ref{sec:phase-plane}.   Concluding remarks are given in  Section \ref{sec:Conclusions}.  Appendix \ref{sec:AppendixA} presents a phase-plane analysis for $a=1/2,~\sigma=0$, as a complement to implicit solution in Section \ref{sec:a0.5sigma0}.

\section{Pole dynamics solutions and analysis for $a=0$} \label{sec:a=0}

When $a=0$ the nonlinear advection term $a u \omega_x$ in  (\ref{CLM}) is absent which greatly simplifies the construction of  pole dynamics  solutions.  
We begin with this simplified case and  derive exact solutions for  a single pair of c.c. poles, for which a rather complete analysis of finite-time singularity formation can be provided. The governing PDE also allow for pole dynamics solutions consisting of $N$ pairs of  complex conjugate poles evolving in $\mathbb C$ for any integer $N \geq 1$, and we provide  evolution equations for the pole locations and amplitudes in this general case.

\cblue{In this section and throughout the remainder of the text} we use a decomposition of $\omega$ as follows
\begin{equation}\label{omegaplusminus}
\omega=\omega_+ + \omega_-,
\end{equation}
 where $\omega_+$ and $\omega_-$
   are projections onto the upper and lower analytic components of $\omega$, respectively, and recall  the solution $\tilde{\omega}$ to (\ref{CLM}) is given by \e{eq:omega_decomp}.  The real-valuedness of $\tilde{\omega}$ together with \e{eq:omega_decomp} implies that
   \begin{equation}\label{complexconjugationomega}
  \omega_+(x,t)=\overline{\omega_-(\overline{x},t)}
 \end{equation}
 and that  $\omega_{av}(t)$ is real, where the overbar denotes complex conjugate. Therefore the solution is completely specified by, say,  $\omega_-$ and $\omega_{av}.$

\cblue{Following  \cite{LushnikovStokesParIIJFM2016} we consider  the transformation
\begin{equation}\label{Xdef}
X=\tan \left( \frac{x}{2} \right)
\end{equation}
which  maps $(-\pi,\pi)$ to the real line in $X$. In the complex plane of $x$, \e{Xdef} is the conformal map of a vertical strip $-\pi<Re[x]<\pi$ onto $\mathbb{C}\setminus \left([\I,\I\infty)\bigcup [-\I,-\I\infty)\right ) $, i.e., the entire complex plane $\mathbb{C}$ of $X$ except two vertical segments  $[\I,\I\infty)$ and $ [-\I,-\I\infty)$. The inverse of \e{Xdef} has branch points at $X=\pm \I$ with corresponding branch cuts which are convenient to put along $[\I,\I\infty)$ and $ [-\I,-\I\infty)$.  Crossing any of these branch cuts corresponds to extending the strip $-\pi<Re[x]<\pi$ to adjacent strips  $-\pi+2\pi n<Re[x]<\pi+2\pi n, \ n\in \mathbb{N}$ along real axis. The boundary $Re[x] =\pm \pi$ of the vertical strips is mapped by \e{Xdef}  onto the branch cuts in $X$.   Since we restrict to periodic in $x$ solutions, there are no jumps at these branch cuts making them `invisible' and thus they are ignored below. It allows us throughout this paper to look for rational solutions in $X$ space which ensures $2\pi$ periodicity of these solution in $x$ space.
}

\cblue{We look for a real-valued solution of \e{CLM} in which $\omega_-$ possesses $N$  simple poles in the upper half-plane of  $X$, corresponding to $N$ periodically repeated poles in the $x$-plane:
\begin{equation}\label{N_poles}
\omega_-(x,t)=\sum_{k=1}^N \omega_{-1,k}(t)\left[\frac{1}{\tan(\frac{x}{2})-\I v_{c,k}(t)}- \frac{1}{-\I-\I v_{c,k}(t)}\right],~~Im[\omega_{av}(t)]=0.
\end{equation}
Here  $\omega_{-1,k}(t)$, and $v_{c,k}(t)$ are smooth complex valued functions of $t$ with $Re[v_{c,k}(t)]  \geq  0$.
The term $\frac{1}{-\I-\I v_{c,k}(t)}$ is subtracted inside the square brackets so
that $\omega_-(x,t)$ has  zero mean on $x \in [-\pi,\pi]$, i.e., $\int_{-\pi}^\pi{\omega_-(x,t)dx}=0$.   Simple poles to $\omega_+$  exist at  complex conjugate  locations in the lower half-plane.
}

\cblue{Collapse can occur when $Re[v_{c,k}(t_c)]=0$ or $Re[v_{c,k}(t_c)]=\infty$ for some $k$ and $t_c$, so that $\tan(\frac{x}{2})-\I v_{c,k}(t_c)=0$ at some real $x=x_{c,k}=2 \arctan(\I v_{c,k}(t_c))$.  In both instances a pole impinges on the real line in complex $x$-space. When  $v_{c,k}$ is real, $x_{c,k}$ is pure imaginary   for $0<v_{c,k}<1$  and  $x_{c,k}= \pm \pi+ \I Im(x_{c,k})$ for $1<v_{c,k}<\infty$, with $Im(x_{c,k}) \rightarrow 0$ as $v_{c,k} \rightarrow 0$ or $\infty$ and  $Im(x_{c,k}) \rightarrow \infty$ as $v_{c,k} \rightarrow 1$.}

\cblue{In the case $N=1$ of a single c.c.\  pair of poles, we denote $v_{c,1}$, $\omega_{-1,1}$  and $x_{c,1}$ simply by $v_{c}$, $\omega_{-1}$ and and $x_{c}$.}
We compute in this case  the norms of $\omega=\omega_+ + \omega_-$ that are relevant in the well-posedness theory of \cite{AmbroseLushnikovSiegelSilantyev}.
These are
\begin{equation}
\begin{split}
 \| \omega(\cdot,t) \|_{L^2}^2 &= 4\pi\frac{|\omega_{-1}(t)|^2}{Re[v_c(t)] |1+v_c(t)|^2}  \label{a0sigma0_periodic_solution_norms}\\
  \| \omega(\cdot,t) \|_{B_0}& =\frac{|\omega_{-1}(t)|}{Re[v_c(t)]}\left(\left|\frac{1 - v_c(t)}{1 + v_c(t)}\right| +1\right).
 \end{split}
 \end{equation}.


Following  \cite{AmbroseLushnikovSiegelSilantyev}, we are interested in whether a solution with arbitrarily small initial data in these norms can blowup  in finite time.   We observe that
the norm of  the initial data can be made arbitrarily small  by choosing $Re[v_c(0)]$ large enough.

\subsection{$a=0$ and $\sigma=0$} \label{sec:a0sigma0}

We take $a=0$ and  $\sigma=0$ in which case  (\ref{CLM}) becomes
\begin{equation} \label{mainEquationa0sigma0_periodic}
\tilde{\omega}_{t}=\tilde{\omega} \mathcal{H}(\tilde{\omega})-\nu \tilde{\omega}.
\end{equation}
Using $\tilde{\omega}=\omega_+ + \omega_-+\omega_{av}$ and the Hilbert transform representation \e{Hilbert_rep} we can rewrite \e{mainEquationa0sigma0_periodic} as:
\begin{equation} \label{mainEquationa0sigma0_periodic_lower}
(\omega_-)_t=\I \omega_-^2 + (\I \omega_{av}-\nu) \omega_-, \quad (\omega_+)_t=-\I \omega_+^2 + (-\I \omega_{av}-\nu) \omega_+, \quad (\omega_{av})_t= -\nu \omega_{av},
\end{equation}

{\underline{\em One pair of c.c. poles (N=1).} \cblue{We consider (\ref{N_poles}) with $N=1$:
\begin{equation}\label{omegasigma0_periodic_lower}
\omega_-(x,t)=\omega_{-1}(t)\left[\frac{1}{\tan(\frac{x}{2})-\I v_c(t)}- \frac{1}{-\I-\I v_c(t)}\right],  ~~\omega_+(x,t)=\overline{\omega_-(\overline{x},t)}
\end{equation}
}
  Substitute the first equation from \e{omegasigma0_periodic_lower} into the first equation in \e{mainEquationa0sigma0_periodic_lower} and equate like order poles to get the following equations for the pole amplitude and location in the upper half-plane:
\begin{equation}\label{a0sigma0_periodic_eqns} 
\frac{d\omega_{-1}(t)}{dt}=(\I \omega_{av}(t)-\nu)\omega_{-1}(t) + \frac{2\omega_{-1}^2(t)}{1+v_c(t)},  \qquad  \frac{dv_c(t)}{dt} =\omega_{-1}(t), \quad \frac{d \omega_{av}(t)}{dt}=-\nu \omega_{av}(t).
\end{equation}
Equations (\ref{a0sigma0_periodic_eqns}) have a solution
\begin{eqnarray}\label{a0sigma0_periodic_solution_c}
\omega_{-1}(t)&=&\frac{\omega_{-1}(0)e^{-\nu t} \zeta(t)}{\left(1-\frac{\omega_{-1}(0)}{\I \omega_{av}(0)(1+v_c(0))}(\zeta(t)-1) \right)^2}, \
v_c(t)=\frac{v_c(0)+1}{\left(1-\frac{\omega_{-1}(0)}{\I \omega_{av}(0)(1+v_c(0))}(\zeta(t)-1) \right)}  -1, \nonumber \\
\omega_{av}(t)&=&\omega_{av}(0)e^{-\nu t}, \quad \text{where} \quad \zeta(t)=e^{\I \frac{\omega_{av}(0)}{\nu}(1-e^{-\nu t}) }.
\end{eqnarray}
This generalizes the solution presented in \cite{AmbroseLushnikovSiegelSilantyev} to include nonzero mean. It can be verified that
in the limit $\omega_{av}(0) \rightarrow 0$ this solution reduces to
\begin{equation}\label{a0sigma0_periodic_solution}
\omega_{-1}(t)=\frac{\omega_{-1}(0)e^{-\nu t}}{\left(1-\frac{\omega_{-1}(0)}{\nu(1+v_c(0))}(1-e^{-\nu t}) \right)^2}, \
v_c(t)=\frac{v_c(0)+1}{\left(1-\frac{\omega_{-1}(0)}{\nu(1+v_c(0))}(1-e^{-\nu t}) \right)}  -1 ,
\end{equation}
which is the same as the zero mean solution (eq. (82)) from \cite{AmbroseLushnikovSiegelSilantyev}.

{\underline{\em Inviscid problem ($\nu=0$).} Equations \e{a0sigma0_periodic_solution} for $\nu=0$ reduce to:
\begin{equation}\label{a0sigma0_periodic_solution_nu0}
\omega_{-1}(t)=\frac{\omega_{-1}(0)}{(1-\frac{\omega_{-1}(0)}{1+v_c(0)}t)^2},  \qquad v_c(t)=\frac{v_c(0)+\frac{\omega_{-1}(0)}{1+v_c(0)}t}{1-\frac{\omega_{-1}(0)}{1+v_c(0)}t}.
\end{equation}
Then we always have a collapsing solution at $t=t_c$ (even for arbitrarily small data) with collapse location $x=x_c=2 \arctan(\I v_c(t_c))$.
In case of collapse with $Re[v_c(t_c)]=0$, the singularity approaches the real line in $x$-space at $x_c$ and reaches it  at the time $t_c$ given by
\begin{equation} \label{a0sigma0_tc_periodic}
x_c=-2 \arctan(Im[v_c(t_c)]),~
t_c=\frac{Y + \sqrt{4|\omega_{-1}(0)|^2 Re[v_c(0)] (| v_c(0) | ^2 + 1 + 2 Re[v_c(0)]) + Y^2}}{2|\omega_{-1}(0)|^2},
\end{equation}
where $$Y=Re[\omega_{-1}(0)](1 + Im[v_c(0)]^2 - Re[v_c(0)]^2) - 2 Im[\omega_{-1}(0)] Im[v_c(0)] Re[v_c(0)].$$
In the collapse case where $Re[v_c(t_c)]=\infty$,  the singularity approaches the real line at $x_c$ and time $t_c$ given by
\begin{equation} \label{a0sigma0_tc_periodic_inf}
x_c=\pm \pi,~~~~~~~~~~t_c=\frac{1+ Re[v_c(0)]}{ Re[\omega_{-1}(0)]},
\end{equation}
when the condition $ Im[\omega_{-1}(0)]Re[1+v_c(0)]=Re[\omega_{-1}(0)]Im[v_c(0)]$ is satisfied.

For purely real $v_c(0)>0$ and $\omega_{-1}(0)$ equations \e{a0sigma0_tc_periodic} and  \e{a0sigma0_tc_periodic_inf} reduce to
\begin{align} \nonumber
t_c&=-\frac{v_c(0)(1+v_c(0))}{\omega_{-1}(0)}, &&x_c=0 (v_c \rightarrow 0)  &\mbox{for}\quad \omega_{-1}(0)<0,\\
t_c&=\frac{1+v_c(0)}{\omega_{-1}(0)}, &&x_c=\pm \pi (v_c \rightarrow \infty)  &\mbox{for} \quad \omega_{-1}(0)>0.\nonumber
\end{align}

The solution \e{omegasigma0_periodic_lower}, \e{a0sigma0_periodic_solution_nu0} is a pole dynamics analog of the original solution in \cite{ConstantinLaxMajda} as well as a spatially periodic version of  equation (32) in \cite{Lushnikov_Silantyev_Siegel}, which describes  self-similar blowup in the inviscid problem on the real line.
Finite-time collapse in \e{omegasigma0_periodic_lower}, \e{a0sigma0_periodic_solution_nu0} is of the general self-similar form (\ref{self-similar1}) with $\alpha=\beta=1.$ However, when $\omega_{av}(0) \ne 0$  the solution   (\ref{a0sigma0_periodic_solution})  can have different similarity exponents and exhibit new behavior, even for $\nu=0$. This includes solutions which collapse then smooth out periodically in time, as described  in   Figure \ref{fig:pole_trajectories} and below.

{\underline{Dissipative problem ($\nu>0$).}  When $\nu>0$ and $\omega_{av}(0) = 0$ finite-time singularity formation in the
 solution \e{omegasigma0_periodic_lower},  \e{a0sigma0_periodic_solution} occurs in two ranges of initial data (outside of these ranges we have global-in-time existence). An   analysis of this  singularity formation  appears   in  Section 5.4 of  \cite{AmbroseLushnikovSiegelSilantyev}, which we summarize in the next paragraph.  For simplicity,  we  assume that  $\omega_{-1}(0)$ and $v_c(0)$ are purely real.



The first range of initial data which leads to finite-time collapse is  $\omega_{-1}(0) <- \nu v_c(0)(1+v_c(0))$, for which   the poles reach the real axis at $v_c(t_c)=0$, $x_c=0$, where
\begin{equation}\nonumber
t_c=\frac{1}{\nu}\ln \left(\frac{\omega_{-1}(0)}{\omega_{-1}(0)+\nu v_c(0)(1+v_c(0))}\right).
\end{equation}
The initial data which leads to this type of blowup satisfies
\begin{align} \label{a0sigma0_periodic_solution_norms_real_c1}
& \| \omega(\cdot,0) \|_{L^2} > 2\nu \sqrt{\pi v_c(0)},
 \nonumber
 \\
 &\| \omega(\cdot,0) \|_{B_0} > 2\nu, \mbox{ if } v_c(0)<1, &&
 \| \omega(\cdot,0) \|_{B_0} > 2\nu v_c(0), \mbox{ if } v_c(0)\geq1. \nonumber
\end{align}
We therefore see that collapse can be made to occur for arbitrarily small  $\| \omega(\cdot,0) \|_{L^2}$ by choosing small enough $v_c(0)$,  but
$\| \omega(\cdot,0) \|_{B_0}$ must be greater than $2 \nu$ for blow up.

The second range of initial data which leads to finite-time collapse is  $\omega_{-1}(0) > \nu (1+v_c(0))$ for which   the poles reach the real axis at $v_c(t_c)=\infty$, $x_c=\pm\pi$ at the time
\begin{equation}\nonumber 
t_c=\frac{1}{\nu}\ln \left(\frac{\omega_{-1}(0)}{\omega_{-1}(0)-\nu(1+v_c(0))}\right).
\end{equation}
The initial data which leads to this type of blowup satisfies
\begin{align}
&\| \omega(\cdot,0) \|_{L^2} > 2\nu \sqrt{ \frac{\pi}{v_c(0)}},  \nonumber
 \\
 &\| \omega(\cdot,0) \|_{B_0} > 2\frac{\nu}{v_c(0)},\mbox{ if } v_c(0)<1, &&\| \omega(\cdot,0) \|_{B_0} > 2\nu ,\mbox{ if } v_c(0)\geq1, \nonumber
\end{align}
We again see that collapse can be made to occur for arbitrarily small  $\| \omega(\cdot,0) \|_{L^2}$ by choosing large enough $v_c(0)$,
 whereas
$\| \omega(\cdot,0) \|_{B_0}$  must be sufficiently large for blow up.

In both cases, the collapse is self-similar and the solution \e{omegasigma0_periodic_lower}
belongs to the general self-similar form (\ref{self-similar1}). As is shown below, different values of the similarity exponents $\alpha$ and $\beta$ are possible, depending on the data.  Generically,  however, $\alpha=\beta=1$, which follows from  $\omega_{-1}(t_c)\neq0$   and $v_c\sim (t_c-t)$ in the  collapse case when $v_c \rightarrow 0$, and  $\omega_{-1} \sim (t_c-t)^{-2}$ and $v_c\sim (t_c-t)^{-1}$ in the   case when when $v_c \rightarrow \infty$.

An interesting new observation is that the pole trajectories of  $N=1$ solutions for $a=0$, $\sigma=0$ lie on circles or lines in $\tan(\frac{x}{2})$-space.
This follows from a nontrivial calculation which shows that $v_c(t)$ from \e{a0sigma0_periodic_solution} can be rewritten as
\begin{align}\label{a0sigma0_periodic_solution_circle}
\begin{split}
v_c(t) &=
\I \frac{\Omega + v_c(0)\bar{\Omega}}{2 Im[\Omega]} +  \left| \frac{(v_c(0) + 1) \Omega}{2 Im[\Omega]} \right| e^{\I(\arg[v_c(0) + 1] + \arg[\Omega] - 2\arg[\Omega \theta(t) - 1] -sgn(Im[\Omega]) \pi/2)},  \\
 \mbox{where} \quad\Omega &=\frac{\omega_{-1}(0)}{1+v_c(0)}, \quad \theta(t)=\frac{1 - e^{-\nu t}}{\nu} (=  t \mbox{~when~} \nu=0).
 \end{split}
\end{align}
Only the phase is time-dependent so that when $Im[\Omega] \ne 0$ (\ref{a0sigma0_periodic_solution_circle}) describes a circle. In the case $Im[\Omega]=0$ a separate calculation starting from \e{a0sigma0_periodic_solution} shows that $v_c(t)$ lies on a line.
 The more general solution \e{a0sigma0_periodic_solution_c} for $\omega_{av}(0)\neq 0, \nu \geq 0 $ also has pole trajectories lying on circles and lines in $\tan(\frac{x}{2})$-space.

Figure \ref{fig:pole_trajectories} shows example  pole trajectories  in $\tan(\frac{x}{2})$-space that  result in collapsing solutions with different similarity exponents. In Figure \ref{fig:pole_trajectories}(a) the pole crosses the real-$\tan(\frac{x}{2})$ line  (i.e.,  $Re[v_c]=0$)  with scaling $Re[v_c] \sim t_c-t$ for $t$ near $t_c$, so that  the similarity exponents are $\alpha=\beta=1$. The solution on $x \in [-\pi,\pi]$ is ill-defined after the pole passes through the real $\tan(\frac{x}{2})$-line, i.e.,   for $Re[v_c]<0$, shown in red.
In Figure \ref{fig:pole_trajectories}(b)  the pole  reaches the real $\tan(\frac{x}{2})$-line
when $t \rightarrow \infty$, corresponding to infinite time collapse.
In Figures \ref{fig:pole_trajectories}(c) and (d) the pole touches the real line with
$\frac{d Re[v_c(t)]}{dt} = 0$.
The solution  in Figure \ref{fig:pole_trajectories}(c) collapses at $t=t_c$ but immediately regularizes and  is smooth again for $t>t_c$.  This `collapse followed by smoothing' occurs multiple times  in  Figure \ref{fig:pole_trajectories}(d).   When $\nu=0$ and $\omega_{av} \ne 0$, one can find a circular trajectory like that in Figure \ref{fig:pole_trajectories}(d) that is periodically traversed in time, i.e.,  a periodically collapsing solution.
\cblue{The solution in Figure \ref{fig:pole_trajectories}(c) and each of the periodically collapsing solutions in    Figure \ref{fig:pole_trajectories}(d) are examples of so-called  two-scale self-similar blowup \cite{HuangTwoScale2024} of the form (\ref{self-similar1}), specifically,
\begin{equation}\label{multiscaleblowup}
\omega \sim \frac{1}{\tau^\beta} f \left( \frac{x -x_c-r_c  \cdot \tau^\delta}{\tau^\alpha} \right),~\mbox{where}~r_c=  \frac{2 Im [\dot{v}_c(t_c)]}{1+Im[v_c(t_c)]^2},
\end{equation}
with $\alpha=\beta=2$, $\delta=1$, and the constant $x_c$ given by (\ref{a0sigma0_tc_periodic}). Strictly speaking, an infinite  family of  two-scale self-similar  collapsing solutions can be  generated  from a single-scale self-similar  collapsing solution by a change of reference frame  $x'=x-a\int^t r(s)ds, ~ t'=t$.  This shifts the generally time-dependent (spatial) average  velocity $u_{av}(t)$ by $r(t)$. Such degeneracy can be avoided by considering an appropriate choice of $u_{av}(t)$.}

Figure \ref{fig:pole_trajectories2} shows an example of a degenerate pole trajectory
that results in a collapsing solution with $Re[v_c(t_c)]=\infty$ and $x_c(t_c) = \pm \pi$. Here  $v_c \sim 1/(t-t_c)$ as $t \rightarrow t_c$, and the similarity exponents are $\alpha=\beta=1$.
 The solution for $Im[x_c]<0$ (shown in red) is ill-defined on the real-$x$ line.
(We note that this phenomenon of singularities which come and go in time is qualitatively similar to the dispersive blowup phenomenon studied in \cite{bonaDispersive1},
\cite{bonaDispersive2}; in these works, it is shown that while dispersive equations propagate Sobolev regularity, solutions can momentarily lose their highest $C^{k}$ regularity.
Thus, for dispersive equations, solutions are able to continue through such singularities and become regular again.)


\begin{figure}[h!]
\centering
\includegraphics[width=0.475\textwidth]{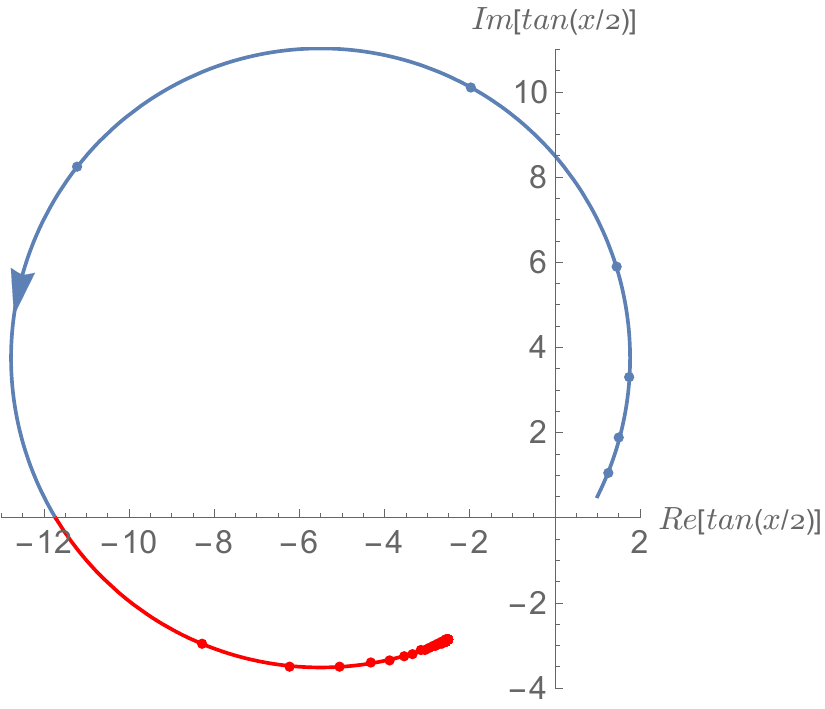}
\put(-230,200){\Large{(a)}}
\put(0,200){\Large{(b)}}
\includegraphics[width=0.475\textwidth]{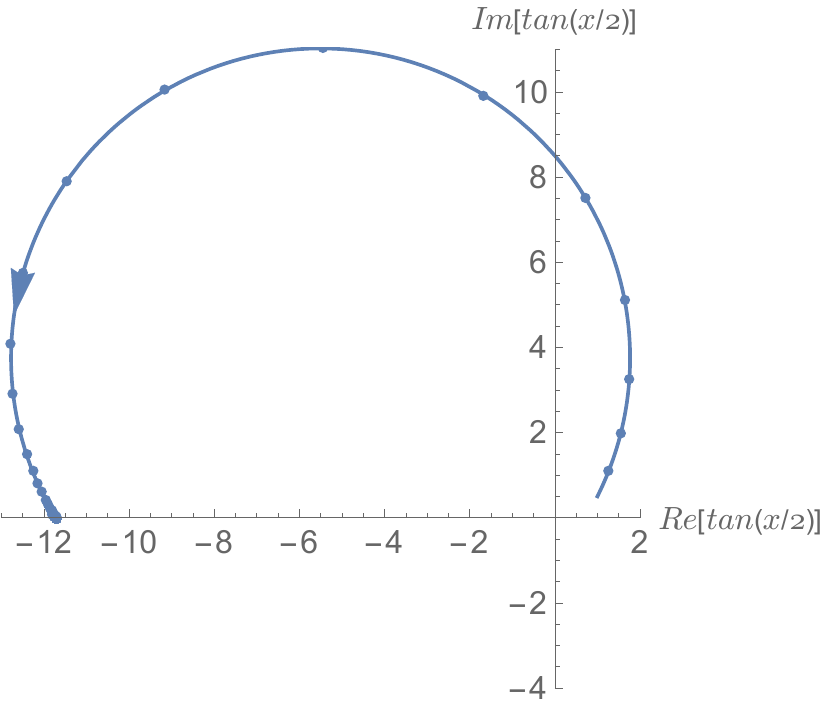}
\includegraphics[width=0.475\textwidth]{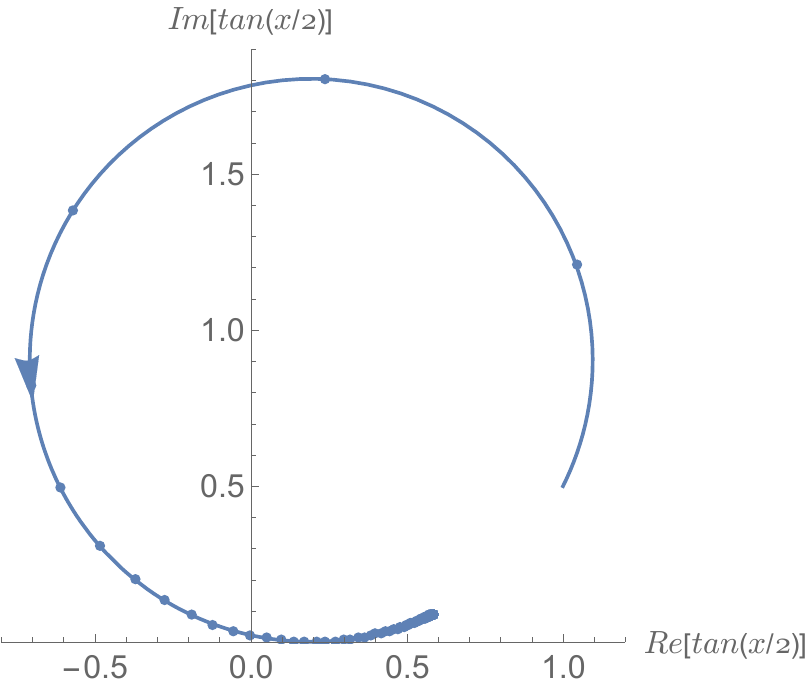}
\includegraphics[width=0.475\textwidth]{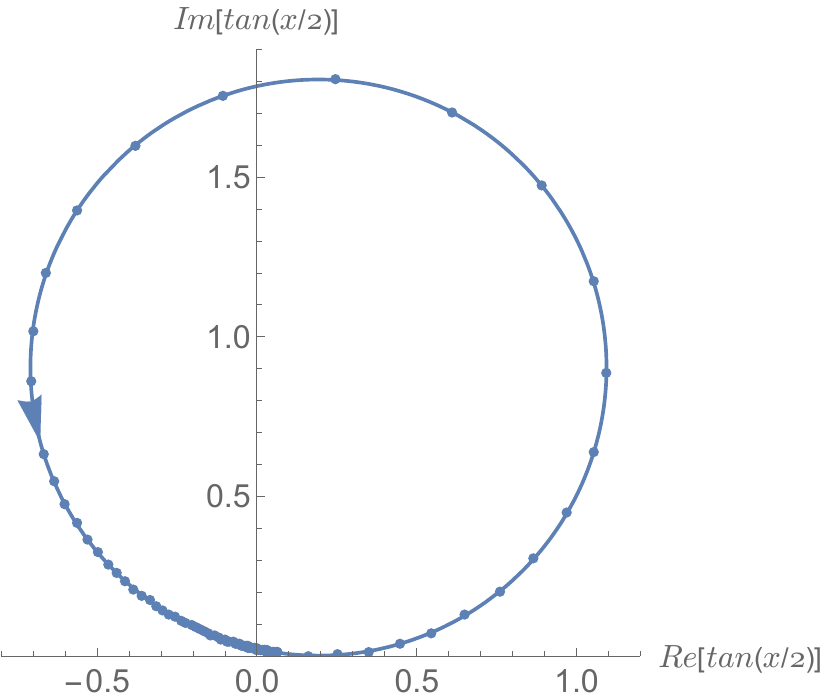}
\put(-455,150){\Large{(c)}}
\put(-225,150){\Large{(d)}}
\caption{Pole trajectories for $N=1,a=0,\sigma=0$ in $\tan(\frac{x}{2})$-space.
 Initial data is $v_c(0)=0.5 - \I$ and  $\omega_{-1}(0)  = 0.5 - 0.25 \I$, with different $\nu$ and $\omega_{av}(0)$.
(a) $\nu = 0.2, \omega_{av}(0)=0$: The pole crosses the real-$\tan(\frac{x}{2})$ line with $\frac{d Re[v_c(t_c)]}{dt} \neq 0$.
(b) $\nu \approx 0.2865, \omega_{av}(0)=0$:  The pole  tends to the real line as $t \rightarrow \infty$.
(c) $\nu =0.1, \omega_{av}(0) \approx 0.543$: The pole touches the real line  with $\frac{d Re[v_c(t)]}{dt} = 0$.
(d) $\nu =0.01, \omega_{av}(0)\approx 0.543$: \cblue{The pole traverses the circle touching the real line multiple times before it settles at $\tan(\frac{x}{2}) \approx 0.0806 + 0.0086\I$}.
Grid points shown on the trajectories are equispaced in $t$ and cluster near the final location of the trajectories.}
 \label{fig:pole_trajectories}
\end{figure}

\begin{figure}[h!]
\centering
\includegraphics[width=0.475\textwidth]{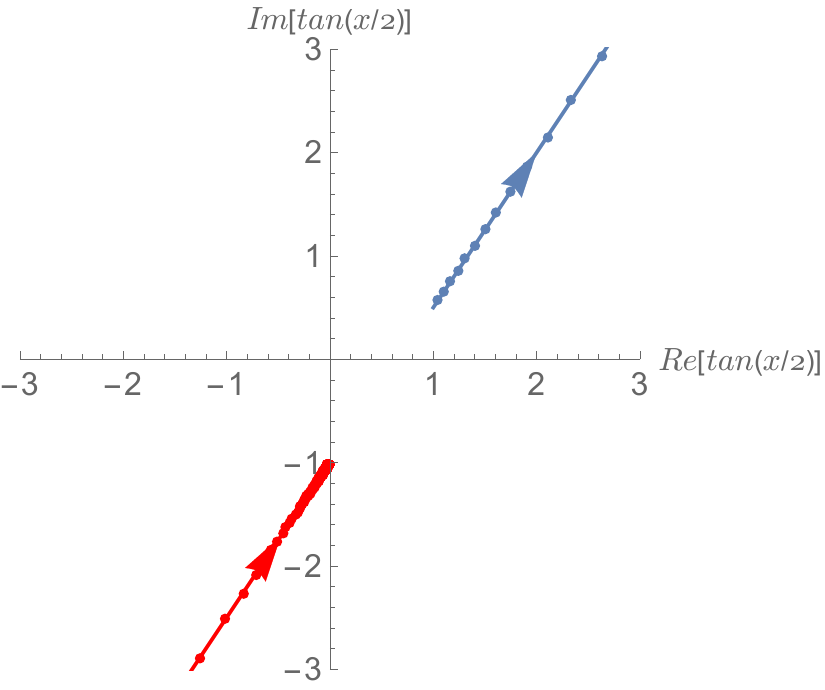}
\includegraphics[width=0.475\textwidth]{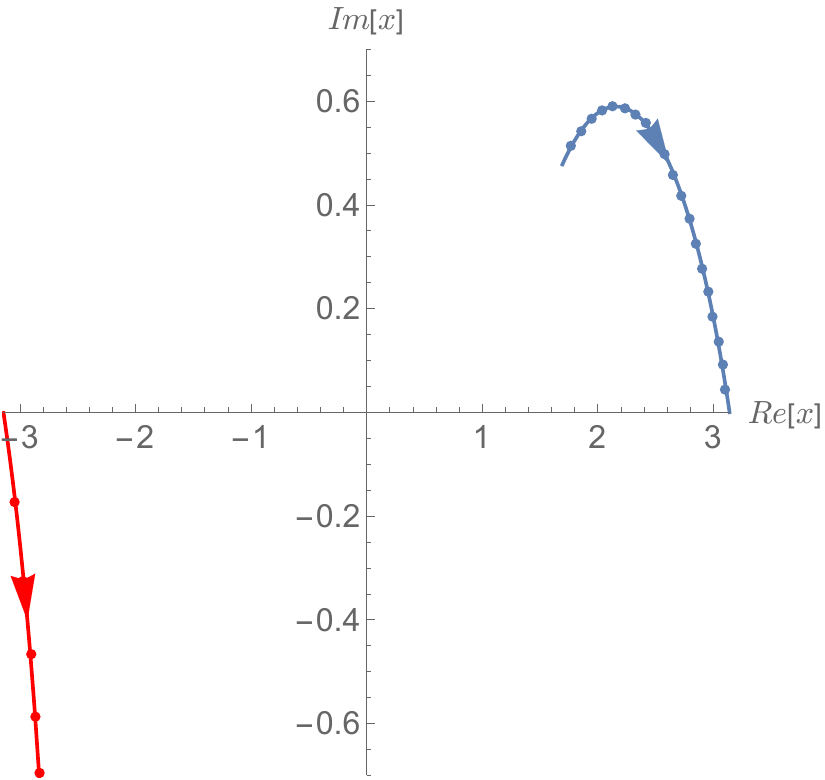}
\put(-450,200){\Large{(a)}}
\put(-220,200){\Large{(b)}}
\caption{Degenerate pole trajectory for $N=1,a=0,\sigma=0$.
Initial data is  $v_c(0)=0.5 - \I$,  $\omega_{-1}(0)  = 1.5 - \I$, $\omega_{av}(0)=0, \nu = 0$.
(a) The trajectory in $\tan(\frac{x}{2})$-space. The pole moves to $Re[v_c(t)] \rightarrow \infty$ as $t \rightarrow t_c$ (blue line)  and continues from  $Re[v_c(t)] = -\infty$ to its final point $v_c(\infty)=-1$ as $t \rightarrow \infty$ (red line).
(b) The same trajectory in $x$-space.  The pole  trajectory $x_c(t)$ crosses the real line at $x=\pm \pi$ with $\frac{d Im[x_c(t_c)]}{dt} \neq 0$.
Grid points shown are as in (a). }
 \label{fig:pole_trajectories2}
\end{figure}

{\underline{\em $N>1$ pairs of c.c. poles.} We now look for a solution of the form (\ref{N_poles}) with $N>1$.
Plug this into the first equation in (\ref{mainEquationa0sigma0_periodic_lower}),  separate products of poles into sums of poles using a partial fraction expansion, and equate like power poles to  obtain the following evolution equations for the (generally complex) pole amplitudes and positions:
\begin{align} \label{a0sigma0_periodic_eqns_N}
\frac{d \omega_{-1,k}}{dt} &= \frac{2}{1 + v_{c,k}} \omega_{-1,k}^2 +2 \omega_{-1,k} \sum_{\stackrel{l=1}{l \neq k}}^{N} \omega_{-1,l} \left( \frac{1}{v_{c,k}-v_{c,l}} + \frac{1}{1+v_{c,l}} \right) + (\I \omega_{av} - \nu) \omega_{-1,k}, \\
\frac{d v_{c,k}}{dt} &=\omega_{-1,k}, \nonumber
\end{align}
where $k=1 \ldots N$.
The term  $\omega_{av}$ still satisfies  (\ref{mainEquationa0sigma0_periodic_lower}), (\ref{a0sigma0_periodic_eqns}). Note that this reduces to (\ref{a0sigma0_periodic_eqns}) when $N=1$, in which case  the sum on the right-hand side of \e{a0sigma0_periodic_eqns_N} is absent.

These equations are difficult to solve analytically when $N \geq 2$, although it is straightforward to obtain numerical solutions.
 Analytical solutions for two pairs of poles in the real line geometry are obtained in  \cite{AmbroseLushnikovSiegelSilantyev}, and can exhibit collapse with different similarity exponents depending on the data.   A numerical example of non-generic collapse for $N=2$ in the case of $a=0, ~\sigma=1$ is given \cblue{at the end of Section \ref{sec:a0sigma1}}.



\subsection{A global existence theorem for ${a=0},$ ${\sigma=0,}$ ${\nu>0}$} \label{sec:a=0_well-posedness}

The results of Section \ref{sec:a0sigma0} include that the blowup solutions found there do not have arbitrarily small Wiener norm.
We show here that this is because solutions with small Wiener norm (in this case $a=0,$ $\sigma=0,$ and $\nu>0$) exist for all time.

In the present section, since $\sigma=0,$ the mean of $\tilde{\omega}$ is not conserved; therefore we will work here directly with
$\tilde{\omega}$ rather than $\omega.$
In the case $a=0,$ $\sigma=0,$ equation \eqref{CLM} becomes
\begin{equation}\nonumber 
\tilde{\omega}_{t}=\tilde{\omega} H(\tilde{\omega})-\nu\tilde{\omega}.
\end{equation}
We write this, with initial data $\tilde{\omega}_{0},$ in Duhamel form as
\begin{equation}\label{mildForm}
\tilde{\omega}(\cdot,t)=e^{-\nu t}\tilde{\omega}_{0}+\int_{0}^{t}e^{-\nu(t-s)}\tilde{\omega}(\cdot,s)H(\tilde{\omega})(\cdot,s)\ \mathrm{d}s.
\end{equation}
We let $B_{0}$ be the Wiener algebra; on $\mathbb{R},$ this is the space of functions with Fourier transform in $L^{1},$ and on
$\mathbb{S},$ this is the space of functions with Fourier series in $\ell^{1}.$  We assume $\tilde{\omega}_{0}\in B_{0}.$

We need space-time versions of $B_{0}.$  We denote these by $\mathcal{B}_{0}(\mathbb{R})$ and $\mathcal{B}_{0}(\mathbb{S}),$
in the real-line and periodic cases, respectively.  The norms are defined as
\begin{equation}\nonumber
\|f\|_{\mathcal{B}_{0}(\mathbb{R})}=\int_{\mathbb{R}}\sup_{t\in[0,\infty)}|\hat{f}(\xi,t)|\ \mathrm{d}\xi,
\end{equation}
\begin{equation}\nonumber
\|f\|_{\mathcal{B}_{0}(\mathbb{S})}=\sum_{k\in\mathbb{Z}}\sup_{t\in[0,\infty)}|\hat{f}(k,t)|.
\end{equation}
Note that $\mathcal{B}_{0}$ is an algebra,
\begin{equation}\nonumber
\|fg\|_{\mathcal{B}_{0}}\leq \|f\|_{\mathcal{B}_{0}}\|g\|_{\mathcal{B}_{0}},
\end{equation}
where $\mathcal{B}_{0}$ denotes either $\mathcal{B}_{0}(\mathbb{R})$ or $\mathcal{B}_{0}(\mathbb{T}).$  This algebra property is inherited from the
Wiener algebra, $B_{0}.$

In what follows, all statements are true in both $\mathcal{B}_{0}(\mathbb{R})$ and $\mathcal{B}_{0}(\mathbb{S}),$ but we will
focus the proof on the $\mathcal{B}_{0}(\mathbb{R})$ case.

\begin{theorem} Let $a=0,$ $\sigma=0,$ and $\nu>0.$  Let $\tilde{\omega}_{0}\in B_{0}(\mathbb{R})$ or $\tilde{\omega}_{0}\in B_{0}(\mathbb{S}).$
Assume $\|\tilde{\omega}_{0}\|_{B_{0}}<\frac{\nu}{4}.$  Then there exists a global mild solution $\tilde{\omega}\in\mathcal{B}_{0}$ of \eqref{CLM}.
That is, this $\tilde{\omega}$ satisfies \eqref{mildForm} for all $t\geq0.$
\end{theorem}

\begin{proof}  As we have said, we will only carry out the case on the real line, but proof in the case on $\mathbb{S}$ is identical, just replacing integrals with sums.

We show that the semigroup maps $B_{0}$ to $\mathcal{B}_{0}.$  We let $\tilde{\omega}_{0}\in B_{0}(\mathbb{R}),$ and we compute the norm of $e^{-\nu t}\tilde{\omega}_{0}:$
\begin{equation}\nonumber
\|e^{-\nu t}\tilde{\omega}_{0}\|_{\mathcal{B}_{0}(\mathbb{R})} =\int_{\mathbb{R}}\sup_{t\in[0,\infty)}e^{-\nu t}\left|\widehat{\tilde{\omega}_{0}}(\xi)\right|\ \mathrm{d}\xi
=\int_{\mathbb{R}}\left|\widehat{\tilde{\omega}_{0}}(\xi)\right|\ \mathrm{d}\xi=\|\tilde{\omega}_{0}\|_{B_{0}(\mathbb{R})}.
\end{equation}

We define the mapping $I^{+}$ to be
\begin{equation}\nonumber
(I^{+}h)(\cdot,t)=\int_{0}^{t}e^{-\nu(t-s)}h(\cdot,s)\ \mathrm{d}s.
\end{equation}
We now show that $I^{+}$ is a bounded linear operator from $\mathcal{B}_{0}(\mathbb{R})$ to $\mathcal{B}_{0}(\mathbb{R}).$
We begin by letting $h\in\mathcal{B}_{0}(\mathbb{R}),$ and we express the norm of $I^{+}h:$
\begin{equation}\nonumber
\|I^{+}h\|_{\mathcal{B}_{0}(\mathbb{R})} = \int_{\mathbb{R}}\sup_{t\in[0,\infty)}\left|
\int_{0}^{t}e^{-\nu(t-s)}\hat{h}(\xi,s)\ \mathrm{d}s\right|\ \mathrm{d}\xi.
\end{equation}
We use the triangle inequality and introduce another supremum to bound this as
\begin{equation}\nonumber
\|I^{+}h\|_{\mathcal{B}_{0}(\mathbb{R})}\leq\int_{\mathbb{R}}\sup_{t\in[0,\infty)}\int_{0}^{t}e^{-\nu(t-s)}\left[\sup_{\tau\in[0,\infty)}|\hat{h}(\xi,\tau)|\right]
\ \mathrm{d}s\ \mathrm{d}\xi.
\end{equation}
Rearranging, this becomes
\begin{equation}\label{integralToBeCompletedForI+Bound}
\|I^{+}h\|_{\mathcal{B}_{0}(\mathbb{R})}\leq\left(\int_{\mathbb{R}}\sup_{\tau\in[0,\infty)}|\hat{h}(\xi,\tau)|\ \mathrm{d}\xi\right)
\left(\sup_{t\in[0,\infty)}\int_{0}^{t}e^{-\nu(t-s)}\ \mathrm{d}s\right).
\end{equation}
The first factor on the right-hand side is just $\|h\|_{\mathcal{B}_{0}(\mathbb{R})},$ and we can compute that the second factor equals $1/\nu:$
\begin{equation}\nonumber
\|I^{+}h\|_{\mathcal{B}_{0}(\mathbb{R})}\leq\frac{\|h\|_{\mathcal{B}_{0}(\mathbb{R})}}{\nu}.
\end{equation}

We define the mapping $\mathcal{T}$ to be
\begin{equation}\nonumber
\mathcal{T}\tilde{\omega}=e^{-\nu t}\tilde{\omega}_{0}+I^{+}(\tilde{\omega} H(\tilde{\omega})).
\end{equation}
Then a mild solution of the $a=0,$ $\sigma=0$ problem is a fixed point of $\mathcal{T}.$
We define $r_{0}>0$ to be $r_{0}=\|e^{-\nu t}\tilde{\omega}_{0}\|_{\mathcal{B}_{0}(\mathbb{R})},$ and
we let $X$ be the closed ball in $\mathcal{B}_{0}(\mathbb{R})$
centered at $e^{-\nu t}\tilde{\omega}_{0},$ with radius $r_{1}.$  We will show that for appropriate choices of $r_{0}$ and $r_{1},$
that $\mathcal{T}$ is a contraction on $X.$  Note that for any $f\in X,$ we have $\|f\|_{\mathcal{B}_{0}(\mathbb{R})}\leq r_{0}+r_{1}.$

We first must show that $\mathcal{T}$ maps $X$ to $X.$  We let $f\in X$ be given.  We compute as follows:
\begin{equation}\nonumber
\|\mathcal{T}f-e^{-\nu t}\tilde{\omega}_{0}\|_{\mathcal{B}_{0}(\mathbb{R})}=\|I^{+}(fH(f))\|_{\mathcal{B}_{0}(\mathbb{R})}\leq
\frac{\|f\|_{\mathcal{B}_{0}(\mathbb{R})}^{2}}{\nu}\leq\frac{(r_{0}+r_{1})^{2}}{\nu}.
\end{equation}
If we have
\begin{equation}\label{XtoX}
\frac{(r_{0}+r_{1})^{2}}{\nu}\leq r_{1},
\end{equation}
then, indeed, $\mathcal{T}$ maps $X$ to $X.$

We now consider the contraction property.  Let $f$ and $g$ both be in the ball $X;$ then we may estimate $\mathcal{T}f-\mathcal{T}g$ as
\begin{equation}\nonumber
\|\mathcal{T}f-\mathcal{T}g\|_{\mathcal{B}_{0}(\mathbb{R})}=\|I^{+}(fH(f)-gH(g))\|_{\mathcal{B}_{0}(\mathbb{R})}\leq\frac{\|fH(f)-gH(g)\|_{\mathcal{B}_{0}(\mathbb{R})}}{\nu}.
\end{equation}
By adding and subtracting, we find the following bound:
\begin{equation}\nonumber
\|\mathcal{T}f-\mathcal{T}g\|_{\mathcal{B}_{0}(\mathbb{R})}\leq\frac{2(r_{0}+r_{1})}{\nu}\|f-g\|_{\mathcal{B}_{0}(\mathbb{R})}.
\end{equation}
Clearly, then, this is contracting as long as
\begin{equation}\label{contractingRequirement}
\frac{2(r_{0}+r_{1})}{\nu}<1.
\end{equation}
Taking $r_{0}=r_{1}<\frac{\nu}{4},$ then both of \eqref{XtoX} and \eqref{contractingRequirement} are satisfied.  This completes the proof.
\end{proof}

\subsection{A global existence theorem for ${a=0},$ ${\sigma>0,}$ ${\nu>0}$} \label{sec:a=0_well-posedness_1}

We remark that in the periodic case, we can also prove the same theorem for $\sigma>0.$ In this case the norm is preserved, so we deal with $\omega$ rather than
$\tilde{\omega}.$
\begin{theorem} Let $a=0,$ $\sigma>0,$ and $\nu>0.$  Let $\omega_{av}\in\mathbb{R}$ be given.  Let $\omega_{0}\in B_{0}(\mathbb{S}).$
Assume $\|\omega_{0}\|_{B_{0}(\mathbb{S})}<\frac{\nu}{4}.$  Then there exists $\omega\in\mathcal{B}_{0}$ such that $\tilde{\omega}=\omega+\omega_{av}$ is a global
mild solution of \eqref{CLM}.  That is, this $\tilde{\omega}$ satisfies \eqref{mildForm} for all $t\geq0.$
\end{theorem}
\begin{proof} We only remark on where the proof differs from the above.  The definition of $I^{+}$ becomes
\begin{equation}\nonumber
(I^{+}h)(\cdot,t)=\int_{0}^{t}e^{-\nu(t-s)\Lambda^{\sigma}}h(\cdot,s)\ \mathrm{d}s.
\end{equation}
Then the analogue of \eqref{integralToBeCompletedForI+Bound} is
\begin{equation}\nonumber
\|I^{+}h\|_{\mathcal{B}_{0}(\mathbb{S})}\leq\left(\sum_{k\in\mathbb{Z}\setminus\{0\}}\sup_{\tau\in[0,\infty)}|\hat{h}(k,\tau)|\right)
\left(\sup_{k\in\mathbb{Z}\setminus\{0\}}\sup_{t\in[0,\infty)}\int_{0}^{t}e^{-\nu|k|^{\sigma}(t-s)}\ \mathrm{d}s\right).
\end{equation}
On the right-hand side we recognize that the first factor is the norm of $h,$ and we evaluate the integral in the other factor.  This yields
\begin{equation}\nonumber
\|I^{+}h\|_{\mathcal{B}_{0}(\mathbb{S})}\leq \|h\|_{\mathcal{B}_{0}}
\left(\sup_{k\in\mathbb{Z}\setminus\{0\}}\sup_{t\in[0,\infty)}
\frac{1-e^{-\nu|k|^{\sigma}t}}{\nu|k|^{\sigma}}\right)=\frac{\|h\|_{\mathcal{B}_{0}(\mathbb{S})}}{\nu}.
\end{equation}
Then, the proof proceeds as in the previous theorem.  Notice that in the continuous case (if $k\in\mathbb{R}$ instead of $k\in\mathbb{Z}$), this
double supremum would not be finite.
\end{proof}

\subsection{${a=0}$ and ${\sigma=1}$} \label{sec:a0sigma1}

We next derive pole dynamics solutions for  $a=0$, $\sigma=1$ in which case  (\ref{CLM}) becomes
\begin{equation} \label{mainEquationa0sigma1_periodic}
\tilde{\omega}_{t}=\tilde{\omega} \mathcal{H}(\tilde{\omega})-\nu \mathcal{H}(\tilde{\omega}_x).
\end{equation}
Using $\tilde{\omega}=\omega_++\omega_-+\omega_{av}$ and  the Hilbert transform representation \e{Hilbert_rep} we can rewrite \e{mainEquationa0sigma1_periodic} as:
\begin{equation} \label{mainEquationa0sigma1_periodic_lower}
(\omega_-)_t=\I \omega_-^2 + \I \omega_{av} \omega_- -\I \nu (\omega_-)_x, \quad (\omega_+)_t=-\I \omega_+^2 -\I \omega_{av}\omega_+  +\I \nu (\omega_+)_x, \quad (\omega_{av})_t= 0.
\end{equation}

{\underline{\em One pair of c.c. poles (N=1).}
We look for a real-valued solution  to the above equations in the form of (\ref{omegasigma0_periodic_lower}).
Substituting \e{omegasigma0_periodic_lower} into \e{mainEquationa0sigma1_periodic_lower}  and equating like power poles gives evolution equations for the position and amplitude of the pole in the upper half-plane
\begin{align}\label{a0sigma1_periodic_eqns} 
\begin{split}
\frac{d\omega_{-1}(t)}{dt}&=(\I \omega_{av}(t)-\nu v_c(t))\omega_{-1}(t) + \frac{2\omega_{-1}^2(t)}{1+v_c(t)}, \\
 \frac{dv_c(t)}{dt} &=\omega_{-1}(t) + \frac{\nu}{2}(1-v_c^2(t)), \; \frac{d \omega_{av}(t)}{dt}=0.
 \end{split}
\end{align}
Solving the above equations gives:
\begin{equation}\label{a0sigma1_periodic_solution2_c}
\omega_{-1}(t) =\frac{ \omega_{-1}(0) e^{(\nu +\I \omega_{av}(0)) t} }{Z_+^2(t)}, \quad v_c(t) =-\frac{Z_-(t)}{Z_+(t)}, \quad \omega_{av}(t)=\omega_{av}(0), \\
\end{equation}
where $Z_\pm(t) = 1 - \frac{\omega_{-1}(0)}{1+v_c(0)}\frac{e^{\I \omega_{av}(0) t}-1}{\I \omega_{av}(0)} + \frac{\pm e^{\nu t}-1}{2}(1+v_c(0))$.


Alternatively, to get the solution \e{a0sigma1_periodic_solution2_c}, one can change variables $\tilde{t} = t, \tilde{x}=x-\I \nu t$ and rewrite (\ref{mainEquationa0sigma1_periodic_lower}) as
\begin{equation} \label{mainEquationa0sigma1_periodic_lower_changed}
(\omega_-)_{ \tilde{t}}=\I \omega_-^2 + \I \omega_{av} \omega_-,
\end{equation}
which is the same as (\ref{mainEquationa0sigma0_periodic_lower}) for $\nu=0$. Therefore the solution of (\ref{mainEquationa0sigma1_periodic_lower}) is
\begin{align}\label{omegasigma1_periodic_lower_solution_1}
\omega_-(x,t)&=\tilde{\omega}_{-1}(t)\left[\frac{1}{\tan(\frac{x-\I\nu t}{2})-\I \tilde{v}_c(t)} - \frac{1}{-\I-\I \tilde{v}_c(t)}\right] \nonumber \\
&=\omega_{-1}(t)\left[\frac{1}{\tan(\frac{x}{2})-\I v_c(t)} - \frac{1}{-\I-\I v_c(t)}\right],
\end{align}
where
\begin{equation}\label{a0sigma1_periodic_solution}
\omega_{-1}(t)=\tilde{\omega}_{-1}(t)\frac{1- \chi^2(t)}{(1+\tilde{v}_c(t) \chi(t))^2}, \
v_c(t)=\frac{\tilde{v}_c(t)+ \chi(t)}{1+\tilde{v}_c(t) \chi(t)}, \
\chi(t)=\tanh\left(\frac{\nu t}{2}\right),
\end{equation}
and  $\tilde{\omega}_{-1}(t), \tilde{v}_c(t)$ is a solution of \e{mainEquationa0sigma1_periodic_lower_changed}, which we can obtain by letting $\nu \rightarrow 0$
in  \e{a0sigma0_periodic_solution_c}, that is
\begin{equation}\label{a0sigma0_periodic_solution_c_nu0}
\tilde{\omega}_{-1}(t)=\frac{\omega_{-1}(0)e^{\I \omega_{av}(0)t}}{\left(1-\frac{\omega_{-1}(0)}{1+v_c(0)}\frac{e^{\I \omega_{av}(0)t}-1}{\I \omega_{av}(0)} \right)^2}, \
\tilde{v}_c(t)=\frac{v_c(0)+\frac{\omega_{-1}(0)}{1+v_c(0)}\frac{e^{\I \omega_{av}(0)t}-1}{\I \omega_{av}(0)}}{\left(1-\frac{\omega_{-1}(0)}{1+v_c(0)}\frac{e^{\I \omega_{av}(0)t}-1}{\I \omega_{av}(0)} \right)}, \
\omega_{av}(t)=\omega_{av}(0).
\end{equation}
Rewriting (\ref{a0sigma1_periodic_solution}) with (\ref{a0sigma0_periodic_solution_c_nu0}) gives the solution (\ref{a0sigma1_periodic_solution2_c}).

In the limit $\omega_{av}(0) \rightarrow 0$ the solution \e{a0sigma1_periodic_solution2_c} reduces to
\begin{equation}\label{a0sigma1_periodic_solution2}
\omega_{-1}(t)=\frac{ \omega_{-1}(0) e^{\nu t} }{(1 - \frac{\omega_{-1}(0)}{1+v_c(0)}t + \frac{e^{\nu t}-1}{2}(1+v_c(0)))^2 }, \
v_c(t)=\frac{ -1 + \frac{\omega_{-1}(0)}{1+v_c(0)}t + \frac{e^{\nu t}+1}{2}(1+v_c(0)) }{ 1 - \frac{\omega_{-1}(0)}{1+v_c(0)}t + \frac{e^{\nu t}-1}{2}(1+v_c(0)) }.
\end{equation}
When $\nu=0$ the solution \e{a0sigma1_periodic_solution2} is reduced to \e{a0sigma0_periodic_solution_nu0}.

We analyze \e{a0sigma1_periodic_solution2_c} to characterize collapsing solutions.
For simplicity, we consider the case with $\omega_{av}(0)=0$ although a more general case $\omega_{av}(0)\ne 0$ can be also analyzed giving qualitatively similar results.
Inspection of  $v_c(t)$ in (\ref{a0sigma1_periodic_solution2}) for real $\omega_{-1}(0)$ and $v_c(0)$ shows that collapse is possible when (i) $v_c(t_c)=0$ for $\omega_{-1}(0)<0$ and $v_c(0)>0$ or (ii) $v_c(t_c)=\infty$ for $\omega_{-1}(0)>0$ and $v_c(0)>0$. We consider these two cases in more detail.

{\underline{ \em (i)  $\omega_{-1}(0)<0$ \mbox{and} $v_c(0)>0$}}:  both denominators of (\ref{a0sigma1_periodic_solution2}) are positive for all $t>0$. Consider two parts of the numerator of $v_c(t)$ in (\ref{a0sigma1_periodic_solution2}):
\begin{equation} \label{a0sigma1_periodic_solution_numerator}
N_1(t)=-1 - \frac{|\omega_{-1}(0)|}{1+v_c(0)}t, \qquad N_2(t)=\frac{e^{\nu t}+1}{2}(1+v_c(0)),
\end{equation}
so that  numerator  is 0 whenever $-N_1(t)=N_2(t)$ and a solution is a collapsing solution. Both $-N_1(t)$ and $N_2(t)$ are increasing functions and the equation $-N_1(t)=N_2(t)$ can have 0, 1 or 2 solutions depending on $|\omega_{-1}(0)|$.
We fix $v_c(0)$ and find the smallest possible $|\omega_{-1}(0)|$ (i.e., the largest $\omega_{-1}(0)$)  so that there is  a collapsing solution. In this case the equation $-N_1(t)=N_2(t)$ has one solution and we have two conditions at the collapse time $t=t_c$:
\begin{equation} \label{a0sigma1_periodic_solution_single_sol_conditions}
-N_1(t_c)=N_2(t_c), \qquad -N_1'(t_c)=N_2'(t_c).
\end{equation}
Solving these two equations gives
\begin{equation} \label{a0sigma1_periodic_solution_single_sol}
e^{\nu t_c}(1-\nu t_c)=\frac{1-v_c(0)}{1+v_c(0)},
\qquad \omega_{-1}^{l}(0)=-\frac{\nu}{2}(1+v_c(0))^2 e^{\nu t_c}.
\end{equation}
where the superscript $l$ denotes that this is the largest such $\omega_{-1}(0)$ leading to a collapsing solution.

With this  $\omega_{-1}^{l}(0)$ we get from (\ref{a0sigma0_periodic_solution_norms})
\begin{equation} \nonumber
\| \omega(\cdot,0) \|_{B_0} =\frac{|\omega_{-1}(0)|}{v_c(0)}\left(\frac{|1 - v_c(0)|}{1 + v_c(0)} +1\right) \geq
\frac{\nu}{2}(1 + v_c(0))\left(\frac{|1 - v_c(0)|+ 1 + v_c(0)}{v_c(0)}\right) e^{\nu t_c} \geq 2e\nu,
\end{equation}
where the last inequality is obtained by minimizing the previous expression, taking into account relation \e{a0sigma1_periodic_solution_single_sol} between $v_c(0)$ and $\nu t_c$. This minimum is at $v_c(0)=1, \nu t_c=1$.

For the  $L^2$-norm  we get from (\ref{a0sigma0_periodic_solution_norms})
\begin{multline}\nonumber
\|\omega(\cdot,0)\|_{L^2} = 2\sqrt{\pi}\frac{|\omega_{-1}(0)|}{\sqrt{v_c(0)}(1 + v_c(0))} \geq 2\sqrt{\pi}\frac{|\omega_{-1}^{l}(0)|}{\sqrt{v_c(0)}(1 + v_c(0))}= \\
\sqrt{\pi}\nu\frac{(1+v_c(0))}{\sqrt{v_c(0)}} e^{\nu t_c}
=\frac{2\sqrt{\pi}\nu}{\sqrt{e^{-2\nu t_c}-(1-\nu t_c)^2}}
\geq \frac{2\sqrt{\pi}\nu}{\sqrt{e^{-2m}-(1-m)^2}} \approx 8.81 \nu,
\end{multline}
where
$\nu t_c=m=0.7968 \dots$
 minimizes the second to last expression in the above inequality.
This shows that $B_0$ and $L_2$ norms of the initial data of a collapsing solution cannot be made arbitrarily small.
Collapsing solutions exist for all $\omega_{-1}(0) \leq \omega_{-1}^{l}(0)$ defined in \e{a0sigma1_periodic_solution_single_sol}.

{\underline {\em (ii)  $\omega_{-1}(0)>0$ and $v_c(0)>0$}.  Note that the numerator of (\ref{a0sigma1_periodic_solution2}) is positive for all $t>0$.
Consider two parts of the denominator of (\ref{a0sigma1_periodic_solution2}):
\begin{equation}\nonumber
D_1(t)=1 - \frac{\omega_{-1}(0)}{1+v_c(0)}t, \qquad D_2(t)=\frac{e^{\nu t}-1}{2}(1+v_c(0)),
\end{equation}
so that the denominator of (\ref{a0sigma1_periodic_solution2}) is 0 whenever $-D_1(t)=D_2(t)$ and the  solution is a collapsing solution. Both $-D_1(t)$ and $D_2(t)$ are increasing functions and the equation $-D_1(t)=D_2(t)$ can have 0, 1 or 2 solutions depending on $\omega_{-1}(0)$.
Fix $v_c(0)$ and find the smallest possible $\omega_{-1}(0)$ so that we have a collapsing solution. In this case the equation $-D_1(t)=D_2(t)$ has one solution and we have two conditions at the collapse time $t=t_c$:
\begin{equation}\nonumber
-D_1(t_c)=D_2(t_c), \qquad -D_1'(t_c)=D_2'(t_c).
\end{equation}
Solving these two equations gives
\begin{equation} \label{a0sigma1_periodic_solution_single_sol_c2}
e^{\nu t_c}(\nu t_c-1)=\frac{1-v_c(0)}{1+v_c(0)},
\qquad \omega_{-1}^{s}(0)=\frac{\nu}{2}(1+v_c(0))^2 e^{\nu t_c}.
\end{equation}
where the superscript $s$ denotes that this is the smallest $\omega_{-1}(0)$ leading to collapse.

Inserting  $\omega_{-1}^{s}(0)$ into  (\ref{a0sigma0_periodic_solution_norms}) we get
\begin{equation}\nonumber
\| \omega(\cdot,0) \|_{B_0} =\frac{|\omega_{-1}(0)|}{v_c(0)}\left(\frac{|1 - v_c(0)|}{1 + v_c(0)} +1\right) \geq
\frac{\nu}{2}(1 + v_c(0))\left(\frac{|1 - v_c(0)|+ 1 + v_c(0)}{v_c(0)}\right) e^{\nu t_c} \geq 2e\nu,
\end{equation}
where the last inequality is obtained by minimizing the previous expression, taking into account relation \e{a0sigma1_periodic_solution_single_sol_c2} between $v_c(0)$ and $\nu t_c$. The minimum is at $v_c(0)=1, \nu t_c=1$.

For the  $L^2$-norm we get from (\ref{a0sigma0_periodic_solution_norms})
\begin{multline}\nonumber
\|\omega(\cdot,0)\|_{L^2} = 2\sqrt{\pi}\frac{|\omega_{-1}(0)|}{\sqrt{v_c(0)}(1 + v_c(0))} \geq 2\sqrt{\pi}\frac{\omega_{-1}^{s}(0)}{\sqrt{v_c(0)}(1 + v_c(0))} \\
=\sqrt{\pi}\nu\frac{(1+v_c(0))}{\sqrt{v_c(0)}} e^{\nu t_c}
=\frac{2\sqrt{\pi}\nu}{\sqrt{e^{-2\nu t_c}-(1-\nu t_c)^2}}
\geq \frac{2\sqrt{\pi}\nu}{\sqrt{e^{-2m}-(1-m)^2}} \approx 8.81 \nu,
\end{multline}
where $\nu t_c=m=0.7968 \dots$ minimizes $\|\omega(\cdot,0)\|_{L^2}$.
This shows that $B_0$ and $L_2$ norms of the initial data of a collapsing solution are bounded from below, i.e.,  cannot be made arbitrarily small.
Additionally, we find that collapsing solutions exist for all $\omega_{-1}(0) \geq \omega_{-1}^{s}(0)$ defined in \e{a0sigma1_periodic_solution_single_sol_c2}.

In both  (i) and (ii), the generic collapse
belongs to the general self-similar form (\ref{self-similar1}) with $\alpha=\beta=1$ and  $\| \omega(\cdot,t) \|_{B_0},\| \omega(\cdot,t) \|_{L^2} \rightarrow \infty$ as $t\rightarrow t_c$.  This follows from  $\omega_{-1}(t_c)\neq0$ and $v_c\sim (t_c-t)$ in  (i), and  $\omega_{-1} \sim (t_c-t)^{-2}$ and $v_c\sim (t_c-t)^{-1}$ in  (ii). However, there can also be non-generic collapse with different $\alpha$ and $\beta$.

When the initial data does not fall into case (i) or (ii), the solution exists globally in time. More precisely,
if $\omega_{-1}^{s}(0) < \omega_{-1}(0) < \omega_{-1}^{l}(0)$ then from (\ref{a0sigma1_periodic_solution2}) we have  $0<v_c(t)<\infty$ for all $t>0$ and the solution \e{omegasigma1_periodic_lower_solution_1}, \e{a0sigma1_periodic_solution2} exists for any $t>0$.

An exact pole dynamics solution to the problem on $\mathbb{R}$ for $a=0$, $\sigma=1$, which similarly consists of a pair of c.c. poles, was constructed
in \cite{AmbroseLushnikovSiegelSilantyev}.  This solution can blow up for arbitrarily small $B_0$ and $L^2$ data, illustrating a contrast between the problems on $\mathbb{R}$ and $\mathbb{S}$.

{\underline{\em $N>1$ pairs of c.c. poles.}  We again look for a solution of the form (\ref{N_poles}).
Substitute  this into the first equation of (\ref{mainEquationa0sigma1_periodic_lower}) and  separate products of poles into sums of poles using a partial fraction expansion.  Note that the presence of the derivative term $(\omega_{-})_x$ leads to factors of the form $\sec^2(x)/(\tan(\frac{x}{2}) - \I v_{c,k})^2$,  which need to be  reexpressed as  a sum of simple and double poles like those in (\ref{N_poles}). This is done  using $\tan(\frac{x}{2})=\tan(\frac{x}{2}) -\I v_{c,k} + \I v_{c,k}$ to rewrite the identity $\sec^2(\frac{x}{2})=1+\tan^2(\frac{x}{2})$  as
\begin{equation}\nonumber
\sec^2\left(\frac{x}{2}\right) = \left(\tan\left(\frac{x}{2}\right) -\I v_{c,k}\right)^2 + 2 \I v_{c,k} \left(\tan\left(\frac{x}{2}\right) -\I v_{c,k}\right) + 1-v_{c,k}^2,
\end{equation}
which immediately gives the desired representation. Equating like power poles then
gives evolution equations for the complex pole amplitudes and positions:
\begin{align}
\begin{split} \label{N_poles_a=0_sig=1}
\frac{d \omega_{-1,k}}{dt} &= \frac{2}{1 + v_{c,k}} \omega_{-1,k}^2 +2 \omega_{-1,k} \sum_{\stackrel{l=1}{l \neq k}}^{N} \omega_{-1,l} \left( \frac{1}{v_{c,k}-v_{c,l}} + \frac{1}{1+v_{c,l}} \right) + (\I \omega_{av} - \nu v_{c,k}) \omega_{-1,k} \\
\frac{d v_{c,k}}{dt} &=\omega_{-1,k} + \frac{\nu}{2} (1- v_{c,k}^2),
\end{split}
\end{align}
for $k=1, \ldots N$, while  $\omega_{av}$ satisfies the second equation of (\ref{a0sigma1_periodic_eqns}), so that $\omega_{av}(t)=\omega_{av}(0)$.
Note that (\ref{N_poles_a=0_sig=1}) reduces to  (\ref{a0sigma1_periodic_eqns}) when $N=1$.

An example numerical solution for $N=2$ with real-valued $\omega_{-1,k}$ and $v_{c,k}$ is shown in Figure \ref{fig:two_pole}. The example illustrates   nongeneric singularity formation in which $v_{c,1}(t) \rightarrow 0$ as $t \rightarrow t_c$ with $dv_{c,1}(t_c)/dt=0$.  In this case the pole `bounces off' the real line, and  (\ref{N_poles}) implies  that $\omega(0,t) \sim (t_c-t)^{-2}$ as $t \rightarrow t_c$ since $v_{c,1} \sim  (t_c-t)^2$ near the singularity time. It follows that the similarity exponents are $\alpha=\beta=2$. More typically a pole reaches the real line with $v_{c,1} \sim t-t_c$ as $t \rightarrow t_c$, so that $\alpha=\beta=1$.
This again shows singularity exponents are data dependent. \cblue{We anticipate that adding additional poles will give the flexibility to prescribe $v_c''(t_c)$ and even
higher-order derivatives of $v_c$ to be zero at the singularity time $t_c$,  enabling solutions with other distinct similarity exponents.}

\begin{figure}[h!]
\centering
\includegraphics[width=0.325\textwidth]{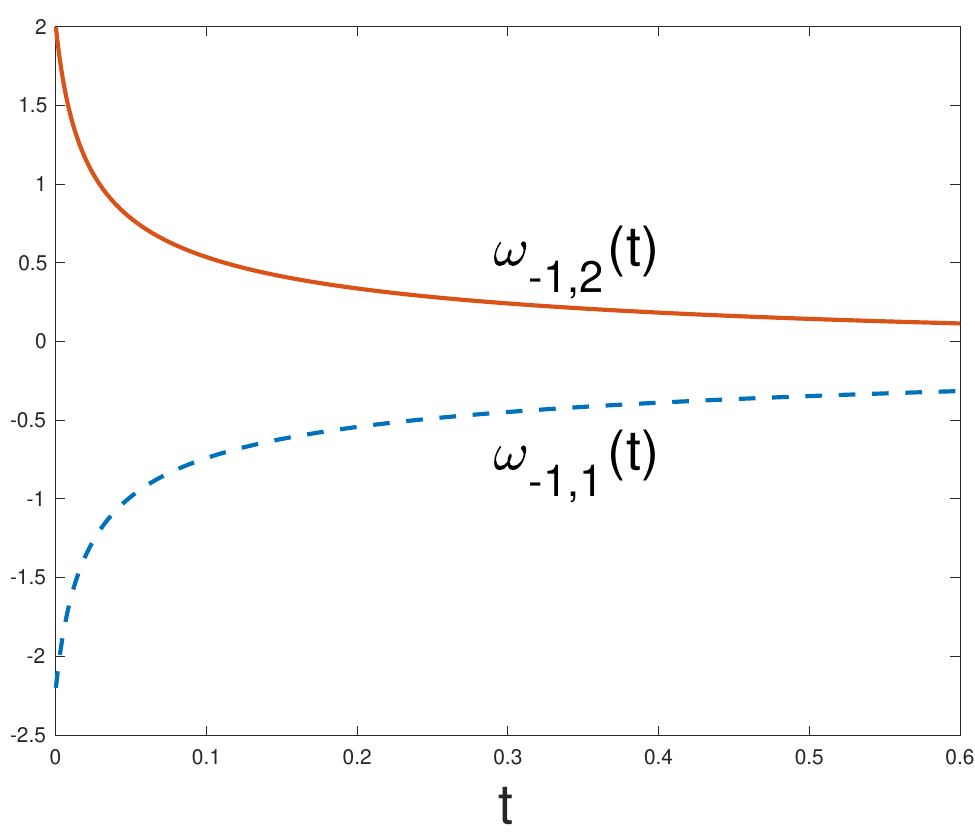}
\includegraphics[width=0.325\textwidth]{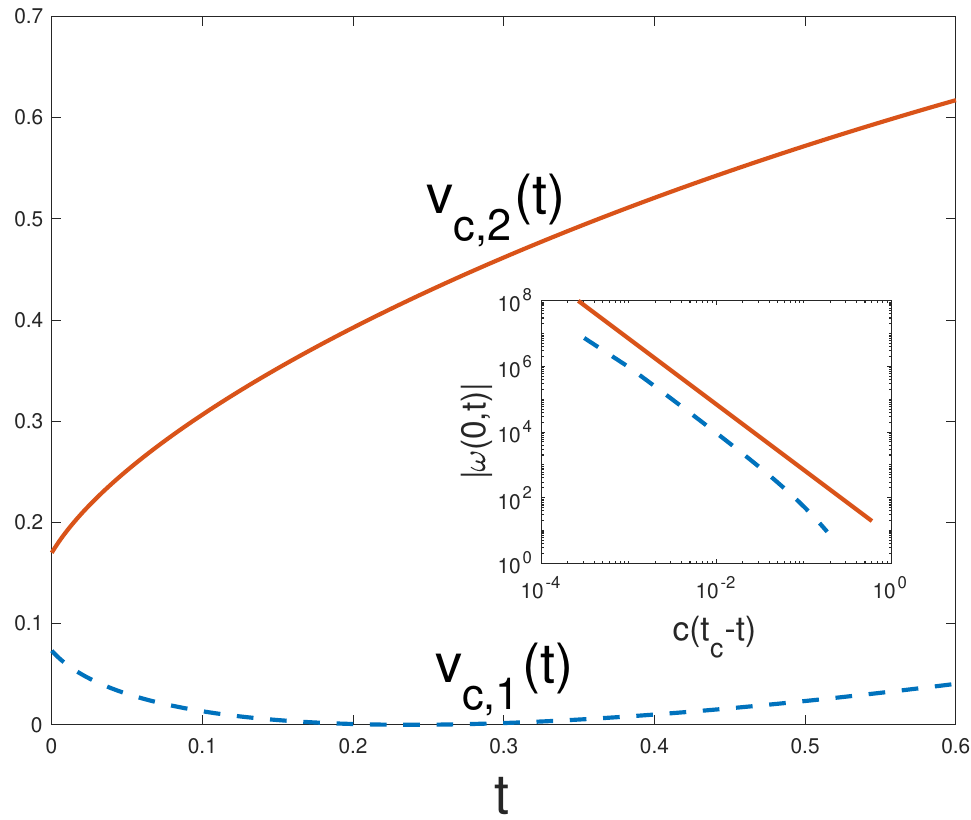}
\includegraphics[width=0.32\textwidth]{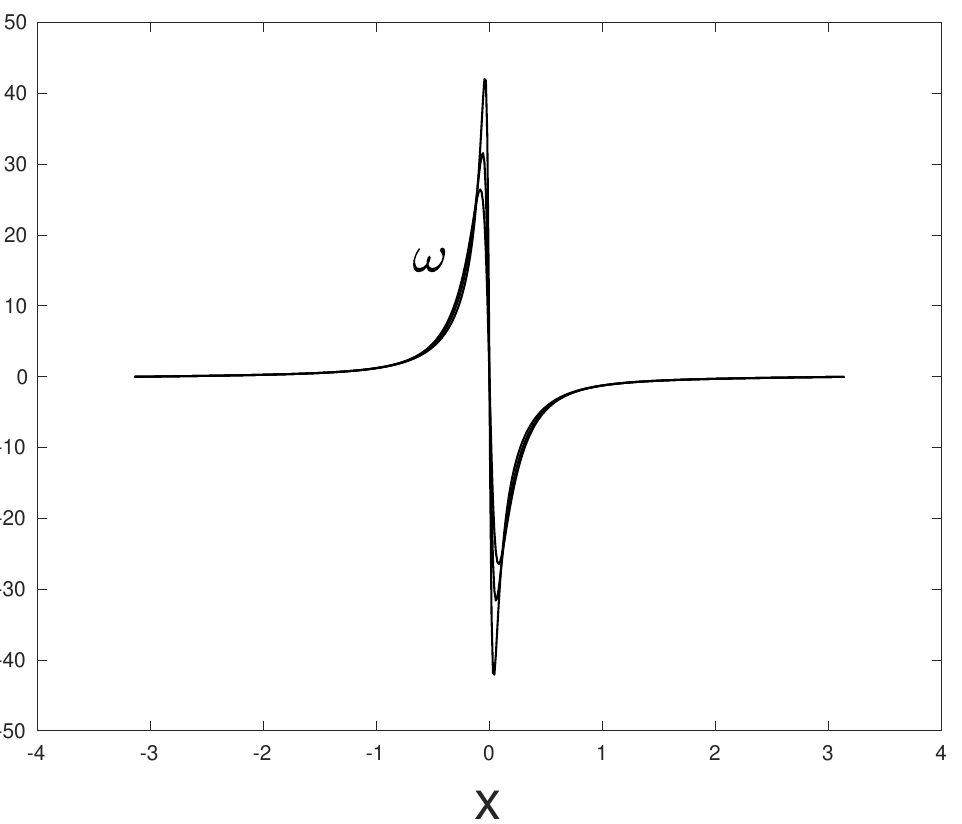}
\put(-450,110){\Large{(a)}}
\put(-295,110){\Large{(b)}}
\put(-137,110){\Large{(c)}}
\caption{ $N=2$ solution for  $\omega_{-1,1}(0)=-2.2$,  $\omega_{-1,2}(0)=2$, $v_{c,1}(0)=0.0732644,~v_{c,2}(0)=0.17$. (a) Time evolution of the pole amplitudes. (b) Time evolution of the pole positions. Inset: $|\omega(0,t)|$ versus $dv_{c,1}/dt \sim (t_c-t)$ (dashed blue curve), which approaches slope $-2$ (solid red curve)  as $t \rightarrow t_c$. This shows the similarity exponent $\beta=2$ in this example. \cblue{(c) Plot of $\omega$ at $t=0.029,~0.051,~0.079$.} }
 \label{fig:two_pole}
\end{figure}

\section{Pole dynamics solutions for $a=1/2$} \label{sec:a=one_half}

When $a=1/2$ the presence of the  nonlinear advection term $a u \omega_x$ in  (\ref{CLM})  generates  logarithmic singularities (starting from initial poles) which complicates finding a closed collection of dynamically evolving complex singularities. Instead of a simple pole, we look for a solution of \e{CLM} and \e{eq:omega_decomp} as $\tilde{\omega}=\omega_+ + \omega_- + \omega_{av}$ from \e{omegaplusminus} and \e{complexconjugationomega}   with $\omega_-$ a sum of a double pole and a simple pole in $X=\tan(\frac{x}{2})$-space \e{Xdef}:
\begin{align}\label{omegaa0.5sigma0_periodic_lower}
 &\omega_-(x,t)=\omega_{-1}(t)\left[\frac{1}{\tan(\frac{x}{2})-\I v_c(t)}- \frac{1}{-\I-\I v_c(t)}\right] \nonumber \\ & \qquad\qquad+ \omega_{-2}(t)\left[\frac{1}{(\tan(\frac{x}{2})-\I v_c(t))^2} - \frac{1}{(-\I-\I v_c(t))^2}\right],
\end{align}
where $\omega_{-1}(0), \omega_{-2}(0)$ and $v_c(0)$ are arbitrary complex constants with $Re[v_c(0)] > 0$ and $\omega_{-1}(t)$, $\omega_{-2}(t)\in\mathbb{C}$, $c(t)\in\mathbb{R}$ are  arbitrary smooth functions of time. The terms $\frac{1}{-\I-\I v_c(t)}$ and $\frac{1}{(-\I-\I v_c(t))^2}$ in \e{omegaa0.5sigma0_periodic_lower} are subtracted  in square brackets such that $\omega_-(x,t)$ has zero spatial mean value \cblue{ as required by \e{eq:omega_decomp}.}
As before, we supplement $\omega_-(x,t)$ from \e{omegaa0.5sigma0_periodic_lower} with the property \e{complexconjugationomega}
 to get a real-valued solution $\omega=\omega_-+\omega_+$ of \e{mainEquationa0.5sigma0_periodic}.  It is shown below in Section \ref{sec:a1p2sigma0} that a special choice of $\omega_{-1}(t)$ given by
 \begin{equation}\label{omegaa0.5sigma0_periodic_lower_1}
\omega_{-1}(t)=\frac{2\I v_c(t)}{1-v_c^2(t)}\omega_{-2}(t).
\end{equation}
 is required to avoid logarithmic singularities.


We compute the norms of $\omega=\omega_+ + \omega_-$, assuming  \e{complexconjugationomega},  \e{omegaa0.5sigma0_periodic_lower}, and \e{omegaa0.5sigma0_periodic_lower_1}:
\begin{equation}
\begin{split}
 \| \omega(\cdot,t) \|_{L^2}^2 &= 2 \pi |\omega_{-2}(t)|^2 \left( \frac{1+|v_c(t)|^2}{Re[v_c(t)]^3((1+|v_c(t)|^2)^2 - 4Re[v_c(t)]^2 ) } \right)  \label{a0.5sigma0_periodic_solution_norms}\\
 \| \omega(\cdot,t) \|_{B_0}& =\frac{|\omega_{-2}(t)|}{2Re[v_c(t)]^2}\left(\left|\frac{1 - v_c(t)}{1 + v_c(t)}\right| +1\right)\left(\left|\frac{1 + v_c(t)}{1 - v_c(t)}\right| +1\right).
\end{split}
\end{equation}
In the special case of purely imaginary $\omega_{-2}(t)$ and purely real $v_c(t)$, i.e.,
\begin{equation}\nonumber
\omega_{-2}(t)=\I \omega_{-2,i}(t), \quad Im[\omega_{-2,i}(t)]= 0, \quad Im[v_c(t)]=0,
\end{equation}
 we find that the norms \e{a0.5sigma0_periodic_solution_norms} simplify to
\begin{equation}
\begin{split}
 \| \omega(\cdot,t) \|_{L^2}^2 &= 2 \pi \omega_{-2,i}^2 \left( \frac{1+v_c^2}{v_c^3(1-v_c^2)^2 } \right)  \label{a0.5sigma0_periodic_solution_norms_real}\\
 \| \omega(\cdot,t) \|_{B_0}&=
\left\{\begin{matrix} \frac{2|\omega_{-2,i}(t)|}{v_c^2(t) (1-v_c^2(t))}, ~\mbox{for} ~ 0< v_c \leq 1 \\[7pt]  \frac{2|\omega_{-2,i}(t)|}{v_c^2(t)-1}, ~\mbox{for}~ v_c > 1.
~~~~~~~~ \end{matrix} \right.
\end{split}
\end{equation}

\subsection{$a=1/2$ and $\sigma=0$} \label{sec:a0.5sigma0}
\label{sec:a1p2sigma0}

We take $a=1/2$, $\sigma=0$, $\nu>0$, in which case  (\ref{CLM}) becomes
\begin{equation} \label{mainEquationa0.5sigma0_periodic}
\tilde{\omega}_{t}=-\frac{1}{2}u \tilde{\omega}_x +\tilde{ \omega} \mathcal{H}(\tilde{\omega})-\nu \tilde{\omega}, \quad u_x =\cal{H} \tilde{\omega}.
\end{equation}
%
%

Using the second equation in \e{mainEquationa0.5sigma0_periodic}, we recover $u(x,t) $ by integration as %
\begin{equation}\label{uint}
u(x)=\int_{x_0(t)}^{x}\mathcal{H}(\omega(x',t))dx',
 \end{equation}
 where $\omega=\omega_-+\omega_+$ with $\omega_-$ given by \e{omegaa0.5sigma0_periodic_lower} and  $\omega_+$   recovered from \e{complexconjugationomega}. Also $x_0(t)\in\mathbb{R}$ is an arbitrary smooth function of time. Note that one can also assume that
$x_0(t)\in\mathbb{C}$ which generalizes $u$ to complex values, which is beyond the scope of this paper.  Generally, the integration in  \e{uint} results in logarithmic terms.  These logarithmic terms show up only in the term $-\frac{1}{2}u\omega_x$ in \e{mainEquationa0.5sigma0_periodic} thus they cannot be canceled out with other terms which generally prevents  \e{omegaa0.5sigma0_periodic_lower} from being the exact solution of \e{mainEquationa0.5sigma0_periodic}. Thus we look for a restriction on values of  $\omega_{-1}(0)$ and $ \omega_{-2}(0)$ which ensures a vanishing of logarithmic terms. We perform a change of  variable \e{Xdef} 
in  \e{uint} resulting in
\begin{equation}\label{uintX}
u(x,t)=\int_{X_0(t)}^{X}\mathcal{H}(\omega(x',t))\frac{2}{X'^2+1}dX',
 \end{equation}
where $X_0(t)=\tan\left (\frac{x_0(t)}{2}\right).$

A partial fraction expansion  of the  integrand $\mathcal{H}(\omega(x))\frac{2}{X^2+1}$ of  \e{uintX} together with \e{Hilbert_rep}    and \e{omegaa0.5sigma0_periodic_lower}
reveal that the only source of the logarithmic singularity in the upper half-plane are the terms%
\begin{equation}\label{logsource}
\frac{-4\I\omega_{-2}(t)v_c(t)}{[X-\I v_c(t)][1-v_c^2(t)]^2}+\frac{2\omega_{-1}(t)}{[X-\I v_c(t)][1-v_c^2(t)]}
\end{equation}
which cancel out under the condition \e{omegaa0.5sigma0_periodic_lower_1}.
%

We additionally wish to avoid a constant term in the integrand of  \e{uint} because it results in a term $\propto x$ in $u$ which prevents \e{omegaa0.5sigma0_periodic_lower} from being the exact solution of  \e{mainEquationa0.5sigma0_periodic}.
Therefore, in  the integrand $\mathcal{H}(\omega(x))\frac{2}{X^2+1}$ of  \e{uintX}, we want to remove terms $\propto \frac{1}{X^2+1}$ after the partial fraction expansion. However, condition \e{omegaa0.5sigma0_periodic_lower_1}  also ensures that removal.  Then the integrand  $\mathcal{H}(\omega(x))\frac{2}{X^2+1}$ in \e{uintX} is given by %
\begin{equation}\label{uintegrandap1p2sigma0}
\mathcal{H}(\omega(x,t))\frac{2}{X^2+1} =\frac{2\I\omega_{-2}(t)}{[X-\I v_c(t)]^2[1- v_c^2(t)]}-\frac{2\I\bar \omega_{-2}(t)}{[X+\I\bar v_c(t)]^2[1- \bar v_c^2(t)]}
\end{equation}
which is immediately integrated in $X$\ giving%
\begin{equation}\label{uap1p2sigma0a}
u(x,t)=\frac{-2\I\omega_{-2}(t)}{[X-\I v_c(t)][1- v_c^2(t)]}+\frac{2\I\bar \omega_{-2}(t)}{[X+\I\bar v_c(t)][1- \bar v_c^2(t)]}+ q(t),
\end{equation}
where $ q(t)\in\mathbb{R}$ is an arbitrary smooth function of time determined by $X_0(t)$ of  \e{uintX}.
It is convenient to rewrite  \e{uap1p2sigma0a} in the following equivalent form
\begin{align}\label{uap1p2sigma0}
u(x,t)=\frac{-2\I\omega_{-2}(t)}{[1- v_c^2(t)]}\left[\frac{1}{X-\I v_c(t)} - \frac{1}{-\I-\I v_c(t)}\right]\nonumber \\+\frac{2\I\bar \omega_{-2}(t)}{[1- \bar v_c^2(t)]}\left[\frac{1}{X+\I\bar v_c(t)} - \frac{1}{\I+\I\bar  v_c(t)}\right]+ \tilde q(t),
\end{align}
where any nonzero  mean value $\frac{1}{2\pi}\int_{-\pi}^\pi{u(x,t)dx}$ of $u(x,t)$ is absorbed into a term $\tilde q(t)$ such that%
\begin{equation}\label{qtdef}
 \tilde q(t)=\frac{-2\I\omega_{-2}(t)}{[1- v_c^2(t)]}  \frac{1}{[-\I-\I v_c(t)]}
 +\frac{2\I\bar \omega_{-2}(t)}{[1- \bar v_c^2(t)]}  \frac{1}{[\I+\I\bar  v_c(t)]}+ q(t).
\end{equation}
%

%
%
Equation \e{CLM} is not affected by any choice of $q(t)$, thus it is matter of convenience to use either $q(t)$ or $\tilde q(t)$. While $\tilde q(t)$ might be fixed by a specific physical application of  \e{mainEquationa0.5sigma0_periodic}, the easiest choice is to set $\tilde q(t)\equiv 0$ with  \e{qtdef} then providing the explicit form for $q(t)$.

Equations \e{Hilbert_rep},   \e{complexconjugationomega} \e{omegaa0.5sigma0_periodic_lower},
\e{uintX}-\e{uap1p2sigma0}
result in terms proportional of different powers of $[X-\I v_c(t)]$ in  \e{mainEquationa0.5sigma0_periodic}. The most singular term $\propto [X-\I v_c(t)]^{-3}$  vanishes provided
\begin{equation}\label{vcODEa1p2sigma0}
\frac{dv_c}{dt} =-\frac{\I \left(1- v_c^2\right) \bar \omega_{-2}}{2 \left(  1-\bar v_c^2\right) ( v_c+\bar  v_c)}+\omega_{-2}\frac{\I v_c}{2(1-v_c^2)}-\frac{\I q(t)}{4}(1-v_c^2).
\end{equation}

The next order singular term  $\propto [X-\I v_c(t)]^{-2}$  vanishes provided
\begin{align}\label{omm2ODEa1p2sigma0}
\frac{d\omega_{-2}}{dt} =\frac{\I \left(1-2 v_c^2\right)  \omega_{-2}^2}{\left(1- v_c^2\right)^2}+[- \nu+\I q(t)v_c\,+\I \omega_{av}(t)]\omega_{-2}\nonumber \\+\I|\omega_{-2}|^2\left [\frac{ 1 }{( v_c+\bar  v_c)^2}+\frac{1}{  1-\bar v_c^2}\right ].
\end{align}
The function $\omega_{av}(t)$ is determined at order  $\propto [X-\I v_c(t)]^{0}$ as
${d \omega_{av}}/{dt}=-\nu  \omega_{av}$, i.e.
\begin{equation}\label{ca1p2sigma0}
\omega_{av}(t)=\omega_{av}(0)e^{-\nu t}.
\end{equation}
In summary,  \e{vcODEa1p2sigma0}-\e{ca1p2sigma0}
together with initial conditions $\omega_{-2}(0)$, $v_c(0)\in\mathbb{C}$ and $\omega_{av}(0)\in\mathbb{R}$, define the initial value problem to determine $\omega_{-2}(t)$ and $v_c(t)$. Equations \e{vcODEa1p2sigma0}-\e{ca1p2sigma0} 
 define a  system of two complex-valued ODEs, or, equivalently, a system of four real-valued ODEs.
That system includes an arbitrary smooth function $q(t)$ which can be fixed only if we use a physical interpretation of the velocity $u(t)$, e.g., one can set $\tilde q(t)\equiv 0$ together with \e{qtdef}.

\subsubsection{Implicit solution of \e{vcODEa1p2sigma0}-\e{ca1p2sigma0}
} \label{sec:implicit_formula}
A particular reduction of the ODE system \e{vcODEa1p2sigma0}-\e{ca1p2sigma0}  occurs  for
purely imaginary $\omega_{-2}(t)$ and purely real $v_c(t)$, i.e.,
\begin{equation}\label{ODEreductiona1p2sigmao0}
\omega_{-2}(t)=\I \omega_{-2,i}(t), \quad Im[\omega_{-2,i}(t)]= 0, \quad Im[v_c(t)]=0.
\end{equation}
Then  \e{vcODEa1p2sigma0}-\e{omm2ODEa1p2sigma0} imply $\omega_{av}(t)\equiv q(t)\equiv 0$.  Equations  \e{vcODEa1p2sigma0}-\e{ODEreductiona1p2sigmao0} result in the real-valued ODE system
\begin{equation}\label{a0.5sigma0_periodic_eqns} 
\frac{d \omega_{-2,i}}{dt}=  \omega_{-2,i}^2\frac{(1-2v_c^2+5v_c^4)}{4v_c^2(1-v_c^2)^2} - \nu  \omega_{-2,i},  \qquad
\frac{dv_c}{dt} =-\omega_{-2,i}\frac{1+v_c^2}{4v_c(1-v_c^2)},
\end{equation}

%
 %

After rewriting the first equation in \e{a0.5sigma0_periodic_eqns} as
\begin{align}\label{logchange}
\frac{d(\log(\omega_{-2,i})+ \nu t)}{dt}\nonumber
= \omega_{-2,i} \frac{(1-2v_c^2+5v_c^4)}{4v_c^2(1-v_c^2)^2}
\end{align}
and dividing it by the second equation to exclude  $\omega_{-2,i}$ we obtain that
\begin{align}
\frac{d(\log(\omega_{-2,i})+ \nu t)}{dt}= - \frac{(1-2v_c^2+5v_c^4)}{v_c(1-v_c^2)(1+v_c^2)}  \frac{dv_c}{dt} \nonumber \\
 \qquad \qquad= \left(\frac{1}{v_c-1}+\frac{1}{v_c+1}-\frac{1}{v_c}+\frac{4v_c}{1+v_c^2}  \right)\frac{dv_c}{dt}, \nonumber
\end{align}
which has a solution
\begin{equation}\label{a0.5sigma0_periodic_solution1} 
\omega_{-2,i}(t)=  \omega_{-2,i}(0)e^{- \nu t} \frac{v_c(0)}{v_c(t)}  \left(\frac{v_c^2(t)-1}{v_c^2(0)-1}\right) \left(\frac{v_c^2(t)+1}{v_c^2(0)+1}\right)^2.
\end{equation}
Substituting in \e{a0.5sigma0_periodic_eqns} we find
that\begin{equation}\label{a0.5sigma0_periodic_eqns_Vc} 
\frac{dv_c}{dt}=  \frac{\omega_{-2,i}(0)v_c(0)}{(v_c^2(0) - 1)(v_c^2(0)+1)^2} e^{- \nu t} \frac{(v_c^2(t)+1)^3}{4v_c^2(t)},
\end{equation}
and solving it for $v_c(t)$ we obtain an implicit solution
\begin{align}\label{a0.5sigma0_periodic_solution2} 
&\frac{v_c(t)(v_c^2(t) - 1)}{(v_c^2(t)+1)^2} + \arctan[v_c(t)]\nonumber \\= &\qquad\frac{v_c(0)(v_c^2(0) - 1)}{(v_c^2(0)+1)^2}\left( 1 + \frac{2\omega_{-2,i}(0)(1-e^{- \nu t})}{\nu (v_c^2(0) - 1)^2} \right) + \arctan[v_c(0)].
\end{align}
If $\nu=0$, then the expression $(1 - e^{-\nu t})/\nu$ is replaced by $t$.
Here and below without loss of generality we assume a principle branch of $\arctan$.

\subsubsection{Collapse vs. global existence} \label{sec:implicit_start}

We now analyze  \e{a0.5sigma0_periodic_solution2} to find conditions for collapse vs. global existence of the solution \e{omegaa0.5sigma0_periodic_lower}. There are two possibilities for collapse at the collapse time $t=t_{c}:$ either (A) $v_c(t)\to 0^{+},$ or (B)
$v_c(t)\to +\infty  $ as $t\to t_c^-$ (for the second case a  singularity in the complex $x$-plane  approaches  the real line $Re[x]=x$ at $x=\pm \pi $ since $|\tan(\frac{x}{2})|_{x\to\pm \pi}\to \infty$). 
In either case, $|\omega_{-2,i}(t)|\to\infty$ diverges as $t\to t_c^-$  which immediately follows from  \e{a0.5sigma0_periodic_solution1}. \cblue{Consequently, as shown in Theorem \ref{thm:nec_suff} below, both $\| \omega(\cdot, t) \|_{L^2}$ and $\| \omega(\cdot, t) \|_{B_0}$ also diverge at the critical time.} It is further  implied by   \e{a0.5sigma0_periodic_solution1} that  $|\omega_{-2,i}(t)|$ diverges if and only if either (A) or (B) are satisfied. Thus it is sufficient to consider only cases (A) and (B) to fully characterize blowup  of solutions to equations \e{a0.5sigma0_periodic_solution1}  and \e{a0.5sigma0_periodic_solution2}; we describe these cases as
 {\it\  blowup A} and {\it\ blowup B}, respectively.


Equation \e{omegaa0.5sigma0_periodic_lower}  requires that  $v_c(0)>0.$   Thus we look for  blowup in the parametric half-plane $(v_c(0), \omega_{-2,i}(0))$  of all possible initial conditions with  $v_c(0)>0$. That half-plane is divided into four sectors as we now describe.
\begin{equation}\label{condition1}
\text{Sector (1):} \qquad   \omega_{-2,i}(0)>0, \qquad 0<v_c(0)<1.
\end{equation}
 Then  \e{a0.5sigma0_periodic_eqns_Vc} implies that  $\frac{dv_c}{dt}<0$ for any $t\ge 0$ at which a solution exists, i.e. $v_c(t)$ monotonically decreases. Thus only the type A blowup  is allowed in that case. There are only three possibilities due to that monotonicity: (1.a) $v_c(t_c)=0$ for some $0<t_c< \infty;$ (1.b) $0<\lim \limits_{t\to+\infty}v_c(t)<1;$      (1.c)$\lim \limits_{t\to+\infty}v_c(t)=0.$  Case (1.a) is by definition type A  blowup  of the solution \e{omegaa0.5sigma0_periodic_lower}.  Cases (1.b) and (1.c) imply global existence of the solution \e{omegaa0.5sigma0_periodic_lower}.
\begin{equation}\nonumber
\text{Sector (2):} \qquad   \omega_{-2,i}(0)>0, \qquad v_c(0)>1.
\end{equation}
Then  \e{a0.5sigma0_periodic_eqns_Vc} implies that  $\frac{dv_c}{dt}>0$ for any $t\ge 0$ at which a solution exists, i.e., $v_c(t)$ monotonically increases. Thus only  type B  blowup is allowed in that case. There are only three possibilities due to that monotonicity: (2.a) $v_c(t_c)=\infty$ for some $0<t_c< \infty;$ (2.b) $ 1<\lim \limits_{t\to+\infty}v_c(t)<\infty;$      (2.c)$\lim \limits_{t\to+\infty}v_c(t)=\infty.$   Case (2.a) is type B blowup of the solution \e{omegaa0.5sigma0_periodic_lower}.  The cases (2.b) and (2.c) imply global existence of the solution \e{omegaa0.5sigma0_periodic_lower}.
\begin{equation}\nonumber
\text{Sector (3):} \qquad  \omega_{-2,i}(0)<0, \qquad 0<v_c(0)<1.
\end{equation}
  Then  \e{a0.5sigma0_periodic_eqns_Vc} implies that  $\frac{dv_c}{dt}>0$ for any $t\ge 0$ at which a solution exists, i.e., $v_c(t)$ monotonically increases. Thus only the type B blowup is allowed in that case. There are only three possibilities due to that monotonicity: (3.a) $v_c(t_1)=1$ for some $0<t_1<\infty;$ (3.b) $0<\lim \limits_{t\to+\infty}v_c(t)<1; $      (3.c)$\lim \limits_{t\to+\infty}v_c(t)=1.$  The cases (b) and (c) imply  global existence of the solution \e{omegaa0.5sigma0_periodic_lower}. The case (3.a) implies that the solution $v_c(t)$ extends into $v_c(t)>1$ for $t>t_1.$ To prove that we use  \e{a0.5sigma0_periodic_eqns_Vc} at $t=t_1$ giving  $\left .\frac{dv_c}{dt}\right |_{t=t_1}=  \frac{\omega_{-2,i}(0)v_c(0)}{(v_c^2(0) - 1)(v_c^2(0)+1)^2} e^{- \nu t_1}>0,$ i.e., $v_c(t)$ crosses $t=t_1$ with  finite positive speed. Then one can use an arbitrary small positive constant $\varepsilon$ to define new initial conditions
$v_c(t_1+\varepsilon)=1+\varepsilon\left .\frac{dv_c}{dt}\right |_{t=t_1}+O(\varepsilon^2)>1$ and $\omega_{-2,i}(t_1+\varepsilon)>0$ (determined by  \e{a0.5sigma0_periodic_solution1}) at $t=t_1+\varepsilon$ thus reproducing the case of Sector (2) with new initial conditions.
\begin{equation}\label{condition4}
\text{Sector (4):} \qquad  \omega_{-2,i}(0)<0, \qquad  1<v_c(0).
\end{equation}
Then  \e{a0.5sigma0_periodic_eqns_Vc} implies that  $\frac{dv_c}{dt}<0$ for any $t\ge 0$ at which a solution exists, i.e. $v_c(t)$ monotonically decreases. Thus only the type A blowup is allowed in that case. There are only three possibilities due to that monotonicity: (4.a) $v_c(t_1)=1$ for some $0<t_1<\infty;$ (4.b) $1<\lim \limits_{t\to+\infty}v_c(t);$      (4.c)$\lim \limits_{t\to+\infty}v_c(t)=1.$  Cases (4.b) and (4.c) imply  global existence of the solution \e{omegaa0.5sigma0_periodic_lower}.  Case (4.a) implies that the solution $v_c(t)$ extends into $v_c(t)<1$ for $t>t_1.$ To prove that we use  \e{a0.5sigma0_periodic_eqns_Vc} at $t=t_1$ giving  $\left .\frac{dv_c}{dt}\right |_{t=t_1}=  \frac{\omega_{-2,i}(0)v_c(0)}{(v_c^2(0) - 1)(v_c^2(0)+1)^2} e^{- \nu t_1}<0,$ i.e. $v_c(t)$ crosses $t=t_1$ with finite negative speed. Then one can use an arbitrary small positive constant $\varepsilon$ to define new initial conditions
$v_c(t_1+\varepsilon)=1+\varepsilon\left .\frac{dv_c}{dt}\right |_{t=t_1}+O(\varepsilon^2)<1$ and $\omega_{-2,i}(t_1+\varepsilon)<0$ (determined by  \e{a0.5sigma0_periodic_solution1}) at $t=t_1+\varepsilon$ thus reproducing the case of Sector (1) with new initial conditions.

To complete our analysis in Sectors (1)-(4) above we have to determine exact boundaries between blowup  and global existence in the parametric half-plane $(v_c(0), \omega_{-2,i}(0))$   using the implicit solution \e{a0.5sigma0_periodic_solution2}.

\subsubsection{
Necessary conditions for  the type A blowup}

It will be helpful to consider the function
\begin{equation}\label{f-firstTime}
f_0(x):=(x^{2} - 1)^2+\frac{(x^{2}+1)^{2}}{x}(x^{2} - 1)\arctan(x),
\end{equation}
for $x\geq0.$  We define $f_{0}(0)=0,$ as this is the limit from the interior of the domain.
Furthermore, it is clear that $f_{0}(1)=0.$  A plot of $f_0$ is shown in Fig. \ref{fig:f0plot} with %
\begin{equation}\label{fproperties}
-0.73040598\ldots<f_0(v_c(0))<0 \ \text{for} \ 0<v_c(0)<1 \ \text{and} \  f_0(v_c(0))>0 \ \text{for} \ v_c(0)>1.
\end{equation}
\begin{figure}[h]
\centering
\includegraphics[width=4in]{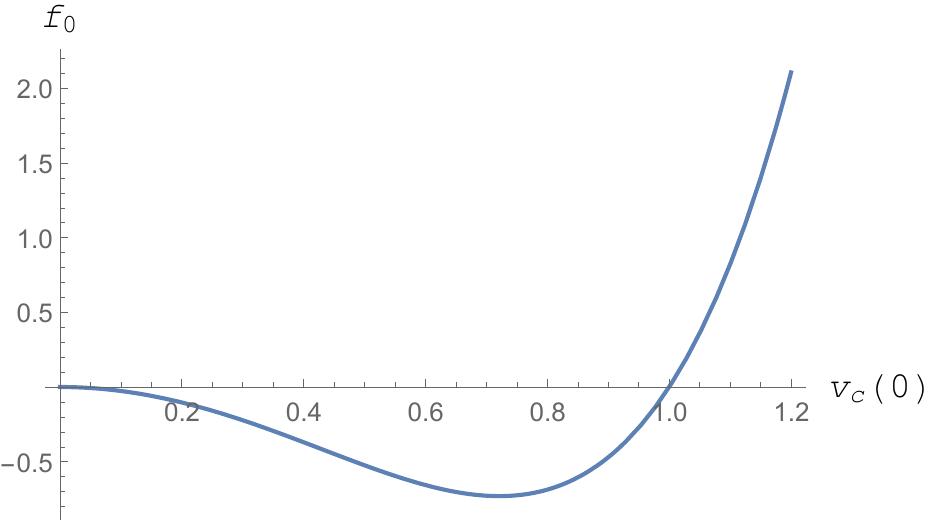}
\caption{Plot of  $f_0(v_c(0))$ from  \e{f-firstTime}
with $x=v_{c}(0)$.
 }
\label{fig:f0plot}
\end{figure}
A minimum of $f_0$ is located at %
\begin{equation}\label{vc0min}
v_c(0)=v_{0,min}=0.7211367\ldots
\end{equation}
as obtained from a solution of the transcendental equation $f_0'(v_{0,min})=0$ with %
\begin{equation}\label{f0mindepth}
f_0(v_{0,min})=-0.73040598\ldots
\end{equation}

\begin{lemma}\label{f0properties} The only zeros of  $f_{0}(x) $ for $x\ge 0$ are  $x=0$ and $x=1$.   Also $f_{0}(x)<0$ if and only if $x\in(0,1)$ and $f_{0}(x)>0$ if and only if $x\in(1,\infty)$.
\end{lemma}
\begin{proof}
To begin we consider $x\in(0,1),$ and on this interval we expect $f_{0}(x)<0.$  To establish this is
equivalent to establishing
\begin{equation}\label{shownForCertainX}
(x^{2}-1)^{2}<(1-x^{2})\frac{(x^{2}+1)^{2}}{x}\arctan(x),
\end{equation}
which in turn is equivalent to establishing
\begin{equation}\label{shownForAllPositiveX}
\frac{(1-x^{2})x}{(x^{2}+1)^{2}}< \arctan(x).
\end{equation}
Notice that both sides are equal to zero at $x=0.$  It is therefore sufficient to establish
that the derivative of the left-hand side is less than the derivative of the right-hand side.
Differentiating, we see that it is enough to show
\begin{equation}\nonumber
\frac{x^{4}-6x^{2}+1}{(x^{2}+1)^{3}}<\frac{1}{x^{2}+1}.
\end{equation}
Rearranging factors from the denominators, we see that we are simply trying to demonstrate
\begin{equation}\nonumber
x^{4}-6x^{2}+1 < (x^{2}+1)^{2}=x^{4}+2x^{2}+1,
\end{equation}
which is clearly true.
This demonstrates that \eqref{shownForAllPositiveX} holds for all $x>0.$
This then implies \eqref{shownForCertainX} for all $x\in(0,1).$
The same argument, adjusting for the sign of $1-x^{2},$ shows that $f_{0}(x)>0$ for $x>1.$
We conclude $f_{0}(x)<0$ if and only if $x\in(0,1)$ and $f_{0}(x)>0$ if and only if $x\in(1,\infty).$
\end{proof}

\begin{lemma}\label{typeANecessary}
If type A blowup occurs, then it must either be the case that
\begin{equation}\label{vcm1}
 0<v_{c}(0)<1 \quad \mathrm{and}\quad
\omega_{-2,i}(0)>-\frac{\nu}{2}f_{0}(v_{c}(0)),
\end{equation}
or
\begin{equation}\label{vcp1}
v_c(0)>1 \quad \mathrm{and}\quad \omega_{-2,i}(0)<-\frac{\nu}{2}f_0(v_c(0)).
\end{equation}
\end{lemma}
\begin{proof}
 If we assume that the type A blowup occurs, then there exists  $t=t_c$ with  $0<t_c<\infty$,  such that $v_c(t_c)=0.$ Plugging these conditions into \e{a0.5sigma0_periodic_solution2} and solving for $e^{-\nu t_c}$  results in
\begin{equation}\label{tcsol1}
e^{-\nu t_c}-1=\frac{\nu}{2\omega_{-2,i}(0)}f_0(v_c(0)),
\end{equation}
%
%
where $f_0$ is defined in \e{f-firstTime}. The blowup  
condition $0<t_c<\infty$  implies that the left-hand side of \e{tcsol1} must be between $-1$ and zero because $\nu>0,$ i.e.,
\begin{equation}\label{negativecondition}
-1<\frac{\nu}{2\omega_{-2,i}(0)}f_0(v_c(0))<0.
\end{equation}
%
%
%

Thus, to look for a solution which blows up, one must first satisfy the second inequality of \e{negativecondition}    by choosing the opposite sign of $\omega_{-2,i}(0)$ compared with the sign of $f_0(v_c(0))$  to ensure that the right-hand side of \e{tcsol1} is negative.
Such a choice of sign is always possible   because the solution ansatz \e{omegaa0.5sigma0_periodic_lower}
allows any sign of  $\omega_{-2,i}(0)$ (while restricting to $v_c(0)>0$). The first inequality of  \e{negativecondition}  must be also satisfied for blowup.

%
Assume that $0<v_c(0)<1. $ Then  $\omega_{-2,i}(0)>0$ by the condition \e{negativecondition} together with  \e{fproperties}. It implies the solution lies in Sector (1) given by
 \e{condition1}.  The condition \e{negativecondition}   together with \e{fproperties}-\e{f0mindepth}  gives the necessary conditions for
the type A blowup,  specified in \eqref{vcm1}.
If instead $v_c(0)>1,$ then the condition \e{negativecondition} together with  \e{fproperties} results in    $\omega_{-2,i}(0)<0.$   This implies the solution is in Sector (4)
given by \e{condition4}. The condition \e{negativecondition}   together with \e{fproperties}-\e{f0mindepth}   gives the necessary conditions for
the type A blowup, specified  in \eqref{vcp1}.
\end{proof}

The yellow area in Fig. \ref{fig:figblowup1p2sigma0} shows the region determined by \e{vcm1}.
\begin{figure}[h]
\centering
\includegraphics[width=6.in]{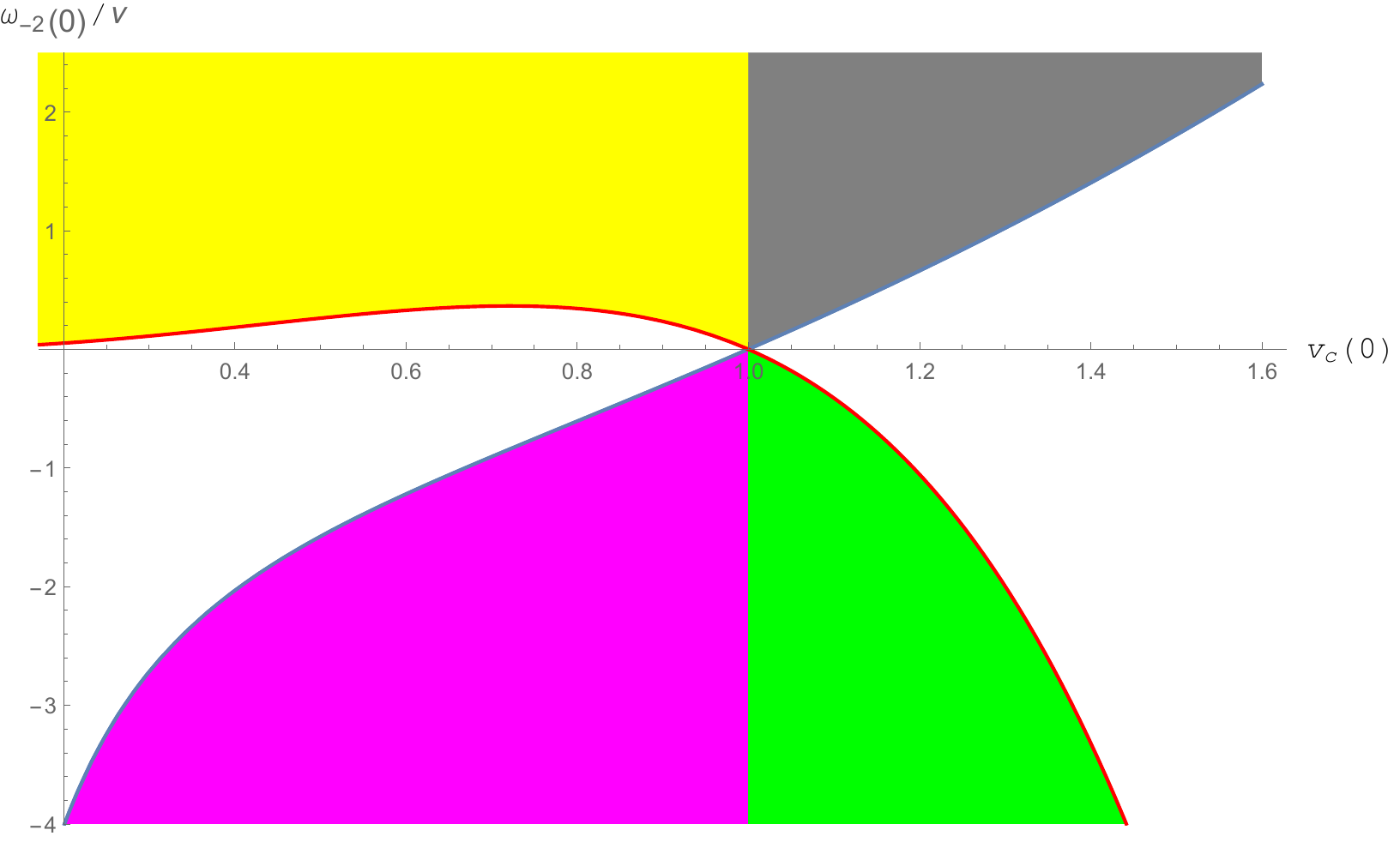}
\caption{A summary  of
behavior for the implicit solution (\ref{a0.5sigma0_periodic_solution2})
on the half-plane of $v_c(0)>0$ vs. $\omega_{-2,i}(0)/\nu$.  Colored areas designate blowup regions and non-colored regions correspond to global existence of solutions. In particular, the yellow region corresponds to  \e{vcm1},  the green region to  \e{vcp1}, the magenta region to  \e{vcm1infty0} and the gray region to    \e{vcp1infty0}. Blue and red lines separate these regions and correspond to infinite time blowup. }
\label{fig:figblowup1p2sigma0}
\end{figure}
 We define
 \begin{equation}\nonumber
 \omega_{-2,i,min,0}(v_c(0)):=-\frac{\nu}{2}f_0(v_c(0)), \ \cblue{ 0 < v_c(0) < 1},
\end{equation}
and note the largest value of $\omega_{-2,min}(v_c(0))$ is achieved at $v_{0,min}=0.7211367\ldots $ so that %
\begin{equation}\nonumber
\max\limits_{0<v_c(0)<1}\omega_{-2,i,min,0}(v_c(0))=-\frac{\nu}{2}f_0(v_{0,min})=\nu \cdot \,0.36520299\ldots
\end{equation}
as follows from \e{vc0min} and \e{f0mindepth}.  In the case $v_{c}(0)>1,$ we may similarly define
\begin{equation}\nonumber
\omega_{-2,i,max,0}(v_c(0)):=-\frac{\nu}{2}f_0(v_c(0))<0,
\end{equation}
and  plot the region determined by \e{vcp1}, i.e., $\omega_{-2,i}(0)<\omega_{-2,i,max,0}(v_c(0))<0.$  We do so with the green area in Fig. \ref{fig:figblowup1p2sigma0}.

\cblue{Considering the alternative conditions
compared with} Lemma \ref{typeANecessary}, if we instead have either
\begin{equation}\label{vcm1global}
  \cblue{ \omega_{-2,i}(0)\le\omega_{-2,i,min,0}(v_c(0)),\ 0<v_c(0)<1,}
\end{equation}
or
\begin{equation}\label{vcp1global}
\cblue{ \omega_{-2,i}(0) \geq \omega_{-2,i,max,0}(v_c(0))\quad \mathrm{and}\quad  v_c(0)>1,}
\end{equation}
then type A blowup does not occur  since there exists no finite $t_c$ satisfying (\ref{tcsol1}).
 %
%
%

The above results can be understood as follows.
\cblue{In the special case where  either $\omega_{-2,i}(0) > 0,$ $0<v_c(0)<1$  or  $\omega_{-2,i}(0) < 0, \ v_c(0)>1$,
we have that  $\frac{dv_c}{dt}<0$, which implies that  $v_c(t)$ decreases monotonically.}
If  \e{vcm1global} and \e{vcp1global} are additionally satisfied, the dissipative decay factor $e^{-\nu t}$ on the right-hand side  of \e{a0.5sigma0_periodic_eqns_Vc}  dominates over the growth factor $1/v_c^2(t)$  and prevents the solution $v_c(t)$ from reaching the origin $v_c=0$ in finite time,  thus ensuring  global existence of the solution  \e{omegaa0.5sigma0_periodic_lower}.   In the complementary cases  \e{vcm1} and \e{vcp1}, the growth factor $1/v_c^2(t)$ instead dominates the decay factor $e^{-\nu t}$, so that  the speed $\frac{dv_c}{dt}\to -\infty$ as $v_c(t)\to 0$.
That  speed guarantees that  collapse occurs, i.e., $v_c(t) \rightarrow 0$ in finite time with $t_c$ given by  \e{tcsol1}.
 A Taylor series expansion of  \e{a0.5sigma0_periodic_solution2} immediately gives that  %
\begin{equation} \label{vczerolimit}
v_c(t)\propto(t_c-t)^{1/3}+O((t_c-t)^{4/3})
\end{equation}
quantifying this ``fall into the origin'' at infinite speed.

The boundary cases are
\begin{equation}\label{vcm1boundary}
 0<v_{c}(0)<1 \quad \mathrm{and}\quad
\omega_{-2,i}(0)=-\frac{\nu}{2}f_{0}(v_{c}(0)),
\end{equation}
or
\begin{equation}\label{vcp1boundary}
v_c(0)>1 \quad \mathrm{and}\quad \omega_{-2,i}(0)=-\frac{\nu}{2}f_0(v_c(0)).
\end{equation}In both these cases  we expand  \e{a0.5sigma0_periodic_solution2} into Taylor series  at $v_c(t)=0$ and solve the resulting equation for $v_c(t)$ using either \e{vcm1boundary} or \e{vcp1boundary} which gives at the leading order for $t\to +\infty$  that
\begin{equation}\label{vccriticalA}
v_c(t)=\frac{3^{1/3}}{2}e^{-\nu t/3}\left (\arctan[v_c(0)]+v_c(0)\frac{v_c(0)^2-1}{[1+v_c(0)^2]^2} \right )^{1/3}+O(e^{-\nu t})
\end{equation}
which holds both for $0<v_c(0)<1$\ and $v_c(0)>1.$ Here the principal root of the power $1/3$ is chosen which  ensures that $v_c(t)>0$  because the corresponding radicand $f_a:=\arctan[v_c(0)]+v_c(0)\frac{v_c(0)^2-1}{[1+v_c(0)^2]^2}> 0$ for $v_c(0)>0.$ The positivity of $f_a$ follows  by calculating a derivative   $\frac{d f_a}{d v_c(0)}=\frac{8v_c(0)^2}{[1+v_c(0)^2]^3}$   which is positive for $v_c(0)>0$ together with $f_a=0$ at $v_c(0)=0$.

Thus solutions both  for \e{vcm1boundary} and \e{vcp1boundary} exist globally for $t>0$ while $v_c(0)=0$ is reached at infinite time $t=+\infty$.  We therefore  refer to the boundary cases as {\it infinite time blowup.}

\subsubsection{
 Necessary conditions for  the type B blowup}

As in the case of type A blowup, before proving a necessary condition for type B blowup we need to first study
an auxiliary function.
We consider the function $f_{\infty},$ defined by
\begin{equation}\label{fdefsigninfty}
f_\infty(x):=(x^{2} - 1)^2+\frac{(x^{2}+1)^{2}}{x}(x^{2} - 1)(\arctan(x)-\pi/2), \ x>0.
\end{equation}
A plot of $f_\infty$ is shown in Fig. \ref{fig:finfplot}. 
%
\begin{figure}[h]
\centering
\includegraphics[width=4in]{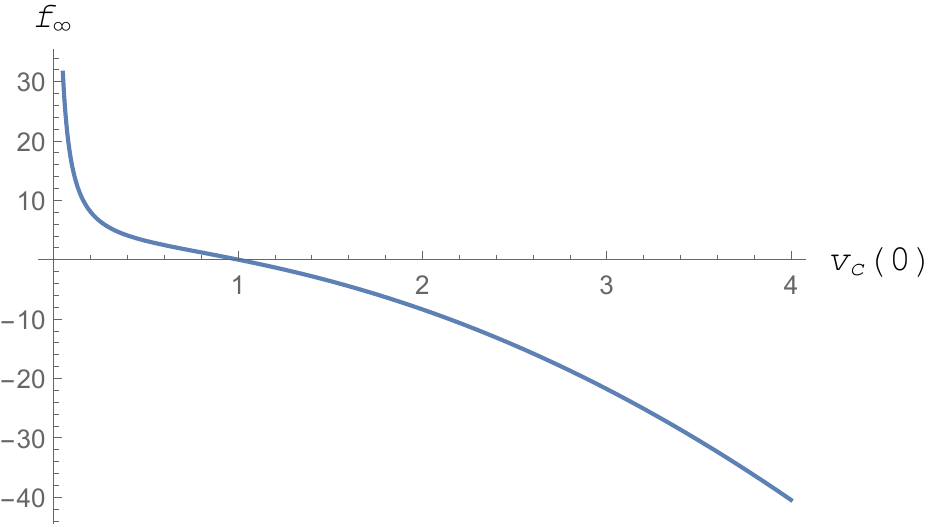}
\caption{A plot of $f_\infty(v_c(0))$ from \e{fdefsigninfty} with $x=v_c(0)$. } \label{fig:finfplot}
\end{figure}

\begin{lemma}\label{finfinityproperties} The auxiliary function
 $f_{\infty}(x)=0$ only for $x=1,$
$f_{\infty}(x)>0$ for $x\in(0,1),$ and $f_{\infty}(x)<0$ for all $x>1$.
\end{lemma}
\begin{proof}
We factor out $x^{2}-1,$ and write $f_{\infty}(x)=(x^{2}-1)g(x).$  We then wish to show that $g(x)<0$ for all $x>0.$
This means that we wish to show
\begin{equation}\nonumber
x^{2}-1 < \frac{(x^{2}+1)^{2}}{x}\left(\frac{\pi}{2}-\arctan(x)\right),
\end{equation}
for all $x>0.$  This is the same as showing
\begin{equation}\nonumber
\frac{(x^{2}-1)x}{(x^{2}+1)^{2}}<\frac{\pi}{2}-\arctan(x),
\end{equation}
for all $x>0.$  We call the function on the left-hand side $h_{1}(x),$ and we call the function on the right-hand side $h_{2}(x).$
The derivatives of $h_{1}$ and $h_{2}$ are
\begin{equation}\nonumber
h_{1}'(x)= -\frac{x^{4}-6x^{2}+1}{(x^{2}+1)^{3}},\qquad h_{2}'(x)=-\frac{1}{x^{2}+1}=-\frac{x^{4}+2x^{2}+1}{(x^{2}+1)^{3}}.
\end{equation}
Notice that $h_{1}'(x)>h_{2}'(x),$ for all $x>0.$  Furthermore, $h_{1}(0)=0$ and $h_{2}(0)=\frac{\pi}{2}.$
So, we see that $h_{2}$ begins larger than $h_{1}$ but $h_{1}$ always grows faster than $h_{2}.$  This implies that if the graphs
of $h_{1}$ and $h_{2}$ ever cross, they may only cross one time.  To determine whether this ever happens, it is enough to determine which
function is larger for large $x$ (i.e.,  as $x$ goes to infinity).  We will use the Laurent series of $h_{1}$ and $h_{2}$ for this purpose.
We may express $h_{1}$ and $h_{2}$ with their Laurent series as
\begin{equation}\nonumber
h_{1}(x)=\sum_{k=1}^{\infty}\frac{(-1)^{k+1}(2k-1)}{x^{2k-1}}=\frac{1}{x}-\frac{3}{x^{3}}+\frac{5}{x^{5}}-\cdots,
\end{equation}
\begin{equation}\nonumber
h_{2}(x)=\sum_{k=1}^{\infty}\frac{(-1)^{k+1}}{(2k-1)x^{2k-1}}=\frac{1}{x}-\frac{1}{3x^{3}}+\frac{1}{5x^{5}}-\cdots.
\end{equation}
Notice that these are both alternating series, and notice that each are valid for all $x$ satisfying $|x|>1.$  Furthermore, the absolute value of the
terms in each series goes monotonically to zero; this is completely obvious for the series for $h_{2},$ but also holds for the series for
$h_{1}$ for sufficiently large $x$ (in fact, $x>2$ is sufficient).
We are interested in approximating each of $h_{1}$ and $h_{2}$ using the first two terms of their respective Laurent series.  We define the
errors from this approximation to be
\begin{equation}\nonumber
E_{1}(x)=h_{1}(x)-\frac{1}{x}+\frac{3}{x^{3}},\qquad
E_{2}(x)=h_{2}(x)-\frac{1}{x}+\frac{1}{3x^{3}}.
\end{equation}
We recall the error estimate for alternating series, in which the error can be bounded by the absolute value of the next term in the series.
Thus, we have the bounds
\begin{equation}\nonumber
|E_{1}(x)|\leq\frac{5}{x^{5}},\qquad |E_{2}(x)|\leq\frac{1}{5x^{5}}.
\end{equation}
We then may conclude
\begin{multline}\nonumber
h_{2}(x)-h_{1}(x)=\left(\frac{1}{x}-\frac{1}{3x^{3}}\right)+E_{1}(x)-\left(\frac{1}{x}-\frac{3}{x^{3}}\right) + E_{2}
\\
\geq\frac{8}{3x^{3}}-|E_{1}(x)|-|E_{2}(x)|\geq\frac{8}{3x^{3}}-\frac{5}{x^{5}}-\frac{1}{5x^{5}}>0,
\end{multline}
for all sufficiently large $x.$

We have shown that $g(x)<0$ for all $x>0.$ Therefore $f_{\infty}(x)=0$ only for $x=1,$ and
$f_{\infty}(x)>0$ for $x\in(0,1),$ and $f_{\infty}(x)<0$ for all $x>1.$
\end{proof}

\begin{lemma}\label{typeBNecessary}
If type B blowup occurs, then it must either be the case that
\begin{equation}\label{vcm1infty0}
0<v_c(0)<1,\quad\mathrm{and}\quad \omega_{-2,i}(0)<-\frac{\nu}{2}f_\infty(v_c(0)),
\end{equation}
or
\begin{equation}\label{vcp1infty0}
v_{c}(0)>1,\quad\mathrm{and}\quad \omega_{-2,i}(0)>-\frac{\nu}{2}f_\infty(v_c(0)).
\end{equation}
\end{lemma}
\begin{proof}

If we assume that the type B blowup occurs, then there exists  $t=t_c$ with   $0<t_c<\infty$,  such that   $v_c(t_c)=+\infty$, i.e., $v_c(t)\to+\infty$ as $t\to t_c^-$. That limit implies that $\arctan[v_c(t_c)]=\pi/2$. Plugging in these conditions to  \e{a0.5sigma0_periodic_solution2} and solving for $e^{-\nu t_c}$  results in

\begin{equation}\label{tcsol1infty}
e^{-\nu t_c}-1=\frac{\nu}{2\omega_{-2,i}(0)}f_\infty(v_c(0)).
\end{equation}
%

A   
blowup condition  $0<t_c<\infty$  implies that the left-hand side of \e{tcsol1infty} must be between $-1$ and zero because $\nu>0,$ i.e.,
\begin{equation}\label{negativeconditioninfty}
-1<\frac{\nu}{2\omega_{-2,i}(0)}f_\infty(v_c(0))<0.
\end{equation}
%
%
%

Thus, to look for a solution which blows up, one must first satisfy the second inequality of  \e{negativeconditioninfty}    by choosing the opposite sign of $\omega_{-2,i}(0)$ compared with the sign of $f_\infty(v_c(0)).$  Such a  sign choice is always possible   because the solution ansatz \e{omegaa0.5sigma0_periodic_lower}
allows any sign of  $\omega_{-2,i}(0)$ (while restricting to $v_c(0)>0$). The first inequality of  \e{negativeconditioninfty} must be also satisfied for blowup.
%

Assume that $0<v_c(0)<1. $ Then  $\omega_{-2,i}(0)<0$ by the properties we have established for $f_{\infty}$ in Lemma \ref{finfinityproperties}.
The necessary conditions \e{vcm1infty0} for
the type B blowup can then be stated as follows:  %
\begin{equation}\label{vcm1infty}
\omega_{-2,i}(0)<\omega_{-2,i,max,\infty}(v_c(0)):=-\frac{\nu}{2}f_\infty(v_c(0))<0, \ 0<v_c(0)<1.
\end{equation}
The magenta area in Fig. \ref{fig:figblowup1p2sigma0} shows the region determined by   \e{vcm1infty}.

Next assume that $v_c(0)>1. $ Then $\omega_{-2,i}(0)>0$ by the condition \e{negativeconditioninfty} together with  the properties we have established for $f_{\infty}$  in Lemma \ref{finfinityproperties}.
The necessary conditions \e{vcp1infty0} for the type B blowup in this case are then%
\begin{equation}\label{vcp1infty}
\omega_{-2,i}(0)>\omega_{-2,i,min,\infty}(v_c(0)):=-\frac{\nu}{2}f_\infty(v_c(0))>0, \ v_c(0)>1.
\end{equation}
The gray area in Fig. \ref{fig:figblowup1p2sigma0} shows the region determined by   \e{vcp1infty}.
\end{proof}

\cblue{Considering the alternative conditions
compared with} Lemma \ref{typeBNecessary}, we see that if either
\begin{equation}\label{vcm1globalinfty}
\cblue{\omega_{-2,i,max,\infty}(v_c(0)) \leq  \omega_{-2,i}(0), \ 0<v_c(0)<1,}
\end{equation}
or
\begin{equation}\label{vcplusglovinfty}
\cblue{ \omega_{-2,i}(0)\le\omega_{-2,i,min,\infty}(v_c(0)), \ v_c(0)>1,}
\end{equation}
then type B blowup does not occur, since (\ref{tcsol1infty}) is not satisfied for any finite $t_c$.

%

 A qualitative understanding of the above results
 is obtained  from  \e{a0.5sigma0_periodic_eqns_Vc}.  \cblue{In the special case where either  $\omega_{-2,i}(0)<0,$ $0<v_c(0)<1$  or  $\omega_{-2,i}(0)>0, \ v_c(0)>1,$
 we have that $\frac{dv_c}{dt}>0$, which implies that $v_c(t)$ increases monotonically.}
 If \e{vcm1globalinfty} and \e{vcplusglovinfty} are additionally satisfied, the dissipative decay factor $e^{-\nu t}$ on the right-hand side of \e{a0.5sigma0_periodic_eqns_Vc} dominates over the growth factor $v^4_c(t)$ and  prevents  $v_c(t)$ from reaching  $v_c=+\infty$ in finite time, thus ensuring a global existence of the solution  \e{omegaa0.5sigma0_periodic_lower}. For the complementary cases \e{vcm1infty0} and \e{vcp1infty0},  $v_c^4(t)$ instead dominates over  $e^{-\nu t}$ thus $v_c(t) \rightarrow +\infty$  with a diverging speed at a finite value of $t_c>0.$  To quantify this,  we expand   \e{a0.5sigma0_periodic_solution2} in powers of $v_c(t)\to \infty$ giving  that %
\begin{equation} \label{vczerolimitinfty}
v_c(t)\propto(t_c-t)^{-1/3}+O((t_c-t)^{2/3}).
\end{equation}
This shows $v_c(t)$  ``reaches infinity'' in finite time $t_c$, with infinite speed.

The boundary cases are
\begin{equation}\label{vcm1infty0boundary}
0<v_c(0)<1,\quad\mathrm{and}\quad \omega_{-2,i}(0)=-\frac{\nu}{2}f_\infty(v_c(0)),
\end{equation}
and
\begin{equation}\label{vcp1infty0boundary}
v_{c}(0)>1,\quad\mathrm{and}\quad \omega_{-2,i}(0)=-\frac{\nu}{2}f_\infty(v_c(0)).
\end{equation}
In both these cases  we expand  \e{a0.5sigma0_periodic_solution2} in powers of $v_c(t)\to \infty$   and solve the resulting equation for $v_c(t)$ using either \e{vcm1infty0boundary} or \e{vcp1infty0boundary} which gives at the leading order for $t\to +\infty$  that
\begin{equation}\label{vccriticalB}
v_c(t)={3^{-1/3}2}e^{\nu t/3}\left (\frac{\pi}{2}-\arctan[v_c(0)]-v_c(0)\frac{v_c(0)^2-1}{[1+v_c(0)^2]^2} \right )^{-1/3}+O(e^{-\nu t/3})
\end{equation}
both for $0<v_c(0)<1$\ and $v_c(0)>1.$ Here the principal root of the power $1/3$ is chosen which  ensures that $v_c(t)>0$  because the corresponding radicand $f_b:=\frac{\pi}{2}-\arctan[v_c(0)]-v_c(0)\frac{v_c(0)^2-1}{[1+v_c(0)^2]^2}> 0$ for $v_c(0)>0.$
The positivity of $f_b$ is ensured  by calculating a derivative   $\frac{d f_b}{d v_c(0)}=-\frac{8v_c(0)^2}{[1+v_c(0)^2]^3}$   which is negative for $v_c(0)>0$ together with $f_b=\pi/2$ at $v_c(0)=0$ and $f_b=0$ at $v_c(0)=+\infty$.

Thus solutions  for data satisfying  \e{vcm1infty0boundary} or \e{vcp1infty0boundary} exist globally for $t>0$ while $v_c(0)\to+\infty$ as $t\to+\infty$, so these boundary cases exhibit {\it infinite time blowup.}
Fig. \ref{fig:figblowup1p2sigma0} summarizes  the behavior of the implicit solution \e{a0.5sigma0_periodic_solution2}
on the half-plane of $v_c(0)>0$ vs. $\omega_{-2,i}(0)/\nu$.

\subsubsection{Sufficient conditions for blowup} \label{sec:sufficient}

We now prove that \e{vcm1}, \e{vcp1},  \e{vcm1infty0} and \e{vcp1infty0}
provide not only necessary but also sufficient conditions for blowup.
We define auxiliary functions

\begin{equation}\label{f1def}
f_1(v_c(t)):=\frac{v_c(t)(v_c^2(t) - 1)}{(v_c^2(t)+1)^2} + \arctan[v_c(t)]
\end{equation}
representing the left-hand side of \e{a0.5sigma0_periodic_solution2} as well as
\begin{equation}\label{f2def}
f_2(v_c(0)):=\frac{v_c(0)}{(v_c^2(0) - 1)(v_c^2(0)+1)^2}.
\end{equation}
%
%

Differentiating  \e{f1def} with respect to $v_c(t)$  gives that
\begin{equation}\label{f1der}
\frac{d f_1(v_c(t))}{d v_c(t)}:=\frac{8v_c^2(t)}{(v_c^2(t)+1)^3}>0 \quad \text{for}\quad v_c(t)\ne 0,
\end{equation}
which ensures that $f_1(v_c(t))$ grows monotonically with $v_c(t).$ Because $f_1(0)=0$, we conclude that %
\begin{equation}\label{f1positivity}
f_1(v_c(t))> 0  \quad \text{for}\quad v_c(t)> 0 \ \text{while} \ f_1(0)=0.
\end{equation}
Using  the definitions of $f_{0},$ $f_{1},$ and $f_{2}$ (which are \e{f-firstTime}, \e{f1def}, and \e{f2def}, respectively) we can represent
\e{a0.5sigma0_periodic_solution2}
in the following compact form:
\begin{align}\label{a0.5sigma0_periodic_solution2compact} 
&f_1(v_c(t))=f_0(v_c(0))f_2(v_c(0))+  \frac{2\omega_{-2,i}(0)}{\nu } f_2(v_c(0))(1-e^{- \nu t}).
\end{align}


\begin{lemma}\label{lemmasufficient0} Either of  \e{vcm1} or \e{vcp1} is a sufficient  condition for  type A blowup.
\end{lemma}

\begin{proof}

We begin by assuming \eqref{vcm1}.
Assume for the sake of contradiction that $0<v_c(t)<\infty$ for all $t>0,$ i.e., a solution never reaches the origin in finite time if the initial conditions \e{vcm1} are satisfied.
By Lemma \ref{f0properties} we see that that $f_0(v_c(0))<0.$   From the definition of $f_{2}$ we immediately have $f_2(v_c(0))<0.$
Since $f_{0}$ and $f_{2}$ evaluated at $v_{c}$ are initially negative, we see that the right-hand side of \e{a0.5sigma0_periodic_solution2compact} is initially (at $t=0$) positive.
Taking the limit as $t\rightarrow\infty$ of this right-hand side, however, yields
\begin{equation}\nonumber
\left(f_0(v_c(0))+  \frac{2\omega_{-2,i}(0)}{\nu }\right)f_2(v_c(0))<0.
\end{equation}
Here, we have used the second condition in \eqref{vcm1} to see that the first factor is positive.
We conclude by a continuity of $v_c(t)$ that there must be a value of  $t$ for which $f_{1}(v_c(t))=0;$ then, \e{f1positivity} implies $v_{c}(t)=0.$ This is the desired contradiction.

Now assume that the conditions \e{vcp1} are satisfied instead of \e{vcm1}.
Lemma \ref{f0properties} then implies $f_0(v_c(0))>0,$ and we see from the definition \e{f2def} of $f_{2}$ that $f_2(v_c(0))>0.$
The right-hand side of \e{a0.5sigma0_periodic_solution2compact} is then initially positive. 
Again taking the limit of the right-hand side of  \e{a0.5sigma0_periodic_solution2compact} as $t\rightarrow\infty,$ we have
\begin{equation}\nonumber
\left(f_0(v_c(0))+  \frac{2\omega_{-2,i}(0)}{\nu }\right)f_2(v_c(0))<0.
\end{equation}
Here  we have used the second condition in \eqref{vcp1} to conclude that the first factor is negative.
As in the previous case, this again means that there exists a value $t>0$ for which $v_c(t)=0.$  This again is the desired contradiction.  This completes the proof.
\end{proof}

\begin{lemma}\label{sufficient1} Either of \e{vcm1infty0} or \e{vcp1infty0} is a sufficient  condition for type B blowup.
\end{lemma}
%
\begin{proof}

Using $f_{\infty}$ instead of $f_{0},$ we can rewrite
\e{a0.5sigma0_periodic_solution2compact}
as
\begin{align}\label{a0.5sigma0_periodic_solution2compact2} 
&f_1(v_c(t))=f_\infty(v_c(0))f_2(v_c(0))+  \frac{2\omega_{-2,i}(0)}{\nu } f_2(v_c(0))(1-e^{- \nu t})+\frac{\pi}{2}.
\end{align}

Assume \e{vcm1infty0} is satisfied. Then by Lemma \ref{finfinityproperties}, we have $f_{\infty}(v_{c}(0))>0$ and therefore $\omega_{-2,i}<0.$
Then \e{a0.5sigma0_periodic_eqns_Vc} implies that $v_c(t)$ grows monotonically with $t$.  This combined with  \e{f1der} implies that $f_{1}(v_{c}(t))$ grows monotonically with $t.$
The maximum of the function $f_{1}(v_c(t))$ for any value of $v_c(t)$ is given by $\pi/2,$ and this can
only be reached asymptotically as $v_c(t)\to +\infty$ as follows from the monotonic growth \e{f1der} and \e{f1def}. In other words, the maximum possible value of the left-hand side of \e{a0.5sigma0_periodic_solution2compact2} is given by  $\pi/2$.
Taking the limit as $t\rightarrow +\infty$ of the right-hand side of \eqref{a0.5sigma0_periodic_solution2compact2}, we get
\begin{equation}\label{tInfinityTypeBInequality}
\left(f_{\infty}(v_{c}(0))+\frac{2\omega_{-2,i}(0)}{\nu}\right)f_{2}(v_{c}(0))+\frac{\pi}{2}>\frac{\pi}{2}.
\end{equation}
(Here we have used that \eqref{vcm1infty0} together with \e{f2def} implies that each of the two factors in the first term on the left-hand side of this inequality are negative.)
Therefore, equation \eqref{a0.5sigma0_periodic_solution2compact2} cannot be valid in that limit $t\rightarrow +\infty$ and there must have been a finite value of $t_{c}$ for which $v_{c}(t)\rightarrow+\infty$ as $t\rightarrow t_{c}^{-}.$

Next, we assume  \e{vcp1infty0} is satisfied instead of \eqref{vcm1infty0}.  We then have $f_{\infty}(v_{c}(0))<0$ and $\omega_{-2,i}>0.$
Then \e{a0.5sigma0_periodic_eqns_Vc} again implies that $v_c(t)$ grows monotonically with $t.$  Taking the limit as $t\rightarrow\infty$ on the
right-hand side of \eqref{a0.5sigma0_periodic_solution2compact2}, again yields \eqref{tInfinityTypeBInequality} since now each of the two factors in the
first term on the left-hand side are positive.  This again implies the existence of a finite value of $t_{c}$ for which $v_{c}(t)\rightarrow+\infty$ as $t\rightarrow t_{c}^{-}.$
This completes the proof.
\end{proof}

\cblue{Taylor's expansion of (\ref{a0.5sigma0_periodic_eqns}) for  $t$ near $t_c$ using (\ref{vczerolimit})  shows that $\omega_{-2,i} \sim (t_c-t)^{-1/3}$ in type A blowup and $\omega_{-2,i} \sim (t_c-t)^{-5/3}$ in type B blowup (using  (\ref{vczerolimitinfty})).   Consequently,  equation (\ref{a0.5sigma0_periodic_solution_norms_real}) implies that for both types of blowup, $\| \omega(\cdot,t) \|_{L^2} \sim (t_c-t)^{-5/6}$ and $\| \omega(\cdot,t) \|_{B_0} \sim (t_c-t)^{-1}$  as $t \rightarrow t_c^-$. The corresponding similarity parameters (cf. (\ref{self-similar1})) are $\alpha=1/3$ and $\beta=1$.}

Combining Lemmas \ref{typeANecessary}, \ref{typeBNecessary}, \ref{lemmasufficient0} and \ref{sufficient1} we prove the following:

\begin{theorem}{ (Necessary and sufficient condition of collapse).} The conditions  \e{vcm1} or \e{vcp1} are necessary and sufficient  for blowup of type A.  The conditions  \e{vcm1infty} or    \e{vcp1infty} are necessary and sufficient  for blowup of type B. \cblue{ In both blowup types,  $\| \omega(\cdot,t) \|_{L^2}$ and $\| \omega(\cdot,t) \|_{B_0}$ diverge as $t \rightarrow t_c^-$.}   \cblue{If the alternative conditions \e{vcm1global} and \e{vcm1globalinfty} or  \e{vcp1global}  and \e{vcplusglovinfty}  are  satisfied then there is neither type A nor type B blowup and  the solution exists globally for all $t>0$. \label{thm:nec_suff}}
\end{theorem}
%
\noindent{\bf Remark 1.}   Solutions of   \e{a0.5sigma0_periodic_eqns} for initial data at the boundary of the global existence domain, shown by red and blue curves in Fig. \ref{fig:figblowup1p2sigma0},  blow up in infinite time.
The boundary of the global existence domain is also revealed by  numerical solution of (\ref{a0.5sigma0_periodic_eqns}), and gives excellent agreement with the theoretical boundary  determined by  (\ref{vcm1boundary}), (\ref{vcp1boundary}), (\ref{vcm1infty0boundary}), (\ref{vcp1infty0boundary}), shown by red-dashed curves
in Figure \ref{fig:coffee}, Appendix \ref{sec:AppendixANum}.
\vspace{.04in}
\\ \noindent{\bf Remark 2.} The particular initial condition $v_c(0)=1$ has a singularity in the implicit solution  \e{a0.5sigma0_periodic_solution2}. That singularity can be understood as a limit $v_c(0)\to 1$ while also taking the limit $\omega_{-2,i}(0)\to 0$. Depending on how these two limits are taken, the solution can approach the point $(1,0)$ in the plane $(v_c(0),\omega_{-2,i}(0))$   from the different sectors in Fig. \ref{fig:figblowup1p2sigma0}  that surround the point  $(1,0)$.
\vspace{.1in}
\\ \noindent{\bf Remark 3.} Theorem 3.7 shows that finite-time singularities for $a=1/2, \sigma=0$ can occur for arbitrarily small $L^2$ norm of $\omega(x,0)$. This follows from the asymptotic behavior $\omega_{-2,i} \sim (4/3) \nu  v_c^2$ of the boundary of the yellow blow-up region in Figure \ref{fig:figblowup1p2sigma0} when $0<v_c(0) \ll 1$.
Thus from (\ref{a0.5sigma0_periodic_solution_norms_real}),  $ \| \omega( \cdot,0) \|_{L^2}$ can be made arbitrarily small in this blow-up region by taking $v_c$ sufficiently small.  In contrast, (\ref{a0.5sigma0_periodic_solution_norms_real}) implies the Wiener norm blows up only for sufficiently large data.

\subsubsection{Limit of the real line }
\label{sec:limitreallinesigma0}

\cblue{Pole dynamics equations for the problem on the real line are obtained in the limit $v_c\to 0$, for which the influence of the $2\pi$ periodicity is negligible. In this case, Laurent series expansion of  the right-hand sides of   \e{a0.5sigma0_periodic_eqns} about $v_c=0,$ keeping only leading order terms $\propto v_c^{-2}$ and $ \propto v_c^0$ in the first ODE and  $\propto v_c^{-1}$ in the second ODE, results in the reduction
\begin{equation}\label{reductiona0.5sigma0}
\frac{d \omega_{-2,i}}{dt}=  \omega_{-2,i}^2\frac{1}{4v_c^2} - \nu  \omega_{-2,i},  \qquad
\frac{dv_c}{dt} =-\omega_{-2,i}\frac{1}{4v_c},
\end{equation}
which recovers  (35) and (36) of \cite{Lushnikov_Silantyev_Siegel} for $\nu=0$ with $\omega_{-2,i}$ replaced by $\tilde \omega_{-2}/4$ and $v_c$ replaced by $v_c/2$ to match the definitions of \cite{Lushnikov_Silantyev_Siegel}. This replacement can be understood by expanding the right hand side  of \e{omegaa0.5sigma0_periodic_lower} in $|x|\ll 1$ and keeping only the most singular term $(x-2\I v_c(t))^{-2}$, which gives
\begin{equation}\label{omegam2reduction}%
\omega_-(x,t)=\frac{4\omega_{-2}(t)}{(x-2\I v_c(t))^2}.
\end{equation}
This can be directly compared with (33) of  \cite{Lushnikov_Silantyev_Siegel} if we recall that $\omega_{-2}(t)=\I\omega_{-2,i}(t)$, as assumed in \e{a0.5sigma0_periodic_eqns}. In a similar way we expand \e{uap1p2sigma0} and \e{qtdef} to obtain
\begin{align}\label{uap1p2sigma0reduction}
u(x,t)=\frac{-4\I\omega_{-2}(t)}{x-2\I v_c(t)}+\frac{4\I\bar \omega_{-2}(t)}{x+2\I\bar v_c(t)} +  q(t),
\end{align}
which recovers (34)  of  \cite{Lushnikov_Silantyev_Siegel} for $q(t)\equiv 0.$
For $\nu\ne 0$, equations \e{reductiona0.5sigma0}-\e{uap1p2sigma0reduction} 
 generalize the solution (33)-(36) of \cite{Lushnikov_Silantyev_Siegel} to account for nonzero dissipation with $\sigma=0.$
}

\cblue{If
we similarly perform the  Laurent series expansion of the right-hand sides of the more general ODEs \e{vcODEa1p2sigma0} and \e{omm2ODEa1p2sigma0}, new exact solutions of \e{mainEquationa0.5sigma0_periodic} are obtained on the real line in the form \e{omegam2reduction},\e{uap1p2sigma0reduction}
with
\begin{equation}\label{vcODEa1p2sigma0reduction}
\frac{d\omega_{-2}}{dt} =[- \nu+\I \omega_{av}(t)]\omega_{-2}
+\frac{ \I|\omega_{-2}|^2 }{( v_c+\bar  v_c)^2}
, \quad \frac{dv_c}{dt} =-\frac{\I  \bar \omega_{-2}}{2  ( v_c+\bar  v_c)}-\frac{\I q(t)}{4}, \quad \omega_{av}(t)=\omega_{av}(0)e^{-\nu t},
\end{equation}
where we have also used \e{ca1p2sigma0}.
This can also be derived by  direct substitution of \e{omegam2reduction}-\e{vcODEa1p2sigma0reduction} into \e{mainEquationa0.5sigma0_periodic}.
}

\subsection{$a=1/2$ and $\sigma=1$} \label{sec:a0.5sigma1}

In the case $a=1/2$, $\sigma=1$, $\nu>0$,   (\ref{CLM}) becomes
\begin{equation} \label{mainEquationa0.5sigma1_periodic}
\tilde{\omega}_{t}=-\frac{1}{2}u \tilde{\omega}_x + \tilde{\omega}  \mathcal{H}(\tilde{\omega})-\nu \mathcal{H}(\tilde{\omega}_x), \quad u_x =\cal{H} \tilde{\omega}.
\end{equation}
Similar to Section  \ref{sec:a1p2sigma0}, we look for a solution to \e{mainEquationa0.5sigma1_periodic} of the form $\tilde{\omega}=\omega_+ + \omega_- + \omega_{av}$  with $\omega_-$ as in  (\ref{omegaa0.5sigma0_periodic_lower}), i.e., a sum of a double pole and a simple pole in $X=\tan(\frac{x}{2})$-space, and $\omega_+(x,t)=\overline{\omega_-(\overline{x},t)}$.


As before, the special choice (\ref{omegaa0.5sigma0_periodic_lower_1}) for  $\omega_{-1}$ in (\ref{omegaa0.5sigma0_periodic_lower})
ensures logarithmic terms  to cancel out.  Substituting \e{omegaa0.5sigma0_periodic_lower} into \e{mainEquationa0.5sigma1_periodic} and performing a similar analysis as  in Section \ref{sec:a1p2sigma0}, we obtain the following equations for  the complex quantities $\omega_{-2}(t)$ and $v_c(t)$:
\begin{equation}\label{vcODEa1p2sigma1}
\frac{dv_c}{dt} =\frac{\nu}{2}(1-v_c^2)-\frac{\I \left(1- v_c^2\right) \bar \omega_{-2}}{2 \left(  1-\bar v_c^2\right) ( v_c+\bar  v_c)}+\omega_{-2}\frac{\I v_c}{2(1-v_c^2)}-\frac{\I q(t)}{4}(1-v_c^2),
\end{equation}
and
\begin{align}\label{omm2ODEa1p2sigma1}
\frac{d\omega_{-2}}{dt} =-2\nu v_c\omega_{-2}+\frac{\I \left(1-2 v_c^2\right)  \omega_{-2}^2}{\left(1- v_c^2\right)^2}+[\I q(t)v_c\,+\I \omega_{av}(t)]\omega_{-2}\nonumber \\+\I|\omega_{-2}|^2\left [\frac{ 1 }{( v_c+\bar  v_c)^2}+\frac{1}{  1-\bar v_c^2}\right ],
\end{align}
where $ q(t)\in\mathbb{C}$ is an arbitrary smooth function of time and is related to the mean value of $u(x,t)$ that can be prescribed  as in  \e{qtdef}.%

The function $\omega_{av}(t)$ is determined at order  $\propto [X-\I v_c(t)]^{0}$ as
${d \omega_{av}}/{dt}=0$, i.e. it has the time-independent value
\begin{equation}\label{ca1p2sigma1}
\omega_{av}(t)= \omega_{av}(0).
\end{equation}

A particular reduction of such ODE system occurs for
purely imaginary $\omega_{-2}(t)$ and purely real $v_c(t)$ as
\begin{equation}\label{ODEreductiona1p2sigmao1}
\omega_{-2}(t)=\I \omega_{-2,i}(t), \quad Re[ \omega_{-2,i}(t)]= \omega_{-2,i}(t), \quad v_c(t)=Re[v_c(t)].
\end{equation}
Then  \e{vcODEa1p2sigma1}-\e{omm2ODEa1p2sigma1} imply $\omega_{av}(t)\equiv q(t)\equiv 0$.  Equations \e{vcODEa1p2sigma1}-\e{ODEreductiona1p2sigmao1} result in the real-valued ODE system
\begin{equation}\label{a0.5sigma1_periodic_eqns} 
\frac{d \omega_{-2,i}}{dt}=  \omega_{-2,i}^2\frac{(1-2v_c^2+5v_c^4)}{4v_c^2(1-v_c^2)^2} - 2\nu  v_c\omega_{-2,i},  \qquad
\frac{dv_c}{dt} =\frac{\nu}{2}(1-v_c^2)-\omega_{-2,i}\frac{1+v_c^2}{4v_c(1-v_c^2)}.
\end{equation}
We are not able to find analytical solutions  even to  the simplified system \e{a0.5sigma1_periodic_eqns}.
However, in the next section we  characterize initial data leading to finite-time singularity formation or alternatively to global-in-time existence of solutions  using a phase-plane analysis.

\subsubsection{Phase-plane analysis} \label{sec:phase-plane}

In this section we develop the phase-plane analysis for $a=1/2, \sigma=1$,  and in  Appendix \ref{sec:AppendixA}
we provide the analysis for    $a=1/2,~ \sigma=0$ to complement the results of Section \ref{sec:a1p2sigma0}. An advantage of the phase-plane analysis is that it does not require  closed form solutions, which are usually not available.

We consider the real-valued ODE system  \e{a0.5sigma1_periodic_eqns}  and
first point out a difficulty in the phase-plane analysis.  Inspection  of the right-hand side of the equation for $d \omega_{-2,i}/dt$  shows that it is positive   in the  lower right quarter-plane $\omega_{-2,i}<0$, $v_c>0$, indicating that  trajectories approach $\omega_{-2,i}=0$ as $t$ increases.
Furthermore, the trajectories generally only reach $\omega_{-2,i}=0$ in infinite time.
Thus at first glance the lower right quarter-plane of $(v_c, \omega_{-2,i})$ space  is an invariant domain  in which the solution exists globally in time. However, there exist `pathological' trajectories which  pass through the point $(v_c, \omega_{-2,i})=(1,0)$  with   $\omega_{-2,i}^2/(1-v_c)^2 $ approaching a nonzero constant at that point, since the right-hand side  of the equation for $d \omega_{-2,i}/dt$ is then $O(1)$  as $v_c  \rightarrow 0$,  $\omega_{-2,i} \rightarrow 0$.
These trajectories emerge in the upper right quadrant of $(v_c, \omega_{-2,i})$ space   where their  behavior is undetermined by the analysis.

As will be seen, this difficulty in the analysis can be avoided by using
\be  \label{eq:p_substitution}
p=\frac{\omega_{-2,i}}{v_c(1-v_c^2)},
\ee
as an independent variable instead of $\omega_{-2,i}$. We continue to use $v_c$ as the second independent variable. 

We express the norms (\ref{a0.5sigma0_periodic_solution_norms_real}) in terms of $p(t)$ and $v_c(t)$:
\begin{equation}
\begin{split}
 \| \omega(\cdot,t) \|_{L^2}^2 &= 2 \pi  p^2(t) \left( \frac{1}{v_c(t)} + v_c(t) \right)  \label{a0.5sigma0_periodic_solution_norms_real_p}\\
 \| \omega(\cdot,t) \|_{B_0}& =\left\{ \begin{matrix} 2|p(t)|/v_c(t), ~\mbox{for} ~ 0 < v_c \leq 1 \\ 2|p(t)| v_c(t), ~\mbox{for}~ v_c > 1.~~~~~~
\end{matrix}
\right.
\end{split}
\end{equation}
In terms of the variable $p$ defined in (\ref{eq:p_substitution}) and $v_c$, the system  (\ref{a0.5sigma1_periodic_eqns}) becomes
\begin{align}
\begin{split}
\frac{dp}{dt} &:= J(v_c,p)= p(pR-\nu Q), \label{eq:p_sig=1} \\
\frac{d v_c}{dt}  &:=K(v_c,p)=-\frac{p}{4} (1+v_{c}^{2})+\frac{\nu}{2}(1-v_{c}^{2}).
\end{split}
\end{align}
where
\begin{equation} \label{eq:RQ}
R= \frac{1}{2} \left( \frac{1}{v_c} - v_c \right)~~\mbox{and}~~Q=\frac{1}{2} \left( \frac{1}{v_c} + v_c \right).
\end{equation}
Henceforth let $\nu=1$, which is equivalent to a rescaling of $t$ and $\omega$.  Note that
\be \label{eq:symm}
J(v_c^{-1},-p)=-J(v_c,p)~~\mbox{and}~~K(v_c^{-1},-p)=-v_c^{-2} K(v_c,p),
\ee
implying that  (\ref{eq:p_sig=1}) is
invariant under the substitution
$v_c \rightarrow 1/v_c ~\mbox{and}~ p \rightarrow -p.$
Hence we need only perform the phase-plane analysis for $p \geq 0$, with the corresponding results for   $p<0$ following from symmetry.

We first   consider the nullclines of the system (\ref{eq:p_sig=1}), which are the  curves in the $p$ versus $v_c$  phase plane where $J(v_c,p)=0$ and separately   the curves where $K(v_c,p)=0$. The $J$-nullclines are   $p=0$ and $p=f_J(v_c)$ where
\begin{equation}\nonumber
f_J(v_c) =  \frac{1+v_c^2}{1-v_c^2},
\end{equation}
for $0 <v_c < \infty$.  These are shown by the dashed curves in Figure \ref{fig:phaseplane_sig=1}.   The $J$-nullclines divide $p \geq 0$ phase-space into  two connected regions which differ in their sign of  $dp/dt$:
\begin{align}
\begin{split}
\label{upper_quadrant}
(a)~ \frac{dp}{dt}&>0 ~\mbox{for}~ 0<v_c<1 ~\mbox{and} ~ f_J<p<\infty,  \\
(b)~ \frac{dp}{dt}&<0 ~\mbox{for}~ 0<v_c<1 ~\mbox{and} ~ 0<p<f_J  ~\mbox{or}~ 1 \leq v_c<\infty ~\mbox{and}~ 0<p<\infty.
\end{split}
\end{align}
The $K$-nullcline is given by  the curve $p=f_{K}(v_c)$ where
\begin{equation}\nonumber 
f_{K} (v_c)=2  \left( \frac{1-v_c^2}{1+v_c^2} \right).
\end{equation}
for $0 < v_c < \infty$.  The direction field $S$, which prescribes the slopes of the trajectories in the phase plane, is given  by
\be \label{eq:S}
S(v_c,p) : = \frac{J(v_c,p)}{K(v_c,p)}= \frac{p(pR- Q)}{-\frac{p}{4} (1+v_{c}^{2})+\frac{1}{2}(1-v_{c}^{2})}
\ee
using (\ref{eq:p_sig=1}). Given  a  piecewise smooth curve  $p=f(v_c)$  in the phase plane we also define
\be  \label{eq:T}
T(v_c, f(v_c)) :=  -K(v_c, f(v_c)) f'(v_c) +J(v_c, f(v_c))
\ee
for all $v_c \in (0,\infty)$ at which $f'(v_c)$ exists;  otherwise $T$ is defined to be the average $T(v_c,f(v_c))=(1/2) (T(v_c^+,f(v_c^+)) + T(v_c^-,f(v_c^-)))$.
Note that $T$ gives the component of the trajectory at $(v_c,f(v_c))$  in the direction of the upward pointing normal to the curve $f(v_c)$.
If $T$ is one-signed (i.e., has no zeros) in $0<v_c<\infty$ then the trajectories intersecting the curve  $f(v_c)$ are transverse to that curve.




\subsubsection{Global existence} \label{sec:global}
         We show that there is an invariant region $\Omega$   of the $p$ versus $v_c$  phase plane in which the solution to (\ref{eq:p_sig=1}) exists globally in time (Theorem \ref{thm:global_sig=1}).  At any point in the complement $\Omega^c$ the solution may blowup in finite time -- indeed Theorem \ref{thm:blowup} guarantees it blows up in a subset of $\Omega^c$.

Define $\Omega$ to be the region of the phase plane  that is  bounded above and below   by the curves
\begin{align} \label{eq:p1p2}
\begin{split}
\pone&= 1  + c_1  H(v_c-1) \cdot (v_c-1),\\
\ptwo&=-1  - c_1 H(1-v_c) \cdot(v_c^{-1}- 1),
\end{split}
\end{align}
for $0<v_c<\infty$,
 where  $H(x)$ is the Heaviside function
and $0<c_1<  (3+\sqrt{73})/8$.  The region $\Omega$ is illustrated in Figure \ref{fig:phaseplane_sig=1}. Note that $p_2(v_c)$ is obtained from $p_1(v_c)$ using the symmetry (\ref{eq:symm}), i.e., $p_2(v_c)=-p_1(v_c^{-1})$.    We have for $(v_c(t),p(t))$ satisfying (\ref{eq:p_sig=1}):
\begin{theorem} \label{thm:global_sig=1} {\bf Global existence for ${a=1/2}$, ${ \sigma=1}$}.
For all
 initial data $(v_c(0),p(0))$  in $\Omega$ and all $t>0$,  $(v_c(t),p(t)) \in \Omega$  with  $v_c(t)$ bounded away from zero and infinity. Hence the solution (\ref{omegaa0.5sigma0_periodic_lower}, \ref{ODEreductiona1p2sigmao0})  exists and is analytic for all $t>0$.  Furthermore $p(t) \rightarrow 0$ and  $\omega(x,t) \rightarrow 0$
 as $t \rightarrow \infty$.
\end{theorem}

\begin{figure}[h]
\centering
\includegraphics[width=0.8\textwidth]{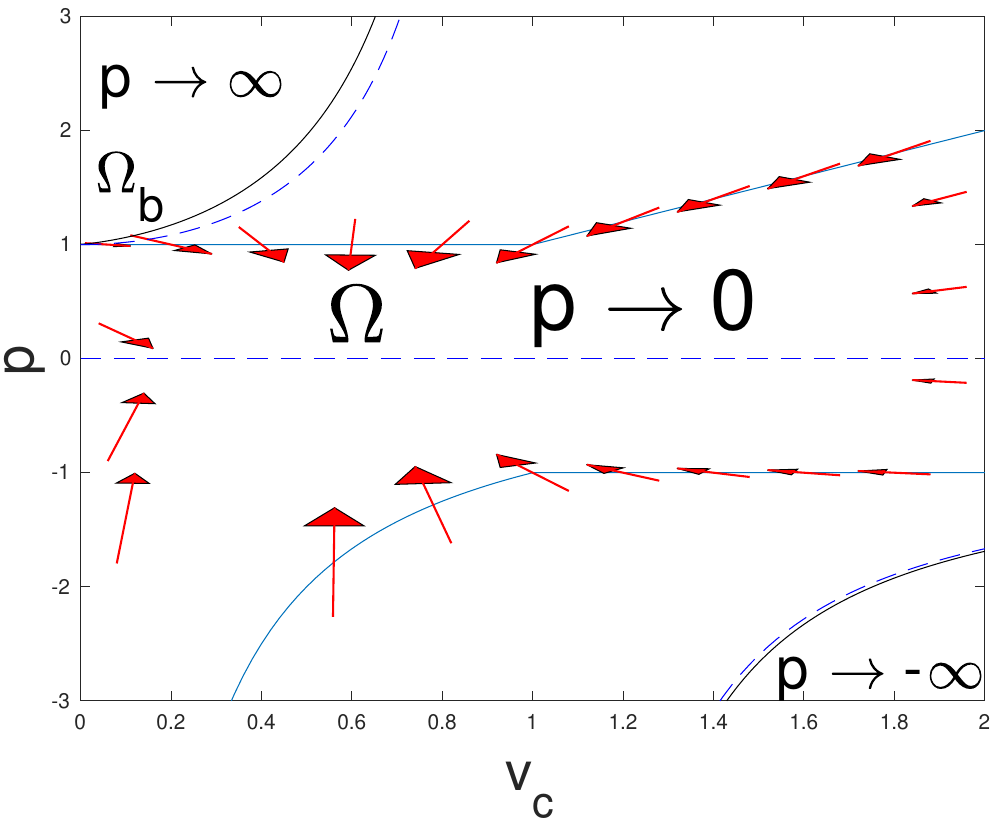}
\caption{Phase plane for the system (\ref{eq:p_sig=1}).  Dashed lines are $J$-nullclines.  A global existence region $\Omega$  is shown with the direction field $S$ overlaid on its boundary.  A blow-up region  $\Omega_b$ is bounded by  a solid curve. There is another blowup region at lower right.}
 \label{fig:phaseplane_sig=1}
\end{figure}

\noindent
 {\it Proof}.  We show that  $\Omega$ is an invariant subset under the dynamics given by (\ref{eq:p_sig=1})
 by demonstrating  that  the  trajectories, which have slope $S(v_c,p)$, are directed into the the interior of  $\Omega$ (in view of the symmetry (\ref{eq:symm}), it is enough to consider only the upper and left boundaries of $\Omega$). This   is a consequence of the following  properties: (i) $dp/dt<0$ at all points  $(v_c,\pone)$ on the upper boundary of $\Omega$, 
 (ii) $T(v_c,p_1(v_c))<0 $ for $0<v_c<\infty$,   
 and (iii) $v_c'(t) >0$ for $v_c \ll 1$ and  $v_c'(t) <0$ for $v_c \gg1$.
 Properties (i) and (ii)  ensure  that the trajectories are transverse to the upper boundary  of $\Omega$ (hence also the lower boundary)  and directed into $\Omega$.   Property (iii) implies that $v_c$ can neither tend to zero or infinity, so the singularity does not impinge on the imaginary axis or go off  to infinity.
 Figure \ref{fig:phaseplane_sig=1} illustrates the domain $\Omega$ and  direction field.

Property (i)  follows  from  (\ref{upper_quadrant}b),  since $0<p_1(v_c)<f_p(v_c)$ for $0<v_c <1$, and otherwise $0<p_1<\infty$. Property (ii) is a consequence of the following estimates, divided into two cases:

{\em{Case 1: $0<v_c \leq 1$}}.    By direct calculation $T(v_c,p_1(v_c)) =-v_c$, which is strictly less than zero for all $v_c$ in the indicated range.
Thus trajectories are transverse to  and directed into the upper boundary of $\Omega$.

 {\em{Case 2: $1<v_c  < \infty.$} } We write
 \begin{align}\nonumber
 T(v_c,p_1(v_c)) &=  -\frac{U(v_c)}{4 v_c}
 \end{align}
 where from (\ref{eq:S})  and (\ref{eq:T}),   $U(v_c) :=2 p_1^2(v_c)(v_c^2-1)+2 p_1(v_c) (v_c^2+1) - c_1 V(v_c)$ with $V(v_c)  := v_c [p_1(v_c) (1+v_c^2) + 2 (v_c^2-1)]$.
 We show that $U(v_c)$  can  be made strictly positive with a suitable choice of $c_1$.   Substitute for $p_1(v_c)$ using (\ref{eq:p1p2})
 to find
 \be \label{eq:U}
 U(v_c)= \sum_{n=0}^{4} A_n (v_c-1)^n
 \ee
 where  $A_0=4-2c_1$, $A_1=8-4c_1-2c_1^2$, $A_2=4+3 c_1-4 c_1^2$, $A_3=3 c_1 + c_1^2$,  and $A_4=c_1^2$. The
 coefficients in (\ref{eq:U})
 are  all non-negative when $0<c_1 \leq \frac{3+\sqrt{73}}{8} \approx 1.443$, in which case the zeros of $U(v_c)$ can only occur for $v_c<1$.
 It follows that $T(v_c, p_1(v_c))$  is
 negative for all $v_c$ in $1<v_c<\infty$.
 This finishes the verification of property (ii).

To verify property (iii), note from the second equation of (\ref{eq:p_sig=1})  that $K(v_c,p)>0$ for $(v_c,p) \in \Omega$ with $v_c$ near zero. Hence trajectories are directed into the interior of $\Omega$ near the boundary $v_c=0$. Similarly, $K(v_c,p)<0$ for  $(v_c,p) \in \Omega$ and $v_c>>1$. This establishes property (iii). It follows that $(v_c(t), p(t)) \in \Omega$ for all $t>0$. Finally,  $p \rightarrow 0$ as $t \rightarrow \infty$ since $J(v_c,p)$ in (\ref{eq:p_sig=1}) is negative
  (respectively positive)  for $(v_c,p) \in \Omega$ with $p>0$ (respectively $p<0)$. As $v_c$ is bounded away from $0$ or infinity, $\omega \rightarrow 0$ as $t \rightarrow \infty$  follows. $\square$

\begin{remark}
More generally, the upper boundary of $\Omega$ can be constructed to have the   power law behavior $p_1(v_c) \sim c v_c^\gamma$ for constant $c>0$.
Then asymptotically for   $v_c \gg 0$,  $T \sim (\gamma-2) c v_c^{\gamma-1}/4$, and the choice $1<\gamma<2$  implies that  $T<0$ for  $v_c \gg 0$. We have constructed such $p_1(v_c)$  so that properties  (i)-(iii) in the proof of Theorem \ref{thm:global_sig=1} are satisfied  in the entire range $0<v_c<\infty$.  This gives an expanded domain $\Omega$ of  global existence (see Figure \ref{fig:pode}).
\end{remark}

A consequence of Theorem \ref{thm:global_sig=1} is that the $L^2$ and $B_0$  norms of the initial data $\omega(x,0)$  must be sufficiently large for the solution  (\ref{omegaa0.5sigma0_periodic_lower})   to blow up in finite time. This is because $|p(0)|$ must be large enough for the initial data to be in $\Omega^c$, outside of the region of global existence (see Figure \ref{fig:phaseplane_sig=1}).  Combined with
(\ref{a0.5sigma0_periodic_solution_norms_real_p}), this  implies that  $\| \omega (\cdot, 0) \|_{L^2}, ~\| \omega (\cdot, 0) \|_{B_0}$ are bounded from below in $\Omega^c$,  and hence there is no blowup for sufficiently small data. This contrasts the situation for $a=1/2, ~\sigma=0$ in which  $\| \omega \|_{L^2}$ can blow up for arbitrarily small data, as noted in Remark 3 of Section \ref{sec:sufficient}.   An analogous `double pole' exact solution for the problem on $\mathbb{R}$ also blows up for arbitrarily small data \cite{AmbroseLushnikovSiegelSilantyev}.

\subsubsection{Finite-time singularity formation}

We next show that $p$ blows up in finite time for  initial data in the set   $\Omega_b$  shown in
Figure  \ref{fig:phaseplane_sig=1}.  To motivate this, note from (\ref{eq:p_sig=1}) (with $\nu=1$)  that if
$p(0)  >0$ and
\be \label{blow_up_region}
p(t) R(v_c(t))-Q(v_c(t))>\lambda p(t)+ \mu
\ee
for constants $\lambda, \mu>0$ and $t>0$, then $p$ increases  in time and $dp/dt$ grows quadratically in $p$, implying that $p$ blows up in finite time.  The key to showing finite-time  blowup is to find $\lambda$ and $\mu$ such that the inequality (\ref{blow_up_region}) holds for all $t>0$.  We show this for the particular choice $\lambda=1/100$ and $\mu=1/5$.
\begin{theorem} \label{thm:blowup} {\bf Finite-time singularity formation for ${a=1/2}$, ${\sigma=1}$.}
Let $\Omega_b$ be the region of the upper-half phase plane lying on or above the curve  $p_b(v_c)=\frac{Q(v_c)+\mu}{R(v_c)-\lambda}$
for  $\lambda=1/100$ and $\mu=1/5$, where $0<v_c < \sqrt{1-2 \lambda}$.  Then for initial data $(v_c(0),p(0)) \in \Omega_b$, we have for $t>0$
\be \label{eq:p_bound}
p(t)>\frac{ \mu D e^{\mu t}}{1-\lambda D e^{\mu t}}
\ee
where $D=p(0)/(\lambda p(0)+\mu)$. Hence $p(t) \rightarrow \infty$,  $\| \omega \|_{L^2} \rightarrow \infty$, and $\| \omega \|_{B_0} \rightarrow \infty$
in finite time $0< t_c<-(1/\mu) \ln (\lambda D)$.
\end{theorem}

\begin{remark}
It follows from the symmetry (\ref{eq:symm}) that $p(t) \rightarrow -\infty$ in finite time for $p(0)<-\frac{Q(v_c^{-1}(0))+\mu}{R(v_c^{-1}(0))-\lambda}<0$ and
$v_c(0)>(1+2 \lambda)^{-1/2}$  in the lower \cblue{right quarter} of $(v_c, p)$ space (see Figure \ref{fig:phaseplane_sig=1}).
\end{remark}

\noindent
 {\it Proof}. We have that  $f_J(v_c)< p_b(v_c) <  \infty$ when $0<v_c<\sqrt{1-2 \lambda}$,  and thus the curve $p_b$ lies above the upper $J$-nullcline.
 Hence $dp/dt>0$ on the boundary of $\Omega_b$. The unique  intersection of   $p_b(v_c)$  with the $K$-nullcline occurs at $v_c = d \simeq 0.3734$. Hence both $dv_c/dt$ and $S$ evaluated on the boundary of $\Omega_b$ are negative  when $v_c>d$, and  it follows that  trajectories are transverse to this part of the boundary and directed into $\Omega_b$.  When $0<v_c \leq d$, transversality of the trajectories is graphically verified. This is done in Figure \ref{fig:blowup}  which  shows that  $T(v_c,p_b(v_c))>0 $ for $0<v_c \leq d$ in the particular case of   $\lambda=1/100$ and $\mu=1/5$.  Since the definition (\ref{eq:T}) uses the upward pointing normal, this implies the trajectories are directed into $\Omega_b$.

The above  arguments  show that $\Omega_b$  is an invariant subset under the dynamics given by (\ref{eq:p_sig=1}). Since  (\ref{blow_up_region}) is satisfied at all $(v_c,p)$  in $\Omega_b$, it follows from the invariance of $\Omega_b$ that  this inequality holds for all $t>0$ if it holds at $t=0 $.  Therefore from (\ref{eq:p_sig=1})  and (\ref{blow_up_region})  we have that
 \[
 \frac{dp}{dt} > p(\lambda p+\mu)~\mbox{for}~ t>0, ~\mbox{if}~ (v_c(0),p(0)) \in \Omega_b.
 \]
 Integrating the above inequality with initial data $p(0)$ gives (\ref{eq:p_bound}). Finite-time blowup of the $L^2$ and Wiener norms follows from (\ref{a0.5sigma0_periodic_solution_norms_real_p}).
  $\square$

  \cblue{Similarity exponents are computed by Taylor expanding (\ref{eq:p_sig=1}) for  $t$ near $t_c$, and expressing the result in terms of $\omega_{-2,i}$ using (\ref{eq:p_substitution}). We find  $v_c \sim (t_c-t)^{1/3}$, $\omega_{-2,i} \sim (t_c-t)^{-1/3}$ in type A blowup, and $v_c \sim (t_c-t)^{1/3}$, $\omega_{-2,i} \sim (t_c-t)^{-5/3}$ in type B blowup.   Consequently,  equation (\ref{a0.5sigma0_periodic_solution_norms_real}) implies that for both types of blowup,  $\| \omega(\cdot,t) \|_{L^2} \sim (t_c-t)^{-5/6}$ and $\| \omega(\cdot,t) \|_{B_0} \sim (t_c-t)^{-1}$  as $t \rightarrow t_c^-$. The corresponding similarity parameters (cf. (\ref{self-similar1})) are $\alpha=1/3$ and $\beta=1$.}

 \begin{figure}[h]
\centering
\includegraphics[width=0.6\textwidth]{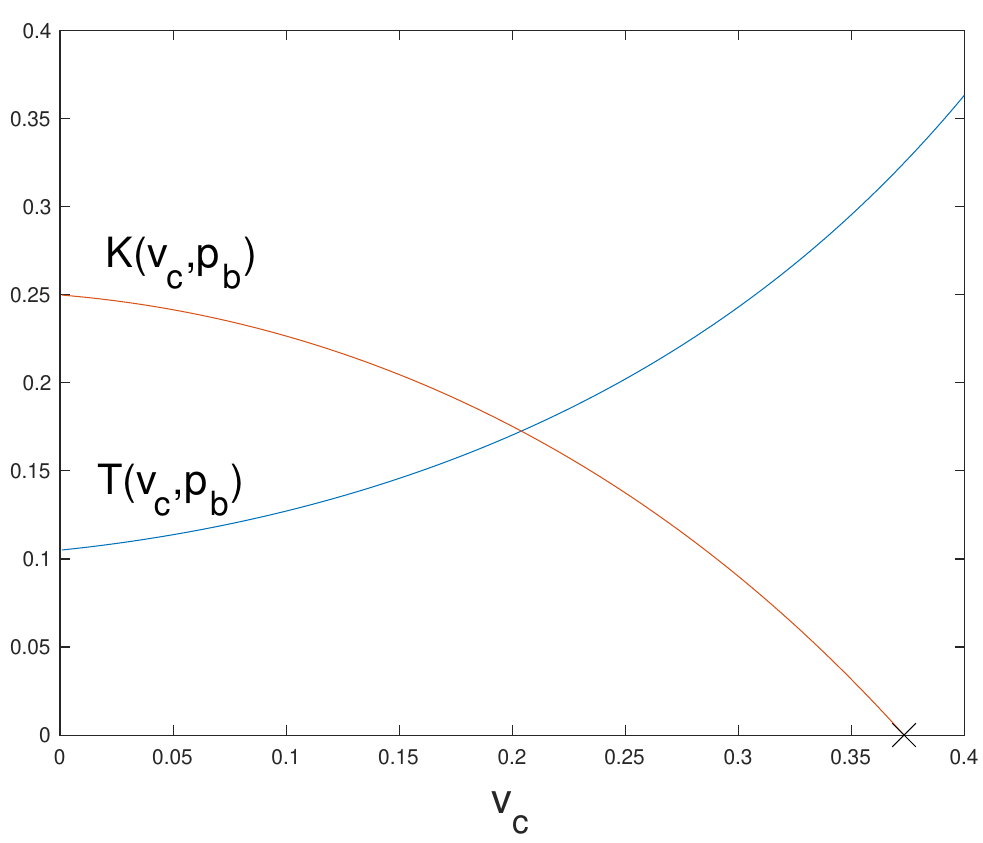}
\caption{Plot showing $T(v_c,p_b)>0$ for $0 < v_c \leq d $, where $d \simeq 0.3734$ (shown by `$\times$') is the zero of $K(v_c,p_b)$. When $v_c>d$,  the trajectories and the boundary of $\Omega_b$ are oppositely sloped, hence transverse.  }
\label{fig:blowup}
\end{figure}




\subsubsection{Numerical solution of (\ref{eq:p_sig=1})}

\begin{figure}[h!]
\centering
\includegraphics[width=0.8\textwidth]{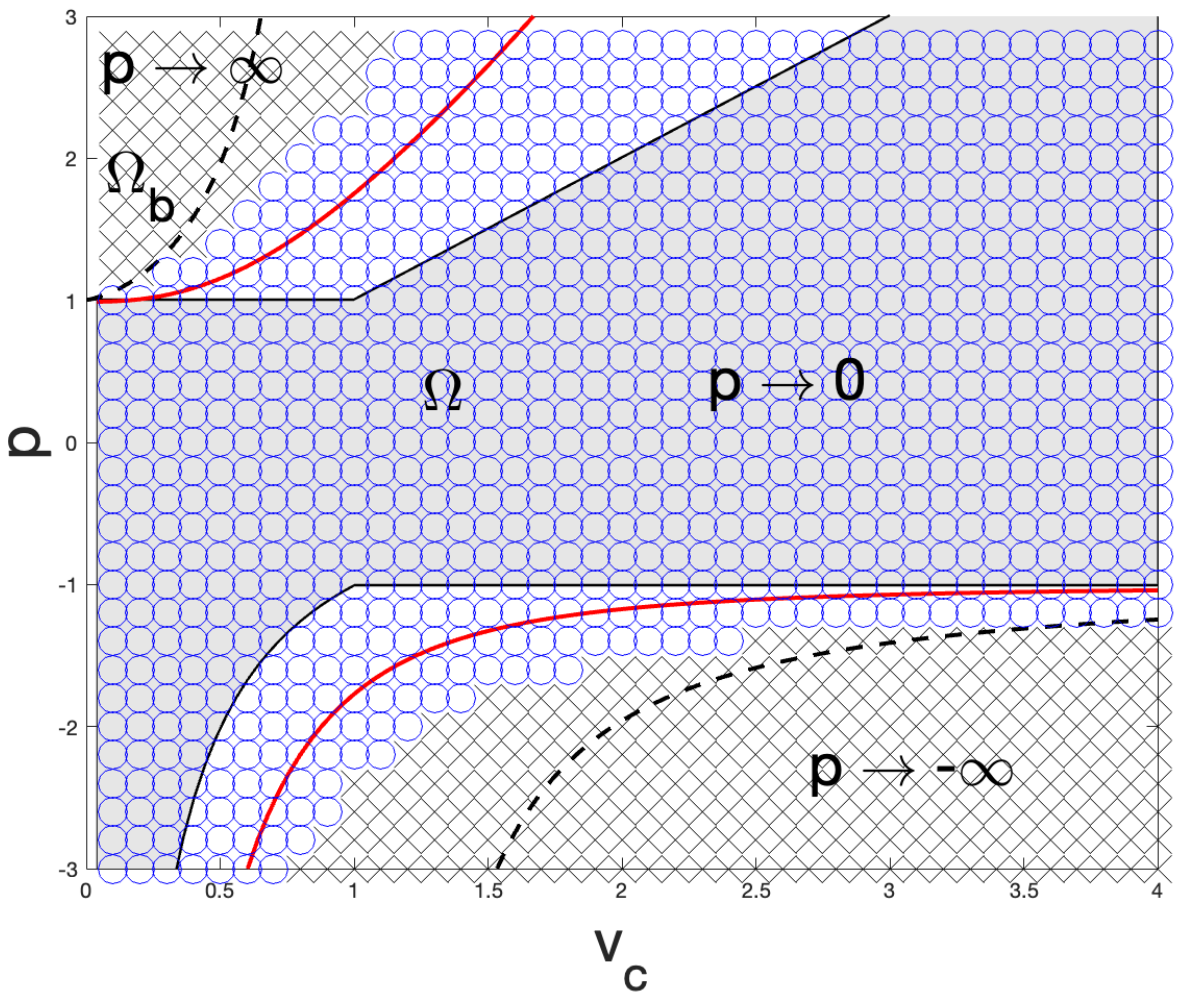}
\caption{Comparison of numerical solution to (\ref{eq:p_sig=1}) with the phase-plane analysis. Circles and $\times$ markers  are numerically determined regions of global existence and blowup, respectively. The global existence region   $ \Omega$ is shaded, the  red curves demarcate a larger region of global existence obtained by a more detailed analysis,  and dashed curves delineate the analytically determined blowup regions.  }
 \label{fig:pode}
\end{figure}

Figure \ref{fig:pode} compares the numerical solution to (\ref{eq:p_sig=1}) with the phase-plane analysis. The blue circles denote initial data $(v_c(0),p(0))$ for which the numerical solution appears to exist globally in time.  This data is consistent with the global existence region $\Omega$ from the analysis of Section \ref{sec:global} (shaded), as well as a larger such region demarcated by red curves, determined from a more detailed analysis in the spirit of the remark following Theorem \ref{thm:global_sig=1}. The `$\times$' markers denote initial data for which the numerical solution blows up in finite time, which is  consistent with the   analytically determined  blowup regions  indicated by dashed curves

We show these results in terms of  the original variables $\omega$ and $v_c$ in Figure \ref{fig:omega}.  The numerical solution is computed from  (\ref{a0.5sigma1_periodic_eqns}),  and gives a `global existence' region that
is consistent with the analysis.
Analytically determined blowup regions  are delineated by dashed curves and are also consistent with the numerics.   The blowup   region in the lower half plane of $\omega_{-2,i}$ in Figure \ref{fig:omega} corresponds to  `$\times$' markers  in  Figure \ref{fig:pode} that lie between  solid and dashed curves. These are blow-up points that are  not classified by the phase-plane analysis.

Note that  in Figure \ref{fig:omega} the upper boundary of the  gray-shaded region tends to zero   as $v_c \rightarrow 0$, linearly in $v_c$, and
both  upper and lower boundaries  tend to zero as $v_c \rightarrow 1$, linearly in $1-v_c$.  However, the $L^2$ and $B_0$  norms of $\omega$ (cf. (\ref{a0.5sigma0_periodic_solution_norms_real_p})) are bounded away from zero on these boundaries. In other words,  $\| \omega \|_{L^2}$  and $\| \omega \|_{B_0}$ are  large outside of the gray-shaded region, implying that  $\| \omega \|_{L^2},  \| \omega \|_{B_0} \rightarrow \infty$ in finite time only for sufficiently large data.

 \begin{figure}[h]
\centering
\includegraphics[width=0.8\textwidth]{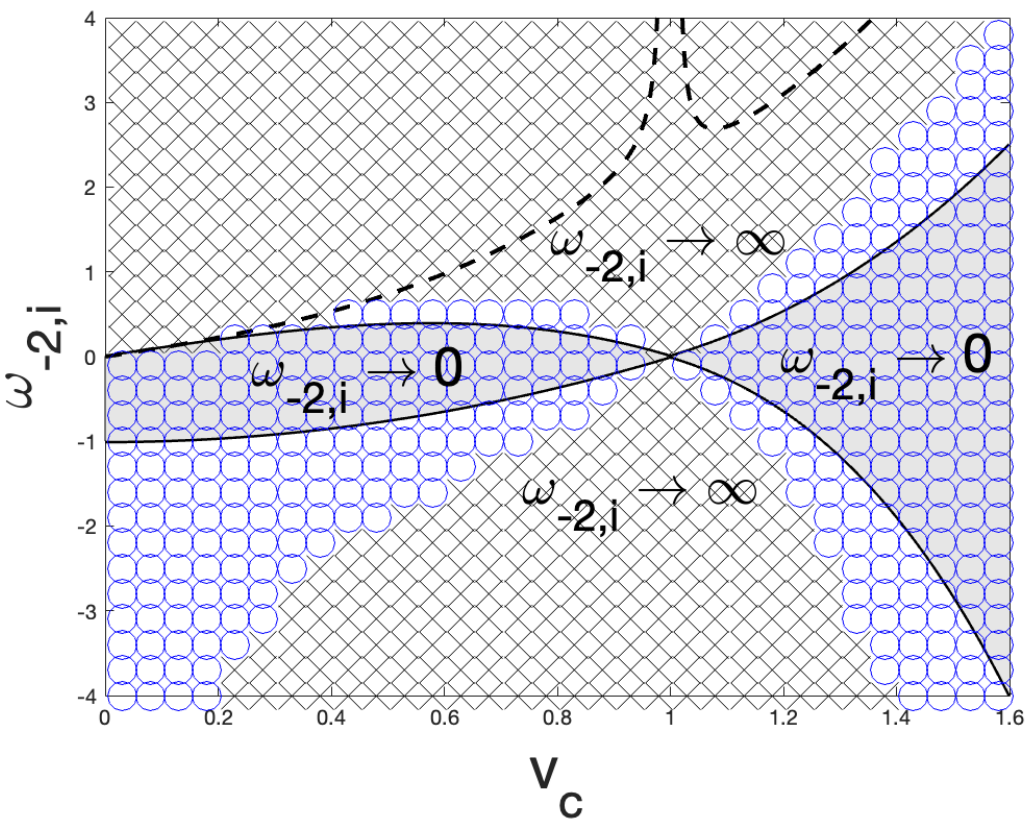}
\caption{Numerical solution  in  the original variables $\omega_{-2,i}$ and $v_c$, plotted as in Figure \ref{fig:pode} and  compared with the phase-plane analysis.  The global existence region corresponding to $\Omega$ in Figure \ref{fig:phaseplane_sig=1}  is  shaded; analytically determined blowup regions are delineated by dashed curves.   }
 \label{fig:omega}
\end{figure}

\subsubsection{Limit of the real line }
\label{sec:limitreallinesigma2}

\cblue{As in Section \ref{sec:limitreallinesigma0}, pole dynamics equations for the problem on the real line are obtained in the limit $v_c\to 0$, for which the influence of the $2\pi$ periodicity is negligible. In this case, Laurent series expansion of  the right-hand sides of   \e{a0.5sigma1_periodic_eqns} about $v_c=0,$ keeping only the leading order term
$\propto v_c^{-2}$  in the first ODE and $\propto v_c^{-1}$ and $\propto v_c^0$ in the second ODE results in the reduction
 %
\begin{equation}\label{reductiona0.5sigma1}
\frac{d \omega_{-2,i}}{dt}=  \omega_{-2,i}^2\frac{1}{4v_c^2},  \qquad
\frac{dv_c}{dt} =-\omega_{-2,i}\frac{1}{4v_c}+\frac{\nu}{2},
\end{equation}
which recovers  (50) and (51) of \cite{AmbroseLushnikovSiegelSilantyev}  with $\omega_{-2,i}$ replaced by $ \omega_{-2}/4$ and $v_c$ replaced by $v_c/2$ to match the definitions of \cite{AmbroseLushnikovSiegelSilantyev}. This replacement can be understood as discussed in Section \ref{sec:limitreallinesigma0}.
The velocity $u(x,t)$, given by  \e{uap1p2sigma0reduction},
recovers (47) of   \cite{AmbroseLushnikovSiegelSilantyev} for $q(t)\equiv 0.$
}

\cblue{If we similarly  perform  Laurent series expansion of the  right-hand sides of the more general ODEs \e{vcODEa1p2sigma1} and \e{omm2ODEa1p2sigma1}, then we obtain new exact solutions of \e{mainEquationa0.5sigma1_periodic} on the real line in the form \e{omegam2reduction}, \e{uap1p2sigma0reduction}
with
\begin{equation}\label{vcODEa1p2sigma1reduction}
\frac{d\omega_{-2}}{dt} =\I \omega_{av}(t)\omega_{-2}
+\frac{ \I|\omega_{-2}|^2 }{( v_c+\bar  v_c)^2}
, \quad \frac{dv_c}{dt} =-\frac{\I  \bar \omega_{-2}}{2  ( v_c+\bar  v_c)}+\frac{\nu}{2}-\frac{\I q(t)}{4}, \quad \omega_{av}(t)=\omega_{av}(0),
\end{equation}
where we have also used \e{ca1p2sigma1}.
This can be also derived by  direct substitution of \e{omegam2reduction}, \e{uap1p2sigma0reduction} and \e{vcODEa1p2sigma1reduction} into \e{mainEquationa0.5sigma1_periodic}.
}

\section{Conclusions} \label{sec:Conclusions}

Exact pole dynamics solutions have been presented  for the generalized Constantin-Lax-Majda equations with dissipation  on a periodic domain. The solutions are obtained  for advection parameters $a=0$ and $1/2$ and  dissipation parameters $\sigma=0$ and $1$, for which there is a balance of singular terms.  When  $a=0$, the solutions  consist of a periodic array of $N$ complex conjugate simple pole singularities with time-dependent positions and amplitudes.   Closed form solutions for the singularity positions and amplitudes are obtained in the case of a single (periodically repeated) c.c. pair of singularities, which enables a rather complete analysis of finite-time singularity formation.   We find finite-time blowup occurs  for arbitrarily small initial data $\omega(x,0)$ in the $L^2$ norm when $\sigma=0$, but only  for sufficiently large data when $\sigma=1$.
In the case of the Wiener norm, the exact solutions exhibit finite-time singularity formation only for sufficiently large data for both $\sigma=0$ and $1$. This result motivates  revisiting the well-posedness theory for the Wiener algebra presented in \cite{AmbroseLushnikovSiegelSilantyev}, which there applies for $\sigma \geq 1$.
We find that $a=0$ is a special case in which the $B_0$ theory in the periodic problem  can be extended   to $\sigma \geq 0$, for which we prove global well-posedness  for  small initial data (our proof also applies to the real-line problem for $\sigma=0$).  
Crucially, the  pole dynamics solutions  show that the small data assumption in the global existence theory is necessary, as finite time  blowup provably occurs for sufficiently large data.

In the case $a=1/2$, solutions are constructed as a sum of  c.c. simple and double pole singularities which enables a problematic $\log$ singularity to be cancelled out.
The analysis is more delicate in this case.
We find an implicit analytical solution for $\sigma=0$,  and a detailed analysis of this solution   gives a complete picture of singularity formation.  When $\sigma=1$ closed form analytical solutions are not available, and we resort to a phase-plane analysis using special variables to obtain useful information about singularity formation.  The results are similar to $a=0$: we find finite-time blowup  for arbitrarily small initial data in the $L^2$ norm when $\sigma=0$, but  global existence for small data when $\sigma=1$. There is global existence for sufficiently small data in the  Wiener norm for $\sigma=0$ and $1$. In both norms we prove that there is collapse for sufficiently large data when $\sigma=0$ and $1$. \cblue{A summary of our results is given in Table 1.}

\renewcommand{\figurename}{Table}
\renewcommand\thefigure{1}

 \begin{figure}[h]
\centering
\includegraphics[width=0.8\textwidth]{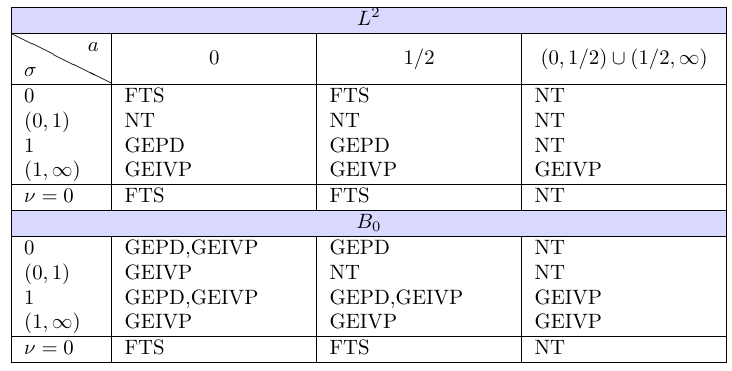}
\caption{\cblue{Summary of our results on global existence versus finite-time singularity formation in $L^2$ and $B_0$. FTS = finite-time singularity formation for arbitrarily small data; GEPD = global existence of pole dynamics solutions for small data; GEIVP = global existence of solutions to initial value problem for small data; NT = neither pole dynamics solutions nor our general theory applies.
The pole dynamics solutions blow up for sufficiently large data.}}
 \label{tab:summary}
\end{figure}

\renewcommand{\figurename}{Figure}
\renewcommand\thefigure{11}

When $a=0$ and $\sigma=2$, a pole dynamics solution for the problem on the real line was found by Schochet \cite{Schochet}. Remarkably,  this solution exhibits finite time blowup in both the $L^2$ and Wiener norms for {\em all} initial data.  Despite  significant effort, we have been unable to generalize this solution to the periodic domain.   However, the general theory  developed here paradoxically shows that   collapse in the Wiener algebra  occurs more readily  for the problem on the real line with $\sigma=2$ than it does \cblue{in the periodic problem} for $\sigma=0$!

In future work, we intend to numerically investigate finite-time singularity formation in the gCLM equation with dissipation  over the full range of $a$ and for more general values of $\sigma$,  both for problems on the real line and periodic domain.
Several interesting questions arise.
For example, in the special case $a=-1$ (for which the gCLM equation is equivalent to the Cordoba-Cordoba-Fontelos equation) global well-posedness in time is known for $\sigma \geq 1$ and finite-time singularity formation occurs for $\sigma<1/2$, but the behavior for $1/2 \leq \sigma<1$ is open \cite{Kiselev}. Techniques developed to study this problem can have ramifications on the question of finite-time singularity formation in the Navier-Stokes equations \cite{Wang2023}.   The uniqueness of blow-up solutions is also of interest. Huang et al. \cite{Huang2024} find that for the nondissipative problem on the real line, there are countably infinite distinct self-similar solutions which blow up in finite time.  It is natural to ask whether such nonuniqueness  also occurs in the problem with dissipation. \cblue{Similarly, it would be interesting to characterize the different scales of self-similar blowup due to the nontrivial time dependence of $x_c(t)$ (including $x'_c(t_c)=\infty$ \cite{Lushnikov_Silantyev_Siegel,Chen2021}) in the spirit of \cite{HuangTwoScale2024}}.

\section*{Acknowledgments}

We are grateful to the National Science Foundation for support through grants DMS-2307638 (DMA) and DMS-1909407 (MS).
PML would like to thank the Isaac Newton Institute for Mathematical Sciences, Cambridge (EPSRC grant EP/R014604/1) for support and hospitality during the programme ``Emergent phenomena in nonlinear dispersive waves,"
where work on this paper was partially undertaken.
Simulations were performed at the Texas Advanced Computing Center using the Extreme
Science and Engineering Discovery Environment (XSEDE), supported by NSF Grant ACI-1053575.

\section*{Author contributions}

D. A. Silantyev initiated the work and contributed to the analysis in Sections \ref{sec:a0sigma0},  \ref{sec:a0sigma1}, and \ref{sec:implicit_formula}.
P. M. Lushnikov  developed the analysis of the implicit solution in Section \ref{sec:a1p2sigma0}, as well as Section \ref{sec:limitreallinesigma2}. M. Siegel contributed to the analysis in  Sections \ref{sec:a0sigma0} and \ref{sec:a0sigma1},  developed the phase-plane analysis in Section \ref{sec:phase-plane} and  Appendix \ref{sec:AppendixA}, and implemented the numerics. D. M. Ambrose developed the theory in Sections \ref{sec:a=0_well-posedness},  \ref{sec:a=0_well-posedness_1} and contributed to Section \ref{sec:a1p2sigma0}.  All authors contributed to the analysis generally and participated in writing the manuscript.

\section*{Data availability statement}
Data sharing not applicable to this article as no datasets were generated or analyzed during the
current study.

\bibliography{lushnikov,ref_1DEuler_model,gclm,surfacewaves,biblio,biblionls}{}
\bibliographystyle{ama.bst.bst}

\appendix
\section*{Appendix}
\section{Phase-plane analysis for  ${{a=1/2, \sigma=0}}$} \label{sec:phase-plane_a=.5_sig=0} \label{sec:AppendixA}

 We provide the phase-plane analysis for the  the real-valued ODE system \e{a0.5sigma0_periodic_eqns} to complement the analysis of Section \ref{sec:a1p2sigma0}.
Again make the substitution (\ref{eq:p_substitution})
for which   (\ref{a0.5sigma0_periodic_eqns}) becomes
\begin{align}
\begin{split}
\frac{dp}{dt} &:= L(v_c,p)= p(pR-\nu), \label{eq:p_sig=0} \\
\frac{d v_c}{dt}  &:=M(v_c,p)=-\frac{p}{4} (1+v_{c}^{2}).
\end{split}
\end{align}
where $R$ is given by (\ref{eq:RQ}). These equations satisfy the  same symmetry as (\ref{eq:symm}), i.e.,
\be \label{eq:symm_sig0}
L(v_c^{-1},-p)=-L(v_c,p)~~\mbox{and}~~M(v_c^{-1},-p)=-v_c^{-2} M(v_c,p).
\ee
The direction field is
\be \nonumber 
{S}:= \frac{L(v_c,p)}{M(v_c,p)} = -\frac{4 (pR-\nu)}{1+v_c^2}.
\ee
Similar to (\ref{eq:T}), we define
\be \label{eq:sig0_T}
{T}(v_c, f(v_c)) :=  -M(v_c, f(v_c)) f'(v_c) +L(v_c, f(v_c))
\ee
where $p=f(v_c)$ is any piecewise smooth curve in the phase plane.   The $L$-nullclines are   $p=0$ and $p={f}_L(v_c)$ where
\begin{equation}\nonumber
{f}_L(v_c) =  \frac{2 v_c}{1-v_c^2}.
\end{equation}
The quantity $|p|$ increases in time when $|p|>|{f}_L|$, and in Theorem \ref{thm:blowup_sig0} we find sets of initial data where $|p| \rightarrow \infty$ in finite time. We first show, however, that there is an invariant region ${\Omega}$ where the solution exists globally in time.

We define ${\Omega}$ to be the region  in the $p$ versus $v_c$ phase plane that is  bounded above and below  by the curves
\begin{align} \label{eq:q1q2}
\begin{split}
\qone&=  v_c,\\
\qtwo&=- v_c^{-1},
\end{split}
\end{align}
for $0<v_c<\infty$.
 Note that $q_2(v_c)$ is obtained from $q_1(v_c)$ using the symmetry (\ref{eq:symm_sig0}).  Figure \ref{fig:sig0_phase_plane}  illustrates the domain ${\Omega}$, L-nullclines  and  direction field ${S}$.   Taking $\nu=1$, we have:
\begin{theorem} {\bf Global existence for ${a=1/2}$, ${\sigma=0}$}. \label{thm:global_sig=0}
For all
 initial data $(v_c(0),p(0))$  in ${\Omega}$ and all $t>0$,  $(v_c(t),p(t)) \in {\Omega}$.  Furthermore, $v_c(t)$ cannot equal zero or infinity in finite time (although it can tend to zero or  infinity in infinite time,  while simultaneously $p(t) \rightarrow 0$ and $\| \omega \|_{L^2} \rightarrow 0$).
 Hence the solution $\omega(x,t)$  exists and is analytic for all $0< t<\infty$.
 \end{theorem}

 \noindent
 {\it Proof}.
 Substitution of (\ref{eq:q1q2}) into (\ref{eq:p_sig=0}) and (\ref{eq:sig0_T})  shows that
 $dp/dt=-v_c(1+v_c^2)$ and ${T}(v_c,q_1(v_c))=-v_c(1+v_c^2)/4$ on $p=q_1(v_c)$.  Hence phase trajectories  are transverse to the upper  boundary of  ${\Omega}$ and are directed into ${\Omega}$ (in view of   the symmetry (\ref{eq:symm_sig0}) we only need  consider the upper and left boundaries of ${\Omega}$).  When  $v_c \ll 1$, we have from (\ref{eq:p_sig=0}) that $dv_c/dt$ is positive when $p<0$, so that trajectories in the lower half phase-plane are directed  away from the $v_c=0$ axis. However, phase trajectories are directed toward this axis when $p>0$.   In this case $dv_c/dt \sim -v_c/4$ for $v_c \ll 1$, and hence  $v_c$ can tend to zero in infinite time, while simultaneously $p \rightarrow 0$  linearly in $v_c$ (cf. Figure \ref{fig:sig0_phase_plane}) and from (\ref{a0.5sigma0_periodic_solution_norms_real_p}) $\| \omega \|_{L^2} \rightarrow 0$. A similar argument applies for $v_c \gg 1$, in which case $v_c$ can tend to infinity in infinite time when $p<0$.  $\square$

We next show that $p$ blows up in finite time for initial data in the set ${\Omega}_{b}$ shown in Figure \ref{fig:sig0_phase_plane}.

\begin{figure}[h!]
  \centering
    \includegraphics[width=0.7 \textwidth]{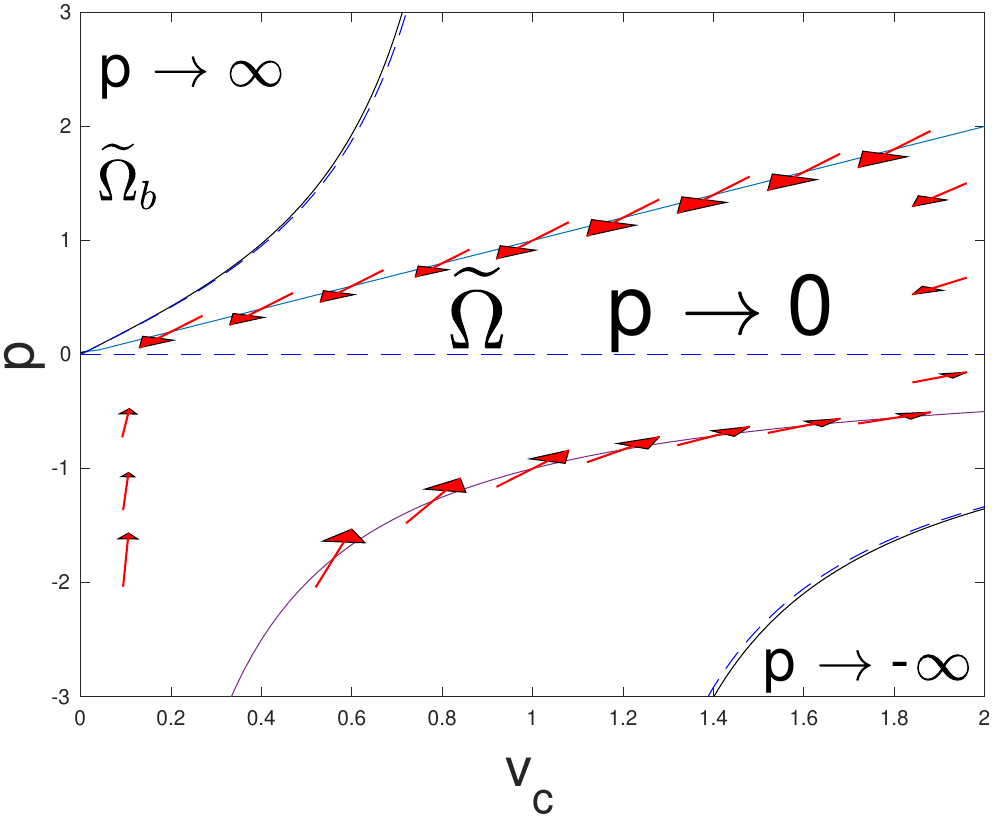}
    \caption{Phase plane  for the system (\ref{eq:p_sig=0}).  Dashed lines are $L$-nullclines.  A global existence region ${\Omega}$ is shown with the direction field ${S}$ overlaid on its boundary.  A blow-up region  ${\Omega}_b$ is bounded by  a solid curve. There is another blowup region at lower right.}
  \label{fig:sig0_phase_plane}
\end{figure}

\begin{theorem} \label{thm:blowup_sig0} {\bf Finite-time singularity formation for ${a=1/2}$, ${\sigma=0}$.}

\noindent
 Let $0< v_c(0)<1$ and
  $0<\cblue{\rho}<R(v_c(0))$ be given,  and define ${\Omega}_b$ be the region of the upper-half phase plane lying above the curve
  \be  \nonumber
  {p}_{b}(v_c)=\frac{1}{R(v_c)-\cblue{\rho}}.
  \ee
 Then for $(v_c(0),p(0)) \in {\Omega}_b$, we have $p(t) \geq p(0)/(1 - \cblue{\rho} p(0) t)$ and $v_c(t)<v_c(0)$ (i.e., ${\Omega}_b$ is invariant under (\ref{eq:p_sig=0})). Thus  $p(t) \rightarrow \infty$, $\| \omega \|_{L^2} \rightarrow \infty$, and $\| \omega \|_{B_0} \rightarrow \infty$ at  $t=t_c$ where $t_c \leq 1/(p(0) \cblue{\rho}$).

\noindent

\end{theorem}

 \begin{remark}
It follows from the symmetry (\ref{eq:symm_sig0}) that  for $1< v_c(0)<\infty$ and  $p(0)<-{p}_{b}(1/v_c(0))<0$,  we have $p(t) \leq p(0)/(1-\cblue{\rho} | p(0)| t)$ where $0<\cblue{\rho}<|R(v_c(0))|$.   For this initial data,
$p(t) \rightarrow -\infty$ at  $t=t_c$ where $t_c \leq 1/(|p(0)| \cblue{\rho}$).
\end{remark}

\noindent
 {\it Proof}.  By definition of ${p}_{b}$  we have that
 \be \label{eq:sig0_blowup_1}
 p(t) > \frac{1}{R(v_c(t))-\cblue{\rho}}
 \ee
 at $t=0$, since by assumption $p(0)>{p}_{b}(0)$. Assume (\ref{eq:sig0_blowup_1})  is satisfied up to  time $t=T$.
 Since $dR/dv_c<0$
 for $0<v_c<1$, and also  from (\ref{eq:p_sig=0}) $v_c$ is a decreasing function of $t$, it follows that $R>\cblue{\rho}$ is increasing in $t$.   Furthermore from (\ref{eq:p_sig=0}) and  (\ref{eq:sig0_blowup_1})  $dp/dt>0$ at the time $T$.  Hence (\ref{eq:sig0_blowup_1}) holds for $t>T$ and by extension for all $t$.  Rewriting (\ref{eq:sig0_blowup_1}) as $pR-1>\cblue{\rho}p$ and substituting into (\ref{eq:p_sig=0})   implies  $dp/dt > \cblue{\rho} p^2$. Integrating this inequality gives the  result on $p(t)$, with the blowup of  the norms following from (\ref{a0.5sigma0_periodic_solution_norms_real_p}).  $\square$

 Theorem \ref{thm:blowup_sig0} illustrates the difference between  blowup of $\| \omega \|_{L^2}$  for $\sigma=0$ and  $\sigma=1$. As noted earlier, finite-time singularity formation for $a=1/2, ~\sigma=0$ can occur for arbitrarily small $L^2$ norm of $\omega(x,0)$.  In the phase-plane analysis this is seen from the asymptotic behavior ${p}_b(v_c) \sim  2 v_c$ of the boundary of the blow-up region  ${\Omega}_b$
 when $v_c \ll 1$. Thus from (\ref{a0.5sigma0_periodic_solution_norms_real_p}),   $\| \omega(x,0) \|_{L^2}$ can be made arbitrarily small in ${\Omega}_b$.  In contrast, when $\sigma=1$ the solution (\ref{omegasigma1_periodic_lower_solution_1}) blows up only  for sufficiently large $L^2$ data, per the comment following the proof of Theorem \ref{thm:global_sig=1}.   The Wiener norm blows up only for sufficiently  large data in our pole dynamics solutions both for  $\sigma=0$ and $\sigma=1$.


\subsection{Numerical solution of (\ref{eq:p_sig=0})} \label{sec:AppendixANum}

\renewcommand{\figurename}{Figure}
\renewcommand\thefigure{12}

\begin{figure}[h!]
  \centering
    \includegraphics[width= 0.8 \textwidth]{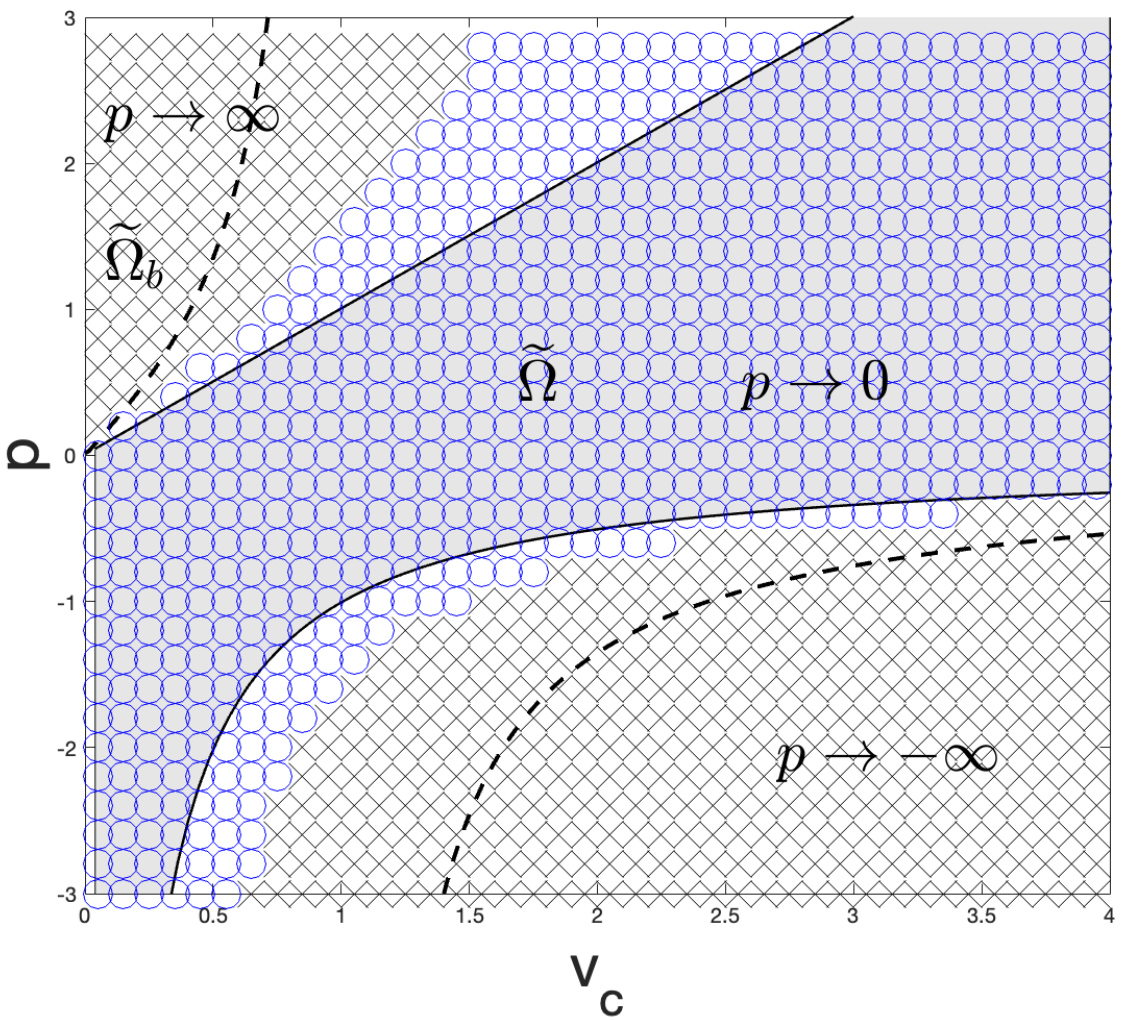}
     \caption{Comparison of numerical solution to (\ref{eq:p_sig=0}) with the phase-plane analysis. Circles and $\times$ markers  are numerically determined regions of global existence and blowup, respectively. Solid curves give the boundary of   ${\Omega}$, and dashed curves show the boundary of ${\Omega}_b$ and second blow-up region at lower right (cf. Figure \ref{fig:sig0_phase_plane}).}
  \label{fig:podesig0}
\end{figure}

\renewcommand{\figurename}{Figure}
\renewcommand\thefigure{13}

\begin{figure}[h!]
  \centering
\includegraphics[width= 0.8 \textwidth]{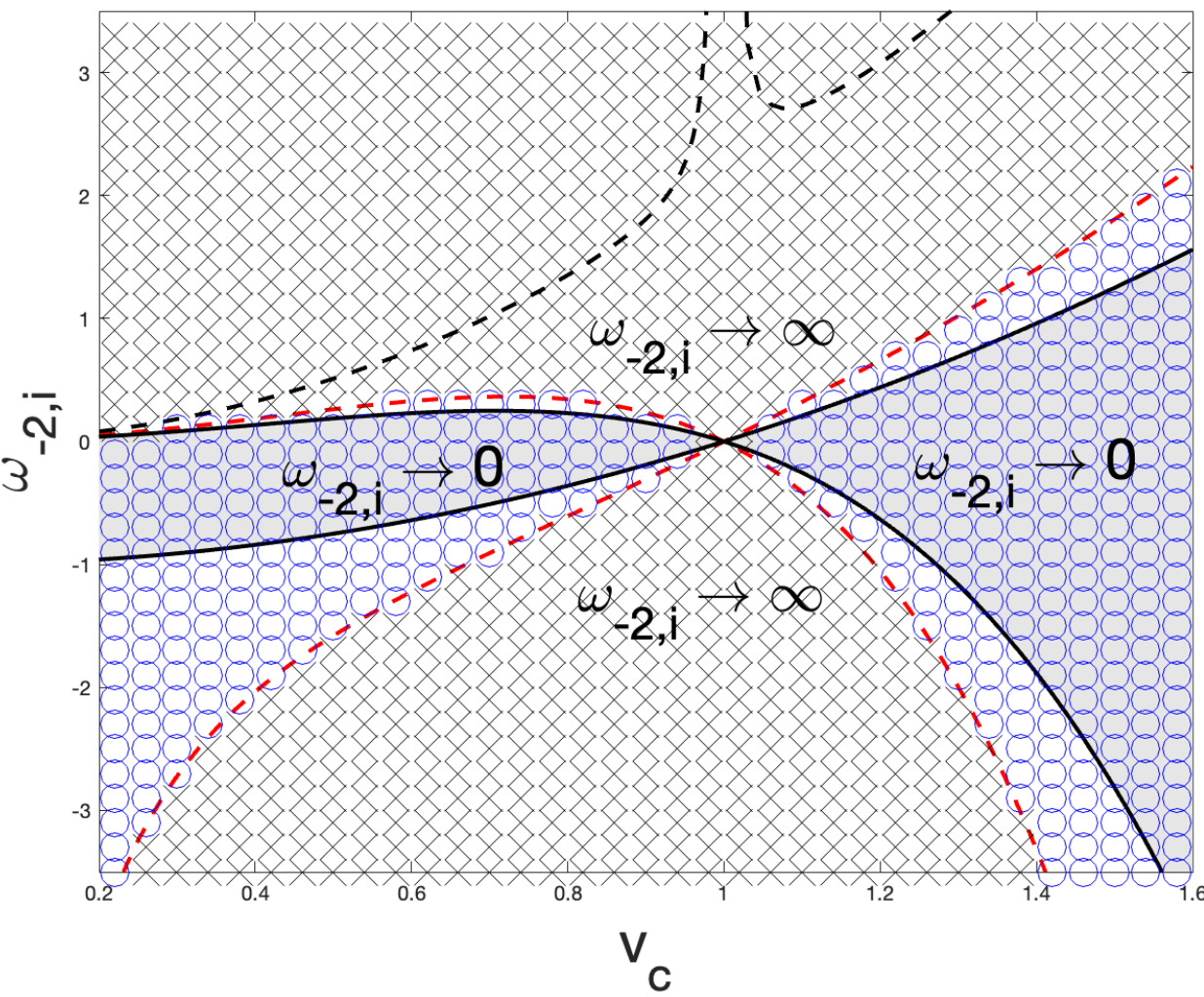}
  \caption{Numerical solution  in  the original variables $\omega_{-2,i}$ and $v_c$, plotted as in Figure \ref{fig:podesig0}.   The shaded domain and dashed-black curves demarcate the global existence region ${\Omega}$ and blowup regions, respectively,  from the phase-plane analysis. The red dashed curves indicate the boundary between the global existence and blowup regions as determined by the implicit solution from Section \ref{sec:a0.5sigma0}}
  \label{fig:coffee}
\end{figure}

Figure \ref{fig:podesig0} compares the numerical solution to (\ref{eq:p_sig=0}) with the phase-plane analysis. The blue circles  denote initial data $(v_c(0),p(0))$ for which the numerical solution appears to exist globally in time, and  are consistent with the global existence region ${\Omega}$ from analysis.   The `$\times$' markers denote initial data for which the numerical solution blows up in finite time. These are   consistent with the   analytically determined blowup regions  indicated by dashed curves.

The numerical solution and invariant regions from the phase-plane analysis are depicted in  the original variables $\omega$ and $v_c$  in Figure \ref{fig:coffee}.  These are compared to the implicit solution from Section \ref{sec:a0.5sigma0}.
The analytical results are again  consistent with the numerics.  In particular  the boundary between the global existence and blow-up regions from the implicit solution, indicated by the red-dashed curves, is in excellent agreement with the numerical results.






\end{document}